\documentclass[11pt]{amsart}

\usepackage[margin=1in]{geometry}
\usepackage{amsmath}
\usepackage{float}
\usepackage{booktabs}
\usepackage[dvipsnames]{xcolor} 

\usepackage{fix-cm}

\usepackage{ wasysym }

\usepackage{blkarray}

\usepackage{amscd,amsmath,amssymb,amsfonts,amsthm}
\usepackage{enumerate,mathrsfs,stmaryrd,latexsym, comment, mathtools, mathdots} 

\usepackage[mathscr]{euscript}

\usepackage{caption}
\usepackage{subcaption}
\usepackage{hyperref}

\usepackage{graphicx}
\usepackage{xcolor}
\usepackage{adjustbox}
\usepackage{multirow}

\usepackage[all]{xy}

\usepackage{tikz-cd}
\usepackage{tikz}
\usetikzlibrary{snakes, 
	3d, matrix,decorations.pathreplacing,calc,arrows,decorations.pathmorphing, patterns}
\usetikzlibrary {positioning}
\usetikzlibrary{calc,backgrounds}

\pgfdeclarelayer{background}
\pgfdeclarelayer{foreground}
\pgfsetlayers{background,main,foreground}

\def\aeMarkRightAngle[size=#1](#2,#3,#4){
   \draw ($(#3)!#1!(#2)$) -- 
         ($($(#3)!#1!(#2)$)!#1!90:(#2)$) --
         ($(#3)!#1!(#4)$);}

\usepackage{tikz-3dplot}
\usepackage{stackengine}
\newcommand\xrowht[2][0]{\addstackgap[.5\dimexpr#2\relax]{\vphantom{#1}}}

\numberwithin{equation}{section}
\allowdisplaybreaks[1]

\makeatletter
\newcommand{\leqnomode}{\tagsleft@true\let\veqno\@@leqno}
\newcommand{\reqnomode}{\tagsleft@false\let\veqno\@@eqno}

\makeatother

\DeclareMathOperator{\SL}{SL}

\DeclareMathOperator{\Hom}{Hom}

\DeclareMathOperator{\Sp}{Sp}
\DeclareMathOperator{\Spin}{Spin}
\DeclareMathOperator{\SO}{SO}

\DeclareMathOperator{\Aut}{Aut}

\DeclareMathOperator{\Ric}{Ric}
\DeclareMathOperator{\PSL}{PSL}

\DeclareMathOperator{\PSp}{PSp}

\newtheorem{theorem}{Theorem}[section]
\newtheorem{lemma}[theorem]{Lemma}
\newtheorem{proposition}[theorem]{Proposition}
\newtheorem{corollary}[theorem]{Corollary}

\theoremstyle{definition}
\newtheorem{example}[theorem]{Example}
\newtheorem{definition}[theorem]{Definition}
\newtheorem{remark}[theorem]{Remark}

\begin{document}
	
\title[K-stability of Gorenstein Fano group compactifications with rank two]{K-stability of Gorenstein Fano \\group compactifications with rank two}

\author{Jae-Hyouk Lee}
\address{Department of Mathematics, Ewha Womans University, 
Seodaemun-gu, Seoul 03760, Republic of Korea}
\email{jaehyoukl@ewha.ac.kr}

\author{Kyeong-Dong Park}
\address{Department of Mathematics and Research Institute of Natural Science, Gyeongsang National University, Jinju 52828, Republic of Korea}
\email{kdpark@gnu.ac.kr}

\author{Sungmin Yoo}
\address{Department of Mathematics, Incheon National University, Yeonsu-gu, Incheon 22012, Republic of Korea}
\email{sungminyoo.math@inu.ac.kr}

\subjclass[2010]{Primary: 14M27, 32Q20, Secondary: 32M12, 53C55} 

\keywords{singular K\"{a}hler--Einstein metrics, equivariant K-stability, Gorenstein Fano group compactifications, moment polytopes, greatest Ricci lower bounds, K\"{a}hler--Ricci flow}

\begin{abstract}
We give a classification of Gorenstein Fano bi-equivariant compactifications of semisimple complex Lie groups with rank two,
and determine which of them are equivariant K-stable and admit (singular) K\"{a}hler--Einstein metrics. 
As a consequence, we obtain several explicit examples of K-stable Fano varieties admitting (singular) K\"{a}hler--Einstein metrics.
We also compute the greatest Ricci lower bounds, equivalently the delta invariants for K-unstable varieties.
This gives us three new examples on which each solution of the K\"{a}hler--Ricci flow is of type II. 
\end{abstract}

\maketitle
\setcounter{tocdepth}{1} 
\date{\today}

\begin{center}
\end{center}

\section{Introduction}

The Yau--Tian--Donaldson conjecture for the Fano case states that a Fano manifold admits a K\"{a}hler--Einstein metric if and only if it satisfies the algebraic geometric condition, called the K-polystability.
This conjecture for the Fano case were completely solved by Chen--Donaldson--Sun \cite{cds1, cds2, cds3} and Tian \cite{tian15}.
Recently, Li~\cite{Li21} generalized this conjecture to $\mathbb{Q}$-Fano varieties with singular K\"{a}hler--Einstein metrics using the notion of uniform K-stability.
In general, however, it is difficult to verify the K-stability condition to show the existence of a K\"{a}hler--Einstein metric since one should consider infinite number of possible degenerations (test configurations).

Despite this issue, when the manifold has large symmetry, we can reduce the problem to checking only finite number of degenerations.
In fact, Datar--Sz\'{e}kelyhidi~\cite{DS16} proved that if a Fano manifold $X$ is {\it equivariant K-stable}, i.e., K-stable with respect to special degenerations that are $G$-equivariant for some reductive subgroup $G$ of $\Aut_0(X)$, then it admits a K\"{a}hler--Einstein metric.
Using this theorem, they could recover the theorem of Wang--Zhu which says that a toric Fano manifold admits a K\"{a}hler--Einstein metric if and only if the Futaki invariant vanishes \cite{wz}.
By Mabuchi's theorem \cite{mabuchi}, this is equivalent to that the barycenter of the moment polytope locates at the origin.

The spherical varieties consist of an important and large class of highly symmetric varieties which include toric varieties and (more generally) bi-equivariant compactifications of reductive Lie groups, called the {\it group compactifications}.
In \cite{Del17}, Delcroix proved that a smooth Fano group compactification admits a K\"{a}hler--Einstein metric if and only if the barycenter of the corresponding moment polytope translated by $-2 \rho$ locates in the relative interior of the cone generated by positive roots, where $2 \rho$ denotes the sum of all positive roots of $G$.
He also proved in \cite{Del20} that for $\mathbb{Q}$-Fano spherical varieties, this combinatorial condition is equivalent to the equivariant K-stability.

Very recently, Delcroix's theorem for smooth Fano group compactifications is generalized to $\mathbb{Q}$-Fano group compactifications by Li--Tian--Zhu~\cite{LTZ20}.
They proved that the same combinatorial criterion is also applicable to check the existence of singular K\"{a}hler--Einstein metric using a variational approach.
Combining this with the above result by Delcroix, we have the following

\begin{theorem}[Delcroix \cite{Del17, Del20}, Li--Tian--Zhu~\cite{LTZ20}]
Let $X$ be a $\mathbb{Q}$-Fano variety which is a bi-equivariant compactification of reductive complex Lie group $G$.
Then the followings are equivalent.
\begin{enumerate}
    \item[\rm (1)] $X$ admits a singular K\"{a}hler--Einstein metric.
    \item[\rm (2)] $X$ is K-stable with respect to special degenerations that are $G$-equivariant.
    \item[\rm (3)] $\textbf{bar}_{DH}(\Delta) \in 2 \rho + \Xi$, where $\textbf{bar}_{DH}(\Delta)$ is the barycenter of the moment polytope $\Delta$ with respect to the Duistermaat--Heckman measure and $\Xi$ is the relative interior of the cone generated by positive roots of $G$.
\end{enumerate}
\end{theorem}

In this paper, we study Gorenstein Fano group compactifications.
Recall that a complete normal variety $X$ is called \emph{Gorenstein Fano} if its anticanonical divisor $-K_X$ is Cartier and ample.
For toric varieties, it is well-known that classifying Gorenstein Fano toric varieties is equivalent to classifying reflexive lattice polytopes; see e.g. \cite[Theorem~8.3.4]{CLS11}. 
In particular, Gorenstein Fano toric surfaces correspond to $16$ equivalence classes of reflexive lattice polygons in the plane up to lattice equivalence (cf. \cite[Theorem~8.3.7]{CLS11}). 

By applying theorems for generalizing the results for the toric case, we classify all Gorenstein Fano bi-equivariant compactifications of semisimple complex Lie groups with rank two,
and determine the equivariant K-stability so as to conclude the admittance of (singular) K\"{a}hler--Einstein metrics.

\begin{theorem}
\label{Main theorem} 
There are exactly $60$ Gorenstein Fano ($22$ smooth, $38$ singular) varieties which are bi-equivariant compactifications of semisimple complex Lie groups with rank two.
Among them, $27$ ($16$ smooth, $11$ singular) varieties are equivariant $K$-stable and admit K\"{a}hler--Einstein metrics.
\vskip -0.5em
\begin{table}[h]
\begin{tabular}{|c|c|c|c|c|c|}\hline
\multirow{2}{*}{\rm Lie\ type} & \multirow{2}{*}{$\dim_{\mathbb{C}}X$} & \multirow{2}{*}{$G$} & {\footnotesize Gorenstein\ Fano}  & {\footnotesize $K$-stable {\rm($\exists$ KE)}}  & {\footnotesize $K$-unstable {\rm($\nexists$ KE)}}\\
\cline{4-6}
	&	& & {\tiny\rm total(smooth, singular)} & {\tiny\rm total(smooth, singular)} & {\tiny\rm total(smooth, singular)}
\\ \hline
\multirow{4}{*}{$\mathsf{A_1\times A_1}$}	&
\multirow{4}{*}{$6$}	&
$\SL_2(\mathbb C) \times \SL_2(\mathbb C)$	&	$15$ $(2,13) $  &	$3$ $(2,1) $ & $12$ $(\cdot, 12) $
\\ \cline{3-6}
		& & 	 $\PSL_2(\mathbb C) \times \PSL_2(\mathbb C)$	&	$7$ $(2,5) $  &	$5$ $(2,3) $ & $2$ $(\cdot, 2) $
\\ \cline{3-6}
		& & 	$\SL_2(\mathbb C) \times \PSL_2(\mathbb C)$	&	$14$ $(2,12) $  &	$4$ $(1,3) $ & $10$ $(\textbf{1},9) $
\\ \cline{3-6}
		& & 	$\SO_4(\mathbb C)$	&	$6$ $(3,3) $  &	$2$ $(1,1) $ & $4$ $(\textbf{2},2) $
\\ \hline
\multirow{2}{*}{$\mathsf{A_2}$}	&
\multirow{2}{*}{$8$}	&
$\SL_3(\mathbb C)$	& $5$ $(3,2) $  &	$2$ $(1,1) $ & $3$  $(\textbf{2},1) $
\\ \cline{3-6}
	& & 	$\PSL_3(\mathbb C)$	&	$3$ $(3,\cdot)$  &	$3$ $(3,\cdot) $ & $\cdot$
\\ \hline 
\multirow{2}{*}{$\mathsf{B_2}$}	&
\multirow{2}{*}{$10$}	&
$\Sp_4(\mathbb C)$	&	$4$ $(3,1) $  &	$2$ $(2,\cdot) $ & $2$ $(\textbf{1},1) $
\\ \cline{3-6}
	& & 	$\SO_5(\mathbb C)$	&	$4$ $(2,2) $  &	$4$ $(2,2) $ & $\cdot$ 
\\ \hline
$\mathsf{G_2}$	&
$14$	&
$G_2$	&	$2$ $(2,\cdot) $  &	$2$ $(2,\cdot) $ & $\cdot$
\\ \hline
\multicolumn{3}{|c|}{\rm total}		&	$60$ $(22,38) $  &	$27$  $(16,11)$ & $33$  $(\textbf{6},27)$
\\ \hline
\end{tabular}
\end{table}
\end{theorem}
\vskip -1em
We note that Delcroix \cite{Del20} already obtained the related results for the case of smooth $\SO_4(\mathbb C)$-compactifications, and Li--Tian--Zhu \cite{LTZ20} gave the full description of Gorenstein Fano $\SO_4(\mathbb C)$-compactifications including K-stabilities.
We also note that in \cite{Del15}, the case of all smooth toroidal simple group compactifications was treated.
\vskip 0.5em
One advantage of the combinatorial approach is that we can explicitly compute the value of the \emph{greatest Ricci lower bound} $R(X)$ (also called Tian's $\beta$-invariant),
which is a measurement how far the Fano manifold $X$ is from being K-stable. 
More precisely, Delcroix~\cite{Del17} proved that
for a smooth Fano group compactification without a K\"{a}hler--Einstein metric, the greatest Ricci lower bound satisfies
$$
R(X)=\sup \Big\{  t < 1 \colon \frac{1}{1-t}(2\rho-t\cdot \textbf{bar}_{DH}(\Delta)) \in  -\Xi + \Delta \Big\}.
$$
Recently, it is known that for a K-unstable Fano manifold $X$, the greatest Ricci lower bound $R(X)$ is equal to the delta invariant $\delta(X)$, 
which is defined by Fujita--Odaka~\cite{FO18} using the log canonical threshold \cite{BBJ21, CRZ19}.
Using this fact for the K-unstable smooth Fano varieties in the above theorem, we also obtain the following.

\begin{theorem}
\label{greatest Ricci lower bounds: Gorenstein Fano equivariant compactifications} 
Let $X$ be a K-unstable smooth Fano variety which is a bi-equivariant compactification of a semisimple complex Lie group $G$ with rank two.
Denote its Picard number by $\rho(X)$.
Then the greatest Ricci lower bound $R(X)$, equivalently the delta invariant $\delta(X)$ is as follows:
\vskip -0.5em
\begin{table}[h]
\begin{tabular}{|c|c|c|c|c|c|c|}\hline
 $G$ & $\SL_2\times \PSL_2$ &\multicolumn{2}{c|}{ $\SO_4(\mathbb C)$} & \multicolumn{2}{c|}{$\SL_3(\mathbb C)$}  & $\Sp_4(\mathbb C)$  
\\ \hline
  \xrowht{5pt} $\rho(X)$ & $3$ & $2$ &  $3$ & $2$  & $3$ & $3$ 
\\ \hline
  \xrowht{11pt} $R(X)$ & $\frac{8869}{9333} \approx 0.95$ 
  & $\frac{49}{51} \approx 0.96$ 
  & $ \frac{75257}{99843} \approx 0.75$ 
	& $\frac{1419621}{1493483} \approx 0.95$  
  & $\frac{21100419}{28437901} \approx 0.74$ 
  & $\frac{1046175339}{1236719713} \approx 0.84$
\\ \hline
\end{tabular}
\end{table}

\end{theorem}
\vskip -1em
The greatest Ricci lower bound  of the last case was already computed in \cite[Example~6.9]{Del17}. 
We remark that by a result of Golota \cite{Gol20}, one can also compute the delta invariants of K-unstable singular Fano varieties applying the same formula.
\vskip 0.5em
The above theorem induces interesting facts in K\"{a}hler--Ricci flow. For a Fano manifold $X$ with a K\"{a}hler form $\omega_0 \in 2\pi c_1(X)$, it is proved by Cao~\cite{Cao85} that the normalized K\"{a}hler--Ricci flow 
$$\frac{d\omega(t)}{dt} = -\text{Ric}(\omega(t)) + \omega(t), \quad \omega(0) = \omega_0,$$
has solution for all $t\geq0$.
Moreover, Perelman~\cite{Perelman} proved that if $X$ has the unique K\"{a}hler--Einstein metric $\omega_{KE}$, then $\omega(t)$ smoothly converges to $\omega_{KE}$.
However, in general, it may not have the limit.
It is proved in \cite{LTZ18} that if a smooth Fano equivariant compactification $X$ of a semisimple complex Lie group admits no K\"{a}hler--Einstein metrics, then any solution of the K\"{a}hler--Ricci flow is of type II, that is, the curvature of $\omega(t)$ is not uniformly bounded. 
As a result, Li--Tian--Zhu~\cite{LTZ18} obtain three examples having type II solutions of the K\"{a}hler--Ricci flow from smooth Fano equivariant compactifications of $\SO_4(\mathbb C)$ and $\Sp_4(\mathbb C)$. 
We get three additional such examples from the above classification.  

\begin{corollary}
The manifolds in Theorem~\ref{greatest Ricci lower bounds: Gorenstein Fano equivariant compactifications}  
have type II solutions of the K\"{a}hler--Ricci flows. 
\end{corollary}

This paper is organized as follows.
In Section~2, we recall the theory of spherical varieties by focusing on equivariant group compactifications. 
In Section~3, we discuss a combinatorial criterion for the existence of a singular K\"{a}hler--Einstein metric on a $\mathbb Q$-Fano group compactification and its greatest Ricci lower bound in terms of the barycenter of its moment polytope with respect to the Duistermaat--Heckman measure.  
In Section~4, we give a classification of Gorenstein Fano bi-equivariant compactifications of semisimple complex Lie groups with rank two, and determine which of them are K-stable and admit (singular) K\"{a}hler--Einstein metrics.  

\vskip 1em

\noindent
\textbf{Acknowledgements}. 
Jae-Hyouk Lee was supported by the National Research Foundation of Korea (NRF) grant funded by the Korea government (MSIT) (NRF-2019R1F1A1058962). 
Kyeong-Dong Park was supported by the National Research Foundation of Korea (NRF) grant funded by the Korea government (MSIT) (NRF-2019R1A2C3010487, NRF-2021R1C1C2092610). 
Sungmin Yoo was supported by the Institute for Basic Science (IBS-R032-D1).  

Kyeong-Dong Park would like to sincerely thank Michel Brion for his kind answers to several questions related to moment polytopes. 
The authors thank the anonymous referee for the careful reading and many valuable suggestions. 

\section{Spherical varieties and algebraic moment polytopes} 

Let $G$ be a connected reductive algebraic group over $\mathbb C$. 

\subsection{Spherical varieties and colors} 
Bi-equivariant compactifications of $G$ are spherical $G$-varieties. 
In the following, we review general notions and results about spherical varieties. 
We refer \cite{Knop91}, \cite{Timashev11} and \cite{Gandini18} as references for spherical varieties. 

\begin{definition}
\label{spherical variety}
A normal variety $X$ equipped with an action of $G$ is called \emph{spherical} if a Borel subgroup $B$ of $G$ acts on $X$ with an open orbit.
\end{definition}

Let $G/H$ be an open $G$-orbit of a spherical variety $X$ and $T$ be a maximal torus of $B$. 
For a character $\chi \in \mathfrak X(B) = \mathfrak X(T)$, 
let $\mathbb C(G/H)^{(B)}_{\chi} = \{ f \in \mathbb C(G/H) : b \cdot f = \chi(b) f \text{ for all } b \in B \}$ be the set of $B$-semi-invariant functions in $\mathbb C(G/H)$ associated to $\chi$, where $\mathbb C(G/H) = \mathbb C(X)$ denotes the function field of $G/H$. 

The \emph{spherical weight lattice} $\mathcal M$ of $G/H$ is defined as a subgroup of the character group $\mathfrak X(T)$ such that each element $\chi \in \mathcal M$ has the non-zero set of $B$-semi-invariant functions, that is, 
$$\mathcal M = \{ \chi \in \mathfrak X(T) : \mathbb C(G/H)^{(B)}_{\chi} \neq 0 \}.$$
Note that every function $f_{\chi}$ in $\mathbb C(G/H)^{(B)}_{\chi}$ is determined by its weight $\chi$ up to constant. 
This is because any $B$-invariant rational function on $G/H$ is constant, that is, $\mathbb C(G/H)^{B} = \mathbb C$. 
The spherical weight lattice $\mathcal M$ is a free abelian group of finite rank. 
We define the \emph{rank} of $G/H$ as the rank of the lattice $\mathcal M$. 
Let $\mathcal N = \Hom(\mathcal M, \mathbb Z)$ denote its dual lattice together with the natural pairing $\langle \, \cdot \, , \, \cdot \, \rangle \colon \mathcal N \times \mathcal M \to \mathbb Z$. 

As the open $B$-orbit of a spherical variety $X$ is an affine variety, its complement has pure codimension one and is a finite union of $B$-stable prime divisors. 

\begin{definition}
\label{color}
For a spherical homogeneous space $G/H$, 
$B$-stable prime divisors in $G/H$ are called \emph{colors} of $G/H$. 
We denote by $\mathfrak D = \{ D_1, \cdots, D_k \}$ the set of colors of $G/H$.
\end{definition}


Every discrete $\mathbb Q$-valued valuation $\nu$ of the function field $\mathbb C(G/H)$ induces a homomorphism $\rho(\nu) \colon \mathcal M \to \mathbb Q$ which is defined by $\langle \rho(\nu), \chi \rangle = \nu(f_{\chi})$, where $f_{\chi} \in \mathbb C(G/H)^{(B)}_{\chi}\backslash \{ 0 \}$. 
Hence, we get a map $$\rho \colon \{ \text{discrete $\mathbb Q$-valued valuations on $G/H$} \} \to \mathcal N \otimes \mathbb Q.$$ 
Luna and Vust \cite{LV83} showed that the restriction of $\rho$ to the set $\mathcal V$ of $G$-invariant discrete valuations on $G/H$ is injective. 
Since the map $\rho$ is injective on $\mathcal V$, we may consider $\mathcal V$ as a subset of $\mathcal N \otimes \mathbb Q$ via $\rho$. 
It is known that $\mathcal V$ is a full-dimensional (co)simplicial cone of $\mathcal N \otimes \mathbb Q$, which is called the \emph{valuation cone} of $G/H$. 
For a $B$-stable divisor $D$ in $X$, we have a discrete valuation $\nu_D$ associated to $D$, that is, $\nu_D(f)$ is the vanishing order of $f$ along $D$. 
For the sake of simplicity, we simplify the notation $\rho(\nu_D)$ as $\rho(D)$. 

\begin{remark} 
The normal equivariant embeddings of a given spherical homogeneous space are classified by combinatorial objects called \emph{colored fans}, 
which generalize the fans appearing in the classification of toric varieties. 
In a brief way, a colored fan is a finite collection of colored cones, 
which is a pair $(\mathcal C, \mathfrak F)$ consisting of $\mathfrak F \subset \mathfrak D$ and a strictly convex cone $\mathcal C \subset \mathcal N \otimes \mathbb Q$ generated by $\rho(\mathfrak F)$ and finitely many elements in the valuation cone $\mathcal V$ (see \cite{LV83} and \cite{Knop91} for details).  
\end{remark}

\subsection{Algebraic moment polytopes and anticanonical line bundles}

Let $L$ be a $G$-linearized ample line bundle on a spherical $G$-variety $X$. 
By the multiplicity-free property of spherical varieties, the algebraic moment polytope $\Delta(X, L)$ encodes the structure of representation of $G$ in the spaces of multi-sections of tensor powers of $L$.

\begin{definition}
\label{definition of moment polytope}
The \emph{algebraic moment polytope} $\Delta(X, L)$ of $L$ with respect to $B$ is defined as 
the closure of $\bigcup_{k \in \mathbb N} \Delta_k / k$ in $\mathfrak X(T) \otimes \mathbb R$, 
where $\Delta_k$ is a finite set consisting of (dominant) weights $\lambda$ such that
\begin{equation*}
H^0(X, L^{\otimes k}) = \bigoplus_{\lambda \in \Delta_k} V_G(\lambda).
\end{equation*} 
Here, $V_G(\lambda)$ means the irreducible representation of $G$ with highest weight $\lambda$. 
\end{definition}

The algebraic moment polytope $\Delta(X, L)$ for a polarized (spherical) $G$-variety $X$ was introduced by Brion in \cite{Brion87} 
as a purely algebraic version of the Kirwan polytope. 
This is indeed the convex hull of finitely many points in $\mathfrak X(T) \otimes \mathbb R$ (see \cite{Brion87}).

For a spherical $G$-variety $X$, consider the open $B$-orbit. 
As $B$ stabilizes the open $B$-orbit, the stabilizer $P$ under the $G$-action of the open $B$-orbit is a parabolic subgroup of $G$. 
Let $\rho_P$ be the half sum of roots of $P$.
Then, by \cite[Remark~4.3]{GH15} 
the $B$-weight $\xi$ of a $B$-semi-invariant global section $s \in \Gamma(X, K_X^{-1})$ of the anticanonical bundle $K_X^{-1}$ 
is equal to $2 \rho_P$. 
From \cite[Theorem~4.2]{Brion97} and \cite[Section~3.6]{Luna97}, 
we can get an explicit expression of the anticanonical divisor $- K_X$ for a spherical variety $X$. 
Furthermore, based on the works of Brion~\cite{Brion89, Brion97}, Gagliardi and Hofscheier~\cite[Section 9]{GH15} described the (algebraic) moment polytope of the anticanonical line bundle on a Gorenstein Fano spherical variety.

\begin{proposition}
\label{moment polytope of spherical variety}
Let $X$ be a Gorenstein Fano embedding of a spherical homogeneous space $G/H$. 
If a $B$-stable Weil divisor 
$-K_X = \sum_{i=1}^k m_i D_i + \sum_{j=1}^{\ell} E_j $ 
represents the anticanonical line bundle $K_X^{-1}$ for colors $D_i$ and $G$-stable divisors $E_j$ in $X$, 
the moment polytope $\Delta(X, K_X^{-1})$ is $2 \rho_P + Q_X^*$, 
where the polytope $Q_X$ is the convex hull of the set 
\begin{equation*}
\left\{ \frac{\rho(D_i)}{m_i} : i= 1, \cdots, k \right\} \cup \{ \rho(E_j) 
: j= 1, \cdots , \ell \}
\end{equation*}  
in $\mathcal N \otimes \mathbb R$ and 
its dual polytope $Q_X^*$ is defined as 
$\{ m \in \mathcal M\otimes \mathbb R : \langle n, m \rangle \geq -1 \text{ for every } n \in Q_X \}$. 
\end{proposition}

\subsection{Gorenstein Fano group compactifications} 

We refer the reader to \cite{Timashev03}, \cite{AB04}, \cite{AK05}, \cite{Timashev11} and \cite{Del17} as references for group compactifications. 
Any reductive algebraic group $G$ is isomorphic to a \emph{symmetric} homogeneous space $(G \times G) / \text{diag} (G)$ 
under the action of the group $G \times G$ for the involution $\theta(g_1, g_2)=(g_2, g_1)$, $g_1, g_2 \in G$. 
Indeed, $G$ is spherical with respect to the action of $G \times G$ by left and right multiplication from the Bruhat decomposition. 
If $T$ is a maximal torus of $G$, then $T \times T$ is a maximal torus of $G \times G$ and 
we get the spherical weight lattice 
$$\mathcal M = \mathfrak X((T \times T) / \text{diag} (T)) = \{ (\lambda, -\lambda) : \lambda \in \mathfrak X(T) \}.$$ 
Thus, $\mathcal M$ can be identified with the character lattice $\mathfrak X(T)$ of $T$ by the projection to the first coordinate. 
Under this identification, we also identify the dual lattice $\mathcal N$ of $\mathcal M$ with the dual character lattice $\mathfrak X(T)^{\vee}$. 

We now describe the colors of $G$. 
As $B^- B \subset G$ is open for the opposite Borel $B^-$ of $B$, by the Bruhat decomposition the colors of $G$ coincide with the Schubert divisors $D_{\alpha} := \overline{B^- s_{\alpha} B}$ for simple root $\alpha$, where $s_{\alpha} \in N_G(T)$ is a representative for the simple reflection associated to $\alpha$. 
Therefore, the colors are identified with the simple roots of $G$ and the image of $D_{\alpha}$ under the map $\rho$ is equal to the correponding coroot $\alpha^{\vee}$, that is, $\rho (D_{\alpha}) = \alpha^{\vee}$ (see e.g. \cite[Example~3.6]{Gandini18}). 
Recall that the \emph{coroot} $\alpha^{\vee}$ of a root $\alpha$ is defined as the unique element in $[\mathfrak g, \mathfrak g] \cap \mathfrak t$ 
such that   $\alpha(x) = \displaystyle \frac{2 \, \kappa(x, \alpha^{\vee})}{\kappa(\alpha^{\vee}, \alpha^{\vee})}$ for all $x \in \mathfrak t$, 
where $\kappa$ is the Cartan--Killing form on the Lie algebra $\mathfrak g$ of $G$ and $\mathfrak t$ denotes the Lie algebra of $T$. 

\begin{corollary}\label{gorenstein polytope}
\label{moment polytope}
Let $X$ be a Gorenstein Fano equivariant compactification of a reductive group $G$. 
If $E_j$ are $G$-stable divisors in $X$ for $1 \leq j \leq \ell$, then the moment polytope $\Delta(X, K_X^{-1})$ is a polytope of which facets lie in any Weyl wall and the affine hyperplanes 
$\{ m \in \mathcal M \otimes \mathbb R : \langle \rho(E_j), m - 2 \rho \rangle = -1 \}$ for $1 \leq j \leq \ell$, 
where $2 \rho = \displaystyle \sum_{\alpha \in \Phi^+} \alpha$ is the sum of all positive roots of $G$.
\end{corollary}

\begin{proof}
Suppose that the rank of $G$ is $r$. 
For colors $D_i$ and $G$-stable divisors $E_j$ in $X$, 
a Weil divisor $-K_X = \sum_{i=1}^r 2 D_i + \sum_{j=1}^{\ell} E_j$ represents the anticanonical line bundle $K_X^{-1}$
by \cite[Section~5.2]{AB04}. 
Then, for a color $D_{i}$ associated to a simple root $\alpha_i$, 
the image $\rho(D_{i}) = \alpha_i^{\vee}$ gives an inequality 
$$\left\langle \frac{\alpha_i^{\vee}}{2}, m - 2 \rho  \right \rangle \geq -1 
\quad \iff \quad \left\langle \alpha_i^{\vee}, m \right\rangle \geq 0 \quad \text{ for all } 1 \leq i \leq r. $$  
Thus, the images of all colors $D_1, \cdots, D_r$  determine the positive restricted Weyl chamber, and the result
follows from Proposition \ref{moment polytope of spherical variety}.
\end{proof}

Conversely, given reductive algebraic group $G$ a lattice 
polytope in $\mathcal M \otimes \mathbb R$ with one vertex as the origin and Weyl walls as facets
determine a Gorenstein Fano group compactification if the other facets lie on affine hyperplanes 
$\{ m \in \mathcal M \otimes \mathbb R : \langle \nu, m - 2 \rho \rangle = -1 \}$ 
defined by primitive elements $\nu$ in $\mathcal N$. 
In general, Gagliardi and Hofscheier~\cite{GH15} classify the Gorenstein Fano spherical embeddings of a spherical homogeneous space $G/H$ in terms of \emph{$G/H$-reflexive polytopes}. 

\begin{example}[Gorenstein Fano compactifications of $\SL_2(\mathbb C)$ and $\PSL_2(\mathbb C)$] 
\label{compactifications of SL2 and PSL2}
As $\SL_2(\mathbb C)$ is simply-connected, 
the spherical weight lattice $\mathcal M$ of $\SL_2(\mathbb C)$ is spanned by the fundamental weight $\varpi_1$ and 
its dual lattice $\mathcal N$ is spanned by the coroot $\alpha_1^{\vee}$. 
The only possible lattice moment polytope is $[0, 3 \varpi_1]$ from the inequality $\langle -\alpha_{1}^{\vee}, (x-2)\varpi_1\rangle \geq -1$.
Indeed, this is the moment polytope of the 3-dimensional hyperquadric $\mathbb Q^3$ 
which is obtained by an equivariant open embedding 
$$\SL_2(\mathbb C) \hookrightarrow X:= \{ [x, t] : \det(x) = t^2 \} \subset \mathbb P(\text{Mat}_{2 \times 2}(\mathbb C) \oplus \mathbb C), x \mapsto [x, 1].$$  
Note that we have the anticanonical divisor $-K_{\mathbb Q^3} = 2D + E$ with $\rho(E) = -\alpha_{1}^{\vee}$, where $D=\overline{B^- s_{\alpha_1} B}$ and $E=\mathbb Q^3 \backslash \SL_2(\mathbb C)$. 

On the other hand, since $\PSL_2(\mathbb C) = \SL_2(\mathbb C)/ \{ \pm I \} \cong \SO_3(\mathbb C)$ is of adjoint type, 
the spherical weight lattice $\mathcal M$ of $\PSL_2(\mathbb C)$ is spanned by the simple root $\alpha_1 = 2 \varpi_1$ and 
its dual lattice $\mathcal N$ is spanned by the coweight $\varpi_1^{\vee}$. 
In this case, the only possible lattice moment polytope is $[0, 2 \alpha_1] = [0, 4 \varpi_1]$ from the inequality $\langle -\varpi_1^{\vee}, (x-1)\alpha_1 \rangle \geq -1$. 
Indeed, this is the moment polytope of the 3-dimensional projective space $\mathbb P^3$ 
which is an example of \emph{wonderful compactifications} constructed by De~Concini and Procesi~\cite{dCP83} (see also \cite[Example~2.4]{LPY21}). 
\qed
\end{example}


\section{Singular K\"{a}hler--Einstein metrics and greatest Ricci lower bounds} 

Let $X$ be a $\mathbb{Q}$-Fano variety, i.e., a normal projective complex variety such that $-K_X$ is $\mathbb{Q}$-Cartier and ample.

\subsection{Singular K\"{a}hler--Einstein metrics} 

By Kodaira's embedding theorem, we can embed a $\mathbb{Q}$-Fano variety $X$ into a projective space $\mathbb P^N$ using a basis of $H^0(X,K_X^{-k})$ for some $k>0$.
Let $\omega_{FS}$ be the Fubini--Study metric of $\mathbb P^N$.
We choose a reference K\"{a}hler form $\omega_0$ on $X$ by
$$
\omega_0:=\frac{1}{k}\omega_{FS}|_X\in 2\pi c_1(X).
$$
Let $h$ be the normalized Ricci potential of $\omega_0$ such that $\Ric(\omega_0) - \omega_0 = \sqrt{-1} \partial \bar{\partial} h$ on $X_{reg}$, 
where $X_{reg}$ means the complement of the singular locus of $X$. 

\begin{definition}
\label{Singular KE metrics}
Let $X$ be a $\mathbb{Q}$-Fano variety with log terminal singularities. 
A \emph{singular K\"{a}hler--Einstein metric} on $X$ is a current $\omega_{\varphi}$ satisfying the complex Monge--Amp\`{e}re equation
$$\omega_{\varphi}^n = e^{h-\varphi} \omega_0^n,$$
where $\omega_{\varphi} = \omega_0 + \sqrt{-1} \partial \bar{\partial} \varphi$ and $\varphi\in \text{PSH}(X,\omega_0)$ has {\it full Monge--Amp\`{e}re mass}. 
(For precise definitions, we refer \cite{BBEGZ19}.)
\end{definition}

\begin{remark}
If $\omega_{\varphi}$ is a singular K\"{a}hler--Einstein metric,
then $\omega_{\varphi}\in C^{\infty}(X_{reg})$ and it satisfies
$$
\Ric(\omega_{\varphi}) = \omega_{\varphi} \qquad \text{ on } X_{reg}.
$$
\end{remark}

Li, Tian and Zhu~\cite[Theorem 1.2]{LTZ20} proved a criterion for the existence of singular K\"{a}hler--Einstein metrics on $\mathbb Q$-Fano group compactifications as a generalization of Delcroix's result \cite{Del17} for smooth Fano group compactifications (see also \cite{Del20}).
Let $X$ be a bi-equivariant compactification of a reductive group $G$, and
let $\Phi = \Phi(G, T)$ be the root system of $G$ with respect to a maximal torus $T$. 
Fix a set of positive roots $\Phi^+$, and denote by $2 \rho = \sum_{\alpha \in \Phi^+} \alpha$ the sum of positive roots and by $\mathcal C^+$ the cone generated by positive roots in $\Phi^+$.
Let $\Xi$ be the relative interior of $\mathcal C^+$.

\begin{theorem}[\cite{Del17, LTZ20}]
\label{KE criterion for Q-Fano group compactifications}
For a $\mathbb Q$-Fano compactification $X$ of a reductive group $G$ with the moment polytope $\Delta:=\Delta(X, K_X^{-1})$, 
$X$ admits a singular K\"{a}hler--Einstein metric if and only if 
$$
\textbf{bar}_{DH}(\Delta) \in 2 \rho + \Xi,
$$ 
where $\textbf{bar}_{DH}(\Delta)$ is the barycenter of the moment polytope with respect to the Duistermaat--Heckman measure $\displaystyle \prod_{\alpha \in \Phi^+} \kappa(\alpha, p)^2 \, dp$ and $\kappa(\cdot,\cdot)$ is the Cartan--Killing inner product.  
\end{theorem}


\subsection{Greatest Ricci lower bounds of Fano manifolds} 
Let $X$ be a Fano manifold, that is, a compact complex manifold whose anticanonical line bundle $K_X^{-1}$ is ample. 
We define the \emph{greatest Ricci lower bound} $R(X)$ of $X$ as
$$ 
R(X):=\sup \{  0 \leq t \leq 1 \colon \text{there exists a K\"{a}hler form } \omega \in c_1(X) \text{ with } \Ric(\omega) > t \, \omega \}.
$$

\begin{theorem}[Sz\'{e}kelyhidi~\cite{Sze11}]
Let $\omega\in 2\pi c_1(X)$.
$$
R(X)=\sup \{  0 \leq t \leq 1 \colon \text{there exists a smooth solution } \varphi_t \text{ of } \omega_{\varphi_t}^n = e^{h-t\varphi} \omega^n\}.
$$
\end{theorem}

In other words, $R(X)$ is the supremum of the existence of Aubin's continuity path:
$$\Ric(\omega_{\varphi_t}) = t \, \omega_{\varphi_t} + (1 - t) \omega.
$$

\begin{remark}
$R(X)$ is an invariant of the Fano manifold $X$ in the sense that this is independent of choice of $\omega\in 2\pi c_1(X)$.
This invariant was first studied by Tian~\cite{tian92}, and was explicitly defined by Rubinstein~\cite{Rubinstein08,Rubinstein09}, where it was called Tian's $\beta$-invariant.
\end{remark}

\begin{remark}
By definition, if $X$ admits a K\"{a}hler--Einstein metric, then the greatest Ricci lower bound $R(X)$ is equal to 1. 
If $X$ does not admit a K\"{a}hler--Einstein metric, 
$R(X)$ measures somehow a Fano manifold $X$ fails to be K\"{a}hler--Einstein. 
However, $R(X)=1$ does not guarantee $X$ to be K\"{a}hler--Einstein. 
For example, Tian~\cite{tian97} constructed the unstable deformation of the Mukai--Umemura 3-fold having $R(X)=1$ (see \cite[Section~3]{Sze11}). 
\end{remark}

\begin{theorem}[Li \cite{Li17}]
$R(X)=1$ if and only if $X$ is K-semistable.
\end{theorem}

The greatest Ricci lower bound $R(X)$ of a Fano manifold $X$ is closely related with Tian's $\alpha$-invariant (or the global log canonical threshold).
If $X$ does not admit a K\"{a}hler--{E}instein metric then by \cite[Theorem~2.1]{tian87} we have a lower bound of $R(X)$ in terms of the alpha invariant 
$$
R(X) \geq \alpha(X) \cdot \displaystyle \frac{\dim X + 1}{\dim X}.
$$

On the other hand, the greatest Ricci lower bound $R(X)$ is also related with the $\delta$-invariant $\delta(X, -K_X)$, defined by Fujita and Odaka~\cite{FO18} using log canonical thresholds of basis type divisors.
In fact, for a K-unstable Fano manifold $X$, the greatest Ricci lower bound $R(X)$ is equal to the delta invariant $\delta(X, -K_X)$.
More precisely, we have
\begin{theorem}[\cite{BBJ21,CRZ19}]
$R(X) = \min\{ 1, \delta(X, -K_X) \}$.
\end{theorem}

For any smooth Fano equivariant compactification of complex Lie groups, 
Delcroix \cite{Del17} obtained an explicit formula for the greatest Ricci lower bound in terms of the moment polytope as follows.
Note that this is a generalization of Li's results for toric manifolds \cite{Li11},
since the root system of this case is trivial so that the cone generated by positive roots is restricted to the origin, the Duistermaat--Heckman measure is the Lebesgue measure, and $2\rho=0$. 
\begin{theorem}[Delcroix~\cite{Del17}]\label{Greatest Ricci lower bounds}
For a smooth Fano group compactification $X$ of $G$,
$$
R(X)=\sup \left \{ t \in (0, 1) : 2 \rho + \frac{t}{1-t} (2 \rho - \textbf{bar}_{DH}(\Delta)) \in \textup{Relint}(\Delta - \mathcal C^+) \right \},
$$ 
where $\textup{Relint}(\Delta - \mathcal C^+)$ means the relative interior of the Minkovski sum of the moment polytope $\Delta$ and the negative $- \mathcal C^+$ of the cone generated by positive roots in $\Phi^+$.  
\end{theorem}

We immediately get an elementary geometric expression for $R(X)$ by Proposition~$\ref{Greatest Ricci lower bounds}$ as follows:

\begin{corollary}\label{formula for greatest Ricci lower bounds}
Let $A$ be the point corresponding to $2 \rho$ in $\mathfrak X(T)$ and $C$ be the barycenter $\textbf{bar}_{DH}(\Delta)$ of the moment polytope $\Delta(X, K_X^{-1})$ with respect to the Duistermaat--Heckman measure. 
If $Q$ is the point at which the half-line starting from the barycenter $C$ in the direction of $A$ intersects the boundary of $\Delta - \mathcal C^+$, 
then we have 
$R(X) = \displaystyle \frac{|\overrightarrow{AQ}|}{|\overrightarrow{CQ}|}.$  
\end{corollary} 

\vskip 1em 



\section{Gorenstein Fano group compactifications with rank two} 

In this section, we give a classification of Gorenstein Fano bi-equivariant compactifications of semisimple complex Lie groups with rank two,
and check their K-stability as follows: 

First, we classify all possible Gorenstein Fano bi-equivariant compactifications of each semisimple 
complex Lie group $G$ using Corollary \ref{gorenstein polytope}. 
In \cite[Section~2.4]{Ruzzi2012}, Ruzzi classified the locally factorial Fano symmetric varieties of rank two in terms of colored fans. 
As bi-equivariant group caompactifications are characterized by their moment polytopes, we can give a more explicit description of Gorenstein Fano group compactifications via moment polytopes. 
Second, for each compactification, we use the following criterion due to Alexeev--Katzarkov to check the singularity of the corresponding projective varieties.
Recall that a polytope $\Delta \subset \mathcal M\otimes\mathbb{R}$ is called \emph{Delzant} 
if the integral generators of the edges at each vertex form a basis of the lattice $\mathcal M$. 

\begin{proposition}[Proposition~2.5 of \cite{AK05}]
\label{AK}
Let $X$ be a polarized bi-equivariant compactification of a reductive group $G$ and 
$\Delta^{toric}$ be its toric polytope formed by the Weyl group action from the moment polytope $\Delta$, that is, $\Delta^{toric}= \bigcup_{w \in W} w \Delta$ for the Weyl group $W$ of $G$.
\begin{enumerate}
    \item[\rm (1)] If $X$ is smooth, then $\Delta^{toric}$ satisfies the Delzant condition.
    \item[\rm (2)] If $\Delta^{toric}$ is Delzant and no vertex of $\Delta^{toric}$ lies in a Weyl wall, then $X$ is smooth.
\end{enumerate}
\end{proposition}
Lastly, we determine which of them admit (singular) K\"{a}hler--Einstein metrics by computing the barycenter of each moment polytope with respect to the Duistermaat--Heckman measure (Theorem~\ref{KE criterion for Q-Fano group compactifications}).
In the case of K-unstable smooth Fano manifolds, we compute the greatest Ricci lower bounds applying Theorem~\ref{Greatest Ricci lower bounds} and Corollary~\ref{formula for greatest Ricci lower bounds}. 

Before we begin the above procedure, we first 
recall the classification of semisimple complex Lie groups $G$ with rank two. 
Note that a connected semisimple Lie group $G$ is determined uniquely up to an isomorphism by its Dynkin diagram and the character lattice $\mathfrak X(T)$ of a maximal torus $T \subset G$ (see e.g. \cite[Theorem~4.3.9 and Theorem~4.3.10]{OV90}). 

 \begin{proposition}
 \label{lattices of semisimple Lie groups} 
 Let $\mathsf{Q} \subset \mathsf{P}$ be the root and weight lattice for a reduced root system $\Phi$. 
 For a lattice $\mathsf{L}$ such that $\mathsf{Q} \subset \mathsf{L} \subset \mathsf{P}$, there exist a connected semisimple Lie group $G$, its maximal torus $T$ and a root system isomorphism $\Phi_G \to \Phi$ mapping $\mathfrak X(T)$ to $\mathsf{L}$. 
 \end{proposition}

 As the group $\mathsf{L}/\mathsf{Q}$ is isomorphic to the center $Z(G)$ of the semisimple Lie group $G$ corresponding to $\mathsf{L}$ by \cite[Theorem~4.3.7]{OV90}, 
 the classification of connected semisimple complex Lie groups can be given in terms of subgroups of the fundamental group $\mathsf{P}/\mathsf{Q}$ of a reduced root system $\Phi$. 
 The fundamental groups of classical complex Lie groups are well-known (see \cite[Chapter~1, \S~3]{OV90}): 

 \begin{lemma}
 \label{fundamental groups} 
The complex Lie goups $\SL_n(\mathbb C)$ and $\Sp_{2n}(\mathbb C)$ are simply-connected, but $\pi_1(\SO_n(\mathbb C)) = \mathbb Z_2$ for $n \geq 3$.
 \end{lemma}

As a result, we obtain a classification of semisimple complex Lie groups $G$ with rank two as shown in Table~\ref{table;semisimple complex Lie groups}. 
\begin{table}[h]
\begin{tabular}{|c|c|}\hline
Lie type &  Lie group $G$ \\
\hline
$\mathsf{A_1} \times \mathsf{A_1}$	& $\SL_2(\mathbb C) \times \SL_2(\mathbb C)$, $\SL_2(\mathbb C) \times \PSL_2(\mathbb C)$, $\PSL_2(\mathbb C) \times \PSL_2(\mathbb C)$, $\SO_4(\mathbb C)$
\\ \hline
 $\mathsf{A_2}$	& $\SL_3(\mathbb C)$, $\PSL_3(\mathbb C)$
\\ \hline
$\mathsf{B_2} = \mathsf{C_2}$	& $\Sp_4(\mathbb C) \cong \Spin_5(\mathbb C)$, $\SO_5(\mathbb C)$
\\ \hline 
$\mathsf{G_2}$	& the simply-connected complex Lie group $G_2$ 
\\ \hline
\end{tabular}
\caption{Classification of semisimple complex Lie groups with rank two.}
\label{table;semisimple complex Lie groups}
\end{table}

\vskip 1em

Now, we start the classification of Gorenstein Fano bi-equivariant compactifications of each semisimple complex Lie group with rank two.
Because of its complexity, we will deal with the cases of $\mathsf{A_1} \times \mathsf{A_1}$ type in the last part of this paper. 
We start with the cases of $\mathsf{A_2}$ type.

\subsection{$\mathsf{A_2}$-type} 
\subsubsection{Gorenstein Fano group compactifications of $\SL_3(\mathbb C)$} 
As $\SL_3(\mathbb C)$ is simply-connected, 
the spherical weight lattice $\mathcal M$ is spanned by the fundamental weights $\varpi_1 = (1, 0)$ and $\varpi_2 = \Big(\frac{1}{2}, \frac{\sqrt{3}}{2}\Big)$, and 
its dual lattice $\mathcal N$ is spanned by the coroots $\alpha_1^{\vee} = \Big(1, -\frac{\sqrt{3}}{3}\Big)$ and $\alpha_2^{\vee} = \Big(0, \frac{2\sqrt{3}}{3}\Big)$.
The Weyl walls are given by $W_1:=\{y=0\}$ and $W_2:=\{y=\sqrt{3}x\}$.
The sum of positive roots is $2\rho= 2 \varpi_1 + 2 \varpi_1 = (3,\sqrt{3})$.
Choosing a realization of the root system $\mathsf{A_2}$ in the Euclidean plane $\mathbb R^2$ with $\varpi_1 = (1, 0)$ and $\varpi_2 = \Big(\frac{1}{2}, \frac{\sqrt{3}}{2}\Big)$, 
for $p=(x, y)$ 
we obtain the Duistermaat--Heckman measure 
\begin{equation*}
\prod_{\alpha \in \Phi^+} \kappa(\alpha, p)^2 \, dp 
= \Big(\frac{3}{2}x - \frac{\sqrt{3}}{2}y\Big)^2 \Big(\frac{3}{2}x +\frac{\sqrt{3}}{2}y\Big)^2 (\sqrt{3} y)^2 \, dxdy
\end{equation*}
because $\mathsf{A_2}$ has 3 positive roots:
$\Phi^+ = \{ \alpha_1, \alpha_2, \alpha_1 + \alpha_2 \} = \Big\{ \Big(\frac{3}{2}, - \frac{\sqrt{3}}{2}\Big), (0, \sqrt{3}), \Big(\frac{3}{2}, \frac{\sqrt{3}}{2}\Big) \Big\}$.

\begin{theorem}
There are five Gorenstein Fano equivariant compactifications of $\SL_3(\mathbb{C})$ up to isomorphism: three smooth compactifications and two singular compactifications. 
Their moment polytopes are given in the following Table~\ref{table;SL3} and Figure~\ref{SL-polytope}. 
Among them only one smooth Fano compactification of $\SL_3(\mathbb{C})$ and a unique singular one admit singular K\"{a}hler--Einstein metrics. 
{\rm\small
\begin{table}[h]
\begin{tabular}{|c|c|l|l|l|c|c|}\hline
$X$ &  $\Delta(K_X^{-1})$ &	Case & Edges (except Weyl walls)  & Vertices &	Smooth? & 	KE \\
\hline
$\overline{\SL_3}(1)$	& $\Delta_1$&	II		&	$x+\frac{1}{\sqrt{3}}y=5$  &  $\Big(\frac{5}{2},\frac{5}{2}\sqrt{3}\Big), (5,0)$ 	&	smooth	 	& Yes  
\\ \hline
$\overline{\SL_3}(2)$	& $\Delta_2$&	I.2	&	$x+\sqrt{3}y=7,\ x+\frac{1}{\sqrt{3}}y=5$  &	$\Big(\frac{7}{4},\frac{7}{4}\sqrt{3}\Big), (4,\sqrt{3}), (5,0)$ 	&	smooth	 	& No 
\\ \hline
\multirow{2}{*}{$\overline{\SL_3}(3)$}	& \multirow{2}{*}{$\Delta_3$}&	\multirow{2}{*}{I.3.2}	&	$x+\sqrt{3}y=7,\ x+\frac{1}{\sqrt{3}}y=5$,    & $\Big(\frac{7}{4},\frac{7}{4}\sqrt{3}\Big), \Big(\frac{5}{2},\frac{3}{2}\sqrt{3}\Big), \Big(\frac{9}{2},\frac{\sqrt{3}}{2}\Big),$  	&	\multirow{2}{*}{smooth}	 	& \multirow{2}{*}{No}
\\
 & & & $2x+\frac{4}{\sqrt{3}}y=11$ & $(5,0)$ & & 
 \\ \hline 
$\overline{\SL_3}(4)$	& $\Delta_4$&	I.1		&	$x+\sqrt{3}y=7$  & $\Big(\frac{7}{4},\frac{7}{4}\sqrt{3}\Big), (7,0)$	& singular 	&	Yes 
\\ \hline
$\overline{\SL_3}(5)$	& $\Delta_5$&	I.3.1	&	$x+\sqrt{3}y=7,\ 3x+\frac{5}{\sqrt{3}}y=15$  & $\Big(\frac{7}{4},\frac{7}{4}\sqrt{3}\Big), \Big(\frac{5}{2},\frac{3}{2}\sqrt{3}\Big), (5,0)$ 	&	singular	 	& No  
\\ \hline
\end{tabular}
\caption{Gorenstein Fano equivariant compactifications of $\SL_3(\mathbb{C})$.}
\label{table;SL3}
\end{table}
}
\end{theorem}

\begin{proof} 
Using Corollary \ref{gorenstein polytope} (cf. \cite[Lemma~3.1]{LTZ20}), we construct all possible Gorenstein Fano bi-equivariant compactifications of $\SL_3(\mathbb C)$.
Let $F_i$ be the facet of the moment polytope with primitive outer normal vector 
$
m_i\alpha_1^{\vee}+n_i\alpha_2^{\vee}=\Big(m_i,\frac{-m_i+2n_i}{3}\sqrt{3} \Big)\in\mathcal N.
$
Then $F_i$ can be written as
$$
F_i=\Big\{ (x, y) \in \mathcal M \otimes \mathbb R : m_ix+\frac{-m_i+2n_i}{\sqrt{3}}y=2m_i+2n_i+1\Big\},
$$
and $0\leq\frac{-m_i+2n_i}{3m_i}\leq1$ by the convexity of the moment polytope.
We start with the edge $F_1$, the facet which intersects the Weyl wall $W_2$.
We divide this into two cases: {\bf Case-I.}  $F_1$ is orthogonal to $W_2$; {\bf Case-II.} $F_1$ is not orthogonal to $W_2$.

\begin{itemize}
\item {\bf Case-I:}  $F_1$ is orthogonal to $W_2$. \\
	In this case, $F_1\perp W_2$ implies that $\frac{-m_1+2n_1}{3m_1}=1$ so that $(m_1,n_1)=(1,2)$.
	Thus 
	$$
	F_1=\{x+\sqrt{3}y=7\},
	$$
	and $F_1\cap W_2=\Big(\frac{7}{4},\frac{7}{4}\sqrt{3}\Big)$.
Then the first vertex point $A_1:=F_1\cap F_2\in \mathcal M$ is given by 
$$
A_1=\Big(7-\frac{3}{2} \cdot \frac{5m_2-2n_2-1}{2m_2-n_2},\frac{\sqrt{3}}{2} \cdot \frac{5m_2-2n_2-1}{2m_2-n_2}\Big).
$$
The convexity and primitive conditions of the moment polytope imply that 
$$A_1=7\varpi_1\ {\rm or}\ 3\varpi_1+2\varpi_2\ {\rm or}\ \varpi_1+3\varpi_2.
$$
	\begin{itemize}
	\item {\bf Case-I.1:} $A_1=7\varpi_1=(7,0)$\\ 
	In this case, $F_1$ intersects the Weyl wall $W_1$ so that the corresponding moment polytope is
	$\Delta_4$ {\small\bf(Figure ~\ref{SL-I.1})}.
	\item {\bf Case-I.2:} $A_1=3\varpi_1+2\varpi_2=(4,\sqrt{3})$ implies that $(m_2,n_2)=(1,1)$ so that we have
	$$F_2=\Big\{x+\frac{\sqrt{3}}{3}y=5\Big\},$$ and it must intersect with $W_1$ at the vertex $A_2=F_2\cap W_1=(5,0)$. The corresponding moment polytope is
	$\Delta_2$ {\small\bf(Figure ~\ref{SL-I.2})}. 
	\item {\bf Case-I.3:} $A_1=\varpi_1+3\varpi_2=\Big(\frac{5}{2},\frac{3}{2}\sqrt{3}\Big)$ implies that $m_2+1=n_2$ so that we have
$$
F_2=\Big\{m_2x+\frac{m_2+2}{\sqrt{3}}y=4m_2+3\Big\}.
$$
By convexity and primitive condition, $A_2=5\varpi_1$ or $A_2=4\varpi_1+\varpi_2$.
	\begin{itemize}
	\item {\bf Case-I.3.1:} $A_2=5\varpi_1$ $\Rightarrow$ $(m_2,n_2)=(3,4)$ so that $F_2=\Big\{3x+\frac{5}{\sqrt{3}}y=15\Big\}$.\\
In this case, $F_2$ intersects $W_1$ at $A_2$. The moment polytope is $\Delta_5$ {\small\bf(Figure ~\ref{SL-I.3.1})}.
	\item {\bf Case-I.3.2:} $A_2=4\varpi_1+\varpi_2$ $\Rightarrow$ $(m_2,n_2)=(2,3)$ so that
$F_2=\Big\{2x+\frac{4}{\sqrt{3}}y=11\Big\}$.\\
In this case, $F_3=\Big\{x+\frac{\sqrt{3}}{3}y=5\Big\}$ intersects $W_1$ at the vertex $A_3 = 5 \varpi_1$.
The corresponding moment polytope is $\Delta_3$ {\small\bf(Figure ~\ref{SL-I.3.2})}.
	\end{itemize}
\end{itemize}
\item {\bf Case-II:} $F_1$ is not orthogonal to $W_2$. \\
In this case, $F_1\cap W_2$ is a vertex of the polytope so that $A_1:=F_1\cap W_2\in\mathcal M$. 
This implies that $A_1=\Big(\frac{2m_1+2n_1+1}{2n_1},\frac{2m_1+2n_1+1}{2n_1}\sqrt{3}\Big)$ and $0\leq\frac{-m_1+2n_1}{3m_1}<1$ by the convexity of the polytope.
Therefore, $A_1=5\varpi_2=(\frac{5}{2},\frac{5}{2}\sqrt{3})$ and $2m_1+1=3n_1$ so that
$$
F_1=\Big\{(3n_1-1)x+\frac{n_1+1}{\sqrt{3}}y=10n_1\Big\}.
$$
Then the second vertex $A_2=5\varpi_1$ or $3\varpi_1+\varpi_2$ or  $2\varpi_1+3\varpi_2$ or  $\varpi_1+4\varpi_2$.\\
Removing duplications (up to isomorphisms), one can check that we have only one case $A_2=5\varpi_1$. 
The corresponding moment polytope is $\Delta_1$ {\small\bf(Figure ~\ref{SL-II})}.
\end{itemize}

(1) Case-II. 
Let $\overline{\SL_3}(1)$ be the equivariant compactification of $\SL_3(\mathbb{C})$ whose moment polytope $\Delta_1$ is the convex hull of three points $0$, $5 \varpi_1$, $5 \varpi_2$ in $\mathcal M \otimes \mathbb R$. 
Then $\overline{\SL_3}(1)$ is a smooth projective symmetric variety with Picard number one studied by Ruzzi~\cite{Ruzzi2010}. 
Using the above formula for the Duistermaat--Heckman measure,
we get the barycenter of $\Delta_1$ with respect to the Duistermaat--Heckman measure 
\begin{align*}
\textbf{bar}_{DH}(\Delta_{1}) 
& = (\bar{x}, \bar{y}) 
= \displaystyle \frac{1}{\text{Vol}_{DH}(\Delta_1)} \left( \displaystyle \int_{\Delta_1} x \prod_{\alpha \in \Phi^+} \kappa(\alpha, p)^2 \, dp , \displaystyle \int_{\Delta_1} y \prod_{\alpha \in \Phi^+} \kappa(\alpha, p)^2 \, dp \right) \\
& = \left(\frac{10}{3}, \frac{10}{9}\sqrt{3} \right)  = \frac{10}{9} \times 2\rho. 
\end{align*}
Since $2 \rho = (3, \sqrt{3})$ and the cone $\mathcal C^+$ is generated by the vectors $\Big(\frac{3}{2}, -\frac{\sqrt{3}}{2}\Big)$ and $(0, \sqrt{3})$, 
the barycenter $\textbf{bar}_{DH}(\Delta_{1})$ is in the relative interior of the translated cone $2 \rho + \mathcal C^+$. 
Therefore, $\overline{\SL_3}(1)$ admits a K\"{a}hler--Einstein metric by Theorem~\ref{KE criterion for Q-Fano group compactifications}.

(2) Case-I.2. 
Let $\overline{\SL_3}(2)$ be the equivariant compactification of $\SL_3(\mathbb{C})$ whose moment polytope $\Delta_2$ is the convex hull of three points $0$, $5 \varpi_1$, $3 \varpi_1 + 2 \varpi_2$, $\frac{7}{2} \varpi_2$ in $\mathcal M \otimes \mathbb R$. 
Recall from \cite{Brion94} that for smooth projective spherical varieties of rank two, the blow-ups can be described and every birational equivariant morphism is a composition of blow-ups along smooth centers. 
As a result, $\overline{\SL_3}(2)$ is the blow-up of 
$\overline{\SL_3}(1)$ along the closed orbit corresponding to a vertex $A_1$ in Figure~\ref{SL-II}. 
We know that 
$$\textbf{bar}_{DH}(\Delta_{2}) = \Big (\frac{156038947}{45427872}, \frac{16309243}{19469088}\sqrt{3} \Big ) \approx (3.435, 0.838 \times \sqrt{3}).$$
Since $\textbf{bar}_{DH}(\Delta_2)$ is not in the relative interior of the translated cone $2 \rho + \mathcal C^+$, 
$\overline{\SL_3}(2)$ does not admit any K\"{a}hler--Einstein metric by Theorem~\ref{KE criterion for Q-Fano group compactifications}.
Let $Q$ be the point at which the half-line starting from the barycenter $C=\textbf{bar}_{DH}(\Delta_2)$ in the direction of $A=(3, \sqrt{3})$ intersects the boundary of $(\Delta_2 - \mathcal C^+)$. 
Considering the line $x+\sqrt{3}y=7$ giving a part of $\partial(\Delta_2 - \mathcal C^+)$ and the half-line $\overrightarrow{CA}$, 
we can compute 
$$Q = \left(-\frac{12664579}{2363584}, \frac{29209667}{7090752} \right) \approx (-5.358, 4.119).$$ 
By Corollary \ref{formula for greatest Ricci lower bounds}, the greatest Ricci lower bound of $\overline{\SL_3}(2)$ is
$$
R(\overline{\SL_3}(2)) = \displaystyle \frac{\overline{AQ}}{\overline{CQ}} = \frac{1419621}{1493483} \approx 0.9505. 
$$

(3) Case-I.3.2. 
As before, the variety $\overline{\SL_3}(3)$ is the blow-up of the smooth Fano compactification $\overline{\SL_3}(2)$ along the closed orbit corresponding to a vertex $A_1$ in Figure~\ref{SL-I.2}. 
Since the barycenter 
$$\textbf{bar}_{DH}(\Delta_{3}) = \Big (\frac{2234103775}{675213408}, \frac{527459083}{675213408}\sqrt{3} \Big ) \approx (3.309, 0.781 \times \sqrt{3})$$ 
of the moment polytope $\Delta_{3}$ with respect to the Duistermaat--Heckman measure is not in the relative interior of the translated cone $2 \rho + \mathcal C^+$, 
$\overline{\SL_3}(3)$ does not admit any K\"{a}hler--Einstein metric by Theorem~\ref{KE criterion for Q-Fano group compactifications}. 
Let $Q$ be the point at which the half-line starting from the barycenter $C=\textbf{bar}_{DH}(\Delta_3)$ in the direction of $A=(3, \sqrt{3})$ intersects the boundary of $(\Delta_3 - \mathcal C^+)$. 
Considering the line $x+\sqrt{3}y=7$ giving a part of $\partial(\Delta_3 - \mathcal C^+)$ and the half-line $\overrightarrow{CA}$, 
we can compute 
$$Q = \left(\frac{495934721}{234799424}, \frac{382553749}{234799424} \right) \approx (2.112, 1.629).$$ 
By Corollary \ref{formula for greatest Ricci lower bounds}, the greatest Ricci lower bound of $\overline{\SL_3}(3)$ is
$$
R(\overline{\SL_3}(3)) = \displaystyle \frac{\overline{AQ}}{\overline{CQ}} = \frac{21100419}{28437901} \approx 0.74198. 
$$

(4) Case-I.1. 
Let $\overline{\SL_3}(4)$ be the equivariant compactification of $\SL_3(\mathbb{C})$ whose moment polytope $\Delta_4$ is the convex hull of three points $0$, $7 \varpi_1$, $\frac{7}{2} \varpi_2$ in $\mathcal M \otimes \mathbb R$. 
As the toric polytope formed by the Weyl group action from $\Delta_4$ is not a Delzant polytope, 
$\overline{\SL_3}(4)$ is singular by Proposition~\ref{AK}. 
Since the barycenter 
$$\textbf{bar}_{DH}(\Delta_{4}) = \left(\frac{11183}{288}, \frac{203}{288} \sqrt{3} \right) \approx (4.108, 0.705 \times \sqrt{3})$$ 
of the moment polytope $\Delta_{4}$ with respect to the Duistermaat--Heckman measure is in the relative interior of the translated cone $2 \rho + \mathcal C^+$, 
$\overline{\SL_3}(4)$ admits a singular K\"{a}hler--Einstein metric by Theorem~\ref{KE criterion for Q-Fano group compactifications}.

(5) Case-I.3.1. 
As the toric polytope formed by the Weyl group action from the moment polytope $\Delta_5$ is not a Delzant polytope, 
$\overline{\SL_3}(5)$ is singular by Proposition~\ref{AK}. 
Since the barycenter 
$$\textbf{bar}_{DH}(\Delta_{5}) = \left(\frac{1580795359}{507050784}, \frac{402732299}{507050784}\sqrt{3}\right) \approx (3.118, 0.794 \times \sqrt{3})$$ 
of the moment polytope $\Delta_{5}$ with respect to the Duistermaat--Heckman measure is not in the relative interior of the translated cone $2 \rho + \mathcal C^+$, 
$\overline{\SL_3}(5)$ does not admit any singuar K\"{a}hler--Einstein metric by Theorem~\ref{KE criterion for Q-Fano group compactifications}.
\end{proof}

\vskip 1em 

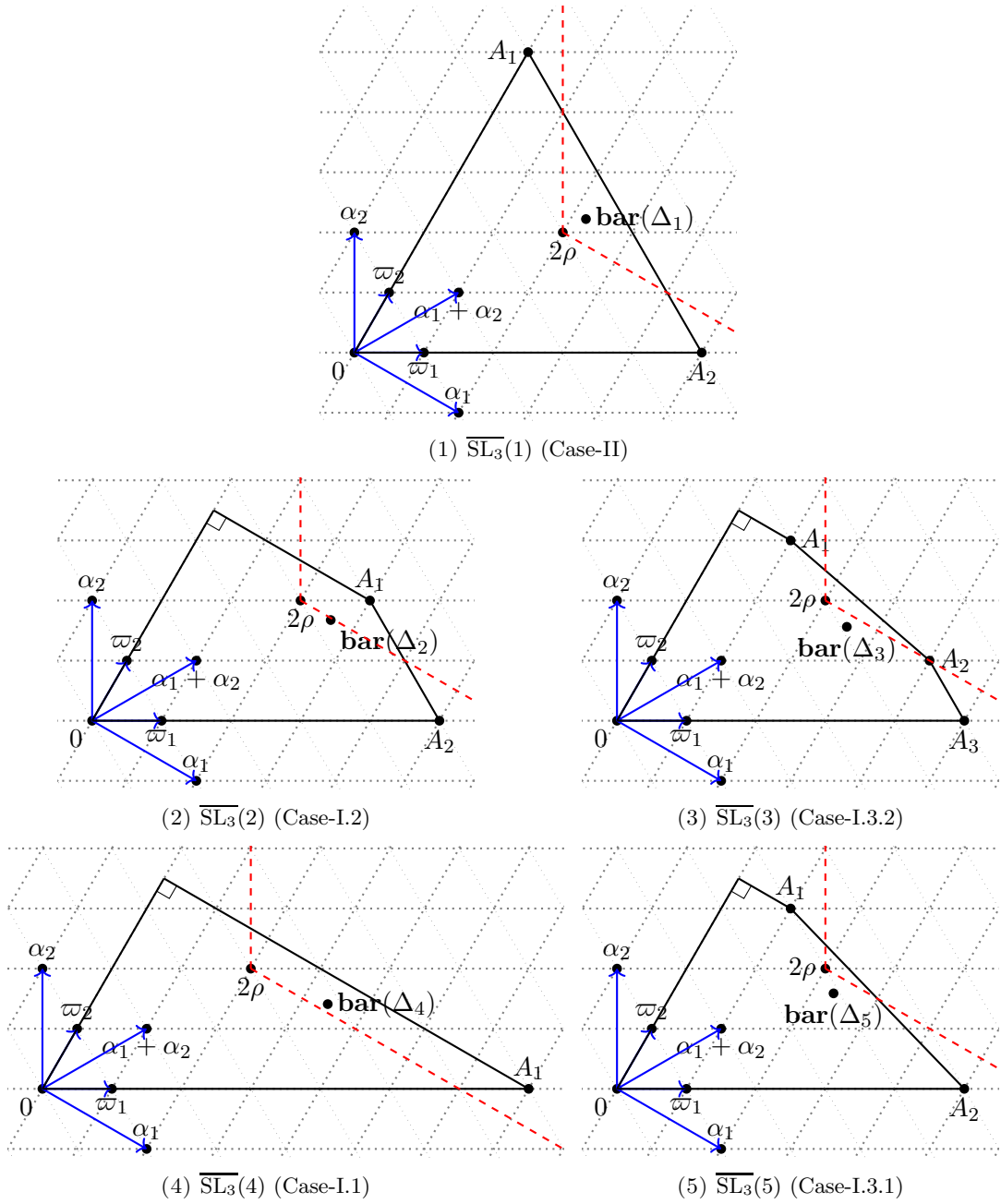
\begin{figure}[h]

\begin{minipage}[b]{.5 \textwidth}
 \centering
\begin{tikzpicture}
\clip (-0.5,-1) rectangle (5.5,5); 
\begin{scope}[y=(60:1)]

\coordinate (pi1) at (1,0);
\coordinate (pi2) at (0,1);
\coordinate (v1) at (0,5);
\coordinate (v2) at (5,0);
\coordinate (a1) at (2,-1);
\coordinate (a2) at (-1,2);
\coordinate (barycenter) at (20/9,20/9);

\coordinate (Origin) at (0,0);
\coordinate (asum) at ($(a1)+(a2)$);
\coordinate (2rho) at ($2*(asum)$);

\foreach \x  in {-8,-7,...,9}{
  \draw[help lines,thick,dotted]
    (\x,-8) -- (\x,9)
    (-8,\x) -- (9,\x);
}

\foreach \x  in {-8,-7,...,9}{
  \draw[help lines,dotted]
    [rotate=60] (\x,-8) -- (\x,9) ;
}

\fill (Origin) circle (2pt) node[below left] {0};
\fill (pi1) circle (2pt) node[below] {$\varpi_1$};
\fill (pi2) circle (2pt) node[above] {$\varpi_2$};
\fill (a1) circle (2pt) node[above] {$\alpha_1$};
\fill (a2) circle (2pt) node[above] {$\alpha_2$};
\fill (asum) circle (2pt) node[below] {$\alpha_1+\alpha_2$};
\fill (2rho) circle (2pt) node[below] {$2\rho$};

\fill (v1) circle (2pt) node[left] {$A_1$};
\fill (v2) circle (2pt) node[below] {$A_2$};

\fill (barycenter) circle (2pt) node[right] {$\textbf{bar}(\Delta_1)$};

\draw[->,blue,thick](Origin)--(pi1);
\draw[->,blue,thick](Origin)--(pi2); 
\draw[->,blue,thick](Origin)--(a1);
\draw[->,blue,thick](Origin)--(a2);
\draw[->,blue,thick](Origin)--(asum); 

\draw[thick](Origin)--(v1);
\draw[thick](v1)--(v2);
\draw[thick](Origin)--(v2);

\draw [shorten >=-4cm, red, thick, dashed] (2rho) to ($(2rho)+(a1)$);
\draw [shorten >=-4cm, red, thick, dashed] (2rho) to ($(2rho)+(a2)$);
\end{scope}
\end{tikzpicture} 
\subcaption{$\overline{\SL_3}(1)$ (Case-II)}
\label{SL-II}
\medskip
\end{minipage}

\begin{minipage}[b]{.45 \textwidth}
 \centering
\begin{tikzpicture}
\clip (-0.5,-1) rectangle (5.5,3.5); 
\begin{scope}[y=(60:1)]

\coordinate (pi1) at (1,0);
\coordinate (pi2) at (0,1);
\coordinate (v0) at (0,7/2);
\coordinate (v1) at (3,2);
\coordinate (v2) at (5,0);
\coordinate (a1) at (2,-1);
\coordinate (a2) at (-1,2);
\coordinate (barycenter) at (3.43487247212460-0.837699382734312, 0.837699382734312*2);

\coordinate (Origin) at (0,0);
\coordinate (asum) at ($(a1)+(a2)$);
\coordinate (2rho) at ($2*(asum)$);

\foreach \x  in {-8,-7,...,9}{
  \draw[help lines,thick,dotted]
    (\x,-8) -- (\x,9)
    (-8,\x) -- (9,\x);
}

\foreach \x  in {-8,-7,...,9}{
  \draw[help lines,dotted]
    [rotate=60] (\x,-8) -- (\x,9) ;
}
\fill (Origin) circle (2pt) node[below left] {0};
\fill (pi1) circle (2pt) node[below] {$\varpi_1$};
\fill (pi2) circle (2pt) node[above] {$\varpi_2$};
\fill (a1) circle (2pt) node[above] {$\alpha_1$};
\fill (a2) circle (2pt) node[above] {$\alpha_2$};
\fill (asum) circle (2pt) node[below] {$\alpha_1+\alpha_2$};
\fill (2rho) circle (2pt) node[below] {$2\rho$};

\fill (v1) circle (2pt) node[above] {$A_1$};
\fill (v2) circle (2pt) node[below] {$A_2$};

\fill (barycenter) circle (2pt) node[below right] {$\textbf{bar}(\Delta_2)$};

\draw[->,blue,thick](Origin)--(pi1);
\draw[->,blue,thick](Origin)--(pi2); 
\draw[->,blue,thick](Origin)--(a1);
\draw[->,blue,thick](Origin)--(a2);
\draw[->,blue,thick](Origin)--(asum); 

\draw[thick](Origin)--(v0);
\draw[thick](v0)--(v1);
\draw[thick](v1)--(v2);
\draw[thick](Origin)--(v2);

\aeMarkRightAngle[size=6pt](Origin,v0,v1)

\draw [shorten >=-4cm, red, thick, dashed] (2rho) to ($(2rho)+(a1)$);
\draw [shorten >=-4cm, red, thick, dashed] (2rho) to ($(2rho)+(a2)$);
\end{scope}
\end{tikzpicture} 
\subcaption{$\overline{\SL_3}(2)$ (Case-I.2)}
\label{SL-I.2}
\medskip
\end{minipage}
\begin{minipage}[b]{.45 \textwidth}
 \centering
\begin{tikzpicture}
\clip (-0.5,-1) rectangle (5.5,3.5); 
\begin{scope}[y=(60:1)]

\coordinate (pi1) at (1,0);
\coordinate (pi2) at (0,1);
\coordinate (v0) at (0,7/2);
\coordinate (v1) at (1,3);
\coordinate (v2) at (4,1);
\coordinate (v3) at (5,0);
\coordinate (a1) at (2,-1);
\coordinate (a2) at (-1,2);
\coordinate (barycenter) at (3.30873727999193-0.781173887767347, 0.781173887767347*2);

\coordinate (Origin) at (0,0);
\coordinate (asum) at ($(a1)+(a2)$);
\coordinate (2rho) at ($2*(asum)$);

\foreach \x  in {-8,-7,...,9}{
  \draw[help lines,thick,dotted]
    (\x,-8) -- (\x,9)
    (-8,\x) -- (9,\x);
}

\foreach \x  in {-8,-7,...,9}{
  \draw[help lines,dotted]
    [rotate=60] (\x,-8) -- (\x,9) ;
}

\fill (Origin) circle (2pt) node[below left] {0};
\fill (pi1) circle (2pt) node[below] {$\varpi_1$};
\fill (pi2) circle (2pt) node[above] {$\varpi_2$};
\fill (a1) circle (2pt) node[above] {$\alpha_1$};
\fill (a2) circle (2pt) node[above] {$\alpha_2$};
\fill (asum) circle (2pt) node[below] {$\alpha_1+\alpha_2$};
\fill (2rho) circle (2pt) node[left] {$2\rho$};

\fill (v1) circle (2pt) node[right] {$A_1$};
\fill (v2) circle (2pt) node[right] {$A_2$};
\fill (v3) circle (2pt) node[below] {$A_3$};

\fill (barycenter) circle (2pt) node[below] {$\textbf{bar}(\Delta_3)$};

\draw[->,blue,thick](Origin)--(pi1);
\draw[->,blue,thick](Origin)--(pi2); 
\draw[->,blue,thick](Origin)--(a1);
\draw[->,blue,thick](Origin)--(a2);
\draw[->,blue,thick](Origin)--(asum); 

\draw[thick](Origin)--(v0);
\draw[thick](v0)--(v1);
\draw[thick](v1)--(v2);
\draw[thick](v2)--(v3);
\draw[thick](Origin)--(v3);

\aeMarkRightAngle[size=6pt](Origin,v0,v1)

\draw [shorten >=-4cm, red, thick, dashed] (2rho) to ($(2rho)+(a1)$);
\draw [shorten >=-4cm, red, thick, dashed] (2rho) to ($(2rho)+(a2)$);
\end{scope}
\end{tikzpicture} 
\subcaption{$\overline{\SL_3}(3)$ (Case-I.3.2)}
\label{SL-I.3.2}
\medskip
\end{minipage}

 \begin{minipage}[b]{.45 \textwidth}
 \centering
\begin{tikzpicture}
\clip (-0.5,-1) rectangle (7.5,3.5); 
\begin{scope}[y=(60:1)]

\coordinate (pi1) at (1,0);
\coordinate (pi2) at (0,1);
\coordinate (v0) at (0,7/2);
\coordinate (v1) at (7,0);
\coordinate (a1) at (2,-1);
\coordinate (a2) at (-1,2);
\coordinate (barycenter) at (1183/288-203/288, 203/144);

\coordinate (Origin) at (0,0);
\coordinate (asum) at ($(a1)+(a2)$);
\coordinate (2rho) at ($2*(asum)$);

\foreach \x  in {-8,-7,...,9}{
  \draw[help lines,thick,dotted]
    (\x,-8) -- (\x,9)
    (-8,\x) -- (9,\x);
}

\foreach \x  in {-8,-7,...,9}{
  \draw[help lines,dotted]
    [rotate=60] (\x,-8) -- (\x,9) ;
}

\fill (Origin) circle (2pt) node[below left] {0};
\fill (pi1) circle (2pt) node[below] {$\varpi_1$};
\fill (pi2) circle (2pt) node[above] {$\varpi_2$};
\fill (a1) circle (2pt) node[above] {$\alpha_1$};
\fill (a2) circle (2pt) node[above] {$\alpha_2$};
\fill (asum) circle (2pt) node[below] {$\alpha_1+\alpha_2$};
\fill (2rho) circle (2pt) node[below] {$2\rho$};

\fill (v1) circle (2pt) node[above] {$A_1$};

\fill (barycenter) circle (2pt) node[right] {$\textbf{bar}(\Delta_4)$};

\draw[->,blue,thick](Origin)--(pi1);
\draw[->,blue,thick](Origin)--(pi2); 
\draw[->,blue,thick](Origin)--(a1);
\draw[->,blue,thick](Origin)--(a2);
\draw[->,blue,thick](Origin)--(asum); 

\draw[thick](Origin)--(v0);
\draw[thick](v0)--(v1);
\draw[thick](Origin)--(v1);

\aeMarkRightAngle[size=6pt](Origin,v0,v1)

\draw [shorten >=-4cm, red, thick, dashed] (2rho) to ($(2rho)+(a1)$);
\draw [shorten >=-4cm, red, thick, dashed] (2rho) to ($(2rho)+(a2)$);
\end{scope}
\end{tikzpicture} 
\subcaption{$\overline{\SL_3}(4)$ (Case-I.1)}
\label{SL-I.1}

\end{minipage}
\begin{minipage}[b]{.45 \textwidth}
 \centering
\begin{tikzpicture}
\clip (-0.5,-1) rectangle (5.5,3.5); 
\begin{scope}[y=(60:1)]

\coordinate (pi1) at (1,0);
\coordinate (pi2) at (0,1);
\coordinate (v0) at (0,7/2);
\coordinate (v1) at (1,3);
\coordinate (v2) at (5,0);
\coordinate (a1) at (2,-1);
\coordinate (a2) at (-1,2);
\coordinate (barycenter) at (3.11762728484412-0.794264226993090, 0.794264226993090*2);

\coordinate (Origin) at (0,0);
\coordinate (asum) at ($(a1)+(a2)$);
\coordinate (2rho) at ($2*(asum)$);

\foreach \x  in {-8,-7,...,9}{
  \draw[help lines,thick,dotted]
    (\x,-8) -- (\x,9)
    (-8,\x) -- (9,\x);
}

\foreach \x  in {-8,-7,...,9}{
  \draw[help lines,dotted]
    [rotate=60] (\x,-8) -- (\x,9) ;
}

\fill (Origin) circle (2pt) node[below left] {0};
\fill (pi1) circle (2pt) node[below] {$\varpi_1$};
\fill (pi2) circle (2pt) node[above] {$\varpi_2$};
\fill (a1) circle (2pt) node[above] {$\alpha_1$};
\fill (a2) circle (2pt) node[above] {$\alpha_2$};
\fill (asum) circle (2pt) node[below] {$\alpha_1+\alpha_2$};
\fill (2rho) circle (2pt) node[left] {$2\rho$};

\fill (v1) circle (2pt) node[above] {$A_1$};
\fill (v2) circle (2pt) node[below] {$A_2$};

\fill (barycenter) circle (2pt) node[below] {$\textbf{bar}(\Delta_5)$};

\draw[->,blue,thick](Origin)--(pi1);
\draw[->,blue,thick](Origin)--(pi2); 
\draw[->,blue,thick](Origin)--(a1);
\draw[->,blue,thick](Origin)--(a2);
\draw[->,blue,thick](Origin)--(asum); 

\draw[thick](Origin)--(v0);
\draw[thick](v0)--(v1);
\draw[thick](v1)--(v2);
\draw[thick](Origin)--(v2);

\aeMarkRightAngle[size=6pt](Origin,v0,v1)

\draw [shorten >=-4cm, red, thick, dashed] (2rho) to ($(2rho)+(a1)$);
\draw [shorten >=-4cm, red, thick, dashed] (2rho) to ($(2rho)+(a2)$);
\end{scope}
\end{tikzpicture} 
\subcaption{$\overline{\SL_3}(5)$ (Case-I.3.1)}
\label{SL-I.3.1}

\end{minipage}
\caption{Moment polytopes of Gorenstein Fano compactifications of $\SL_3(\mathbb{C})$.}
\label{SL-polytope}
\end{figure}

\clearpage

\subsubsection{Gorenstein Fano group compactification of $\PSL_3(\mathbb C)$ }

As $\PSL_3(\mathbb C)$ is of adjoint type, 
the spherical weight lattice $\mathcal M$ is spanned by the simple roots $\alpha_1 = \Big(\frac{3}{2}, -\frac{\sqrt{3}}{2}\Big)$ and $\alpha_2=(0, \sqrt{3})$, and  
its dual lattice $\mathcal N$ is spanned by the fundamental coweights $\varpi_1^{\vee}=\Big(\frac{2}{3},0\Big)$ and $\varpi_2^{\vee}=\Big(\frac{1}{3},\frac{\sqrt{3}}{3}\Big)$.
The Weyl walls are given by  $W_1=\{y=0\}$ and $W_2=\{y=\sqrt{3}x\}$. 
The sum of positive roots is $2\rho=(3,\sqrt{3})$.

\begin{theorem}
There are three Gorenstein Fano equivariant compactifications of $\PSL_3(\mathbb{C})$. 
Their moment polytopes are given in the following Table~\ref{table:PSL3} and Figure~\ref{PSL-polytope}. 
They all are smooth and admit K\"{a}hler--Einstein metrics. 
{\rm\small
\begin{table}[h]
\begin{tabular}{|c|c|l|l|l|c|c|}\hline
$X$ & $\Delta(K_X^{-1})$ &	Case & Edges (except Weyl walls)  & Vertices	&	Smoothness & 	KE \\
\hline
$\overline{\PSL_3}(1)$ & $\Delta_1$	&	I.3	&	$x+\sqrt{3}y=9$  & $\Big(\frac{9}{4},\frac{9}{4}\sqrt{3}\Big), (9,0)$ 	& smooth	&	Yes 
\\ \hline
$\overline{\PSL_3}(2)$ & $\Delta_2$	&	I.2	&	$x+\sqrt{3}y=9,\ x=\frac{9}{2}$  &	$\Big(\frac{9}{4},\frac{9}{4}\sqrt{3}\Big), \Big(\frac{9}{2},\frac{3}{2}\sqrt{3}\Big), \Big(\frac{9}{2},0\Big)$	& smooth	&	Yes  
\\ \hline
\multirow{2}{*}{$\overline{\PSL_3}(3)$} & \multirow{2}{*}{$\Delta_3$}	&	\multirow{2}{*}{I.1}	&	$x+\sqrt{3}y=9,\ x=\frac{9}{2},$   & $\Big(\frac{9}{4},\frac{9}{4}\sqrt{3}\Big), \Big(\frac{9}{2},\frac{\sqrt{3}}{2}\Big), \Big(\frac{9}{2},0\Big),$	& \multirow{2}{*}{smooth}	&	\multirow{2}{*}{Yes} 
\\ 
 & & & $3x+\sqrt{3}y=15$ & $(3,2\sqrt{3})$ & &
\\\hline
\end{tabular}
\caption{Gorenstein Fano equivariant compactifications of $\PSL_3(\mathbb{C})$.}
\label{table:PSL3}
\end{table}}
\end{theorem}

\begin{proof}
Let
$
m_i\varpi_1^{\vee}+n_i\varpi_2^{\vee}=\Big(\frac{2m_i+n_i}{3},\frac{n_i}{3}\sqrt{3}\Big)\in\mathcal N
$
be the primitive outer normal vector of the facet 
$$
F_i=\{(2m_i+n_i)x+n_i\sqrt{3}y=3+6m_i+6n_i\},
$$
where $m_i\geq0,n_i\geq0$ by the convexity of the polytope.
Let $F_1$ be the facet which intersects the Weyl wall $W_2$. 
\begin{itemize}
\item {\bf Case-I:} $F_1$ is orthogonal to $W_2$.\\
$F_1\perp W_2$
$\Rightarrow$ $(m_1,n_1)=(0,1)$
$\Rightarrow$ $F_1=\{x+\sqrt{3}y=9\}$, 
$F_1\cap W_2=\Big(\frac{9}{4},\frac{9}{4}\sqrt{3}\Big)$.\\
Then the vertex point $A_1:=F_1\cap F_2\in \mathcal M$ is given by 
$$
A_1=\Big(3+\frac{3}{2} \Big(\frac{1-n_2}{m_2}\Big), \Big\{2-\frac{1}{2} \Big(\frac{1-n_2}{m_2}\Big) \Big\}\sqrt{3}\Big).
$$
The convexity and primitive conditions imply that
 $A_1=2\alpha_1+3\alpha_2$ or $3\alpha_1+3\alpha_2$ or $4\alpha_1+3\alpha_2$ or $5\alpha_1+3\alpha_2$ or $6\alpha_1+3\alpha_2$.
But if $A_1=4\alpha_1+3\alpha_2$ or $5\alpha_1+3\alpha_2$, there is no lattice polytope.\medskip
\begin{itemize}
\item {\bf Case-I.1:} $A_1=2\alpha_1+3\alpha_2=(3,2\sqrt{3})$ $\Rightarrow$ $n_2=1$ $\Rightarrow$ 
$F_2=\{(2m_2+1)x+\sqrt{3}y=9+6m_2\}$.\\
By convexity and primitive condition, $A_2=3\alpha_1+2\alpha_2$ or $4\alpha_1+2\alpha_2$.\\
If $A_2=4\alpha_1+2\alpha_2=(6,0)$ $\Rightarrow$ $6m_2=3$ so that this is not Gorenstein.\\
If $A_2=3\alpha_1+2\alpha_2=(\frac{9}{2},\frac{\sqrt{3}}{2})$ $\Rightarrow$ 
$(m_2,n_2)=(1,1)$ so that $F_2=\{3x+\sqrt{3}y=15\}$.\\
In this case, $(m_3,n_3)=(1,0)$ so that $F_3=\{x=\frac{9}{2}\}\perp W_1$ {\small\bf($\Delta_3$,\ Figure ~\ref{PSL-I.1})}. 
\item {\bf Case-I.2:} $A_1=3\alpha_1+3\alpha_2=\Big(\frac{9}{2},\frac{3}{2}\sqrt{3}\Big)$ $\Rightarrow$ $\frac{1-n_2}{m_2}=1$ $\Rightarrow$ $(m_2,n_2)=(1,0)$ so that $F_2=\Big\{x=\frac{9}{2}\Big\}$ and $F_2\perp W_1$ {\small\bf($\Delta_2$,\ Figure ~\ref{PSL-I.2})}. 

\item {\bf Case-I.3:} $A_1=6\alpha_1+3\alpha_2=9\varpi_1$ {\small\bf($\Delta_1$,\ Figure ~\ref{PSL-I.3})}.\medskip
\end{itemize}

\item {\bf Case-II:} $F_1$ is not orthogonal to $W_2$.\\
$F_1\cap W_2=:A_1=(\frac{2m_1+2n_1+1}{m_1+2n_1} \cdot \frac{3}{2},\frac{2m_1+2n_1+1}{m_1+2n_1} \cdot \frac{3}{2}\sqrt{3})\in \mathcal M$ $\Rightarrow$
$\frac{2m_1+2n_1+1}{m_1+2n_1} = 1+\frac{m_1+1}{m_1+2n_1}$
$\Rightarrow$ $(m_1,n_1)=(1,0)$.
This is isomorphic to Case-I.3.\medskip
\end{itemize}

(1) Case-I.3. 
By \cite[Theorem 5]{Ruzzi2010}, the variety $\overline{\PSL_3}(1)$ is isomorphic to $\mathbb P(\text{Mat}_{3 \times 3}(\mathbb C)) = \mathbb P^8$ as $\PSL_3(\mathbb C) \times \PSL_3(\mathbb C)$-variety. 
This has three orbits and the closed orbit consists of rank one $3 \times 3$ matrices which is isomorphic to $\mathbb P^2 \times \mathbb P^2$.
As $\overline{\PSL_3}(1)$ is a homogeneous variety, it naturally admits a K\"{a}hler--Einstein metric (see \cite[Section~5]{Matsushima72}). 

(2) Case-I.2.  
The variety $\overline{\PSL_3}(2)$ is the blow-up of the smooth Fano compactification $\overline{\PSL_3}(1)$ along the closed orbit corresponding to a vertex $A_1$ in Figure~\ref{PSL-I.3}. 
Indeed, $\overline{\PSL_3}(2)$ is the wonderful compactification of $\PSL_3(\mathbb C)$. 
Since the barycenter 
$$\textbf{bar}_{DH}(\Delta_{2}) = \left(\frac{24641}{6592}, \frac{24641}{19776}\sqrt{3}\right) = \frac{24641}{19776} \times 2\rho \approx (3.738, 1.246 \times \sqrt{3})$$ 
of the moment polytope $\Delta_{2}$ with respect to the Duistermaat--Heckman measure is in the relative interior of the translated cone $2 \rho + \mathcal C^+$, 
$\overline{\PSL_3}(2)$ admit a K\"{a}hler--Einstein metric by Theorem~\ref{KE criterion for Q-Fano group compactifications}.

(3) Case-I.1.  
The variety $\overline{\PSL_3}(3)$ is the blow-up of 
$\overline{\PSL_3}(2)$ along the closed orbit corresponding to a vertex $A_1$ in Figure~\ref{PSL-I.2}. 
Since the barycenter 
$$\textbf{bar}_{DH}(\Delta_{3}) = \left(\frac{189565091}{57005952}, \frac{189565091}{171017856}\sqrt{3}\right) = \frac{189565091}{171017856} \times 2\rho \approx (3.325, 1.108 \times \sqrt{3})$$ 
of the moment polytope $\Delta_{3}$ with respect to the Duistermaat--Heckman measure is in the relative interior of the translated cone $2 \rho + \mathcal C^+$, 
$\overline{\PSL_3}(3)$ admit a K\"{a}hler--Einstein metric by Theorem~\ref{KE criterion for Q-Fano group compactifications}.
\end{proof}


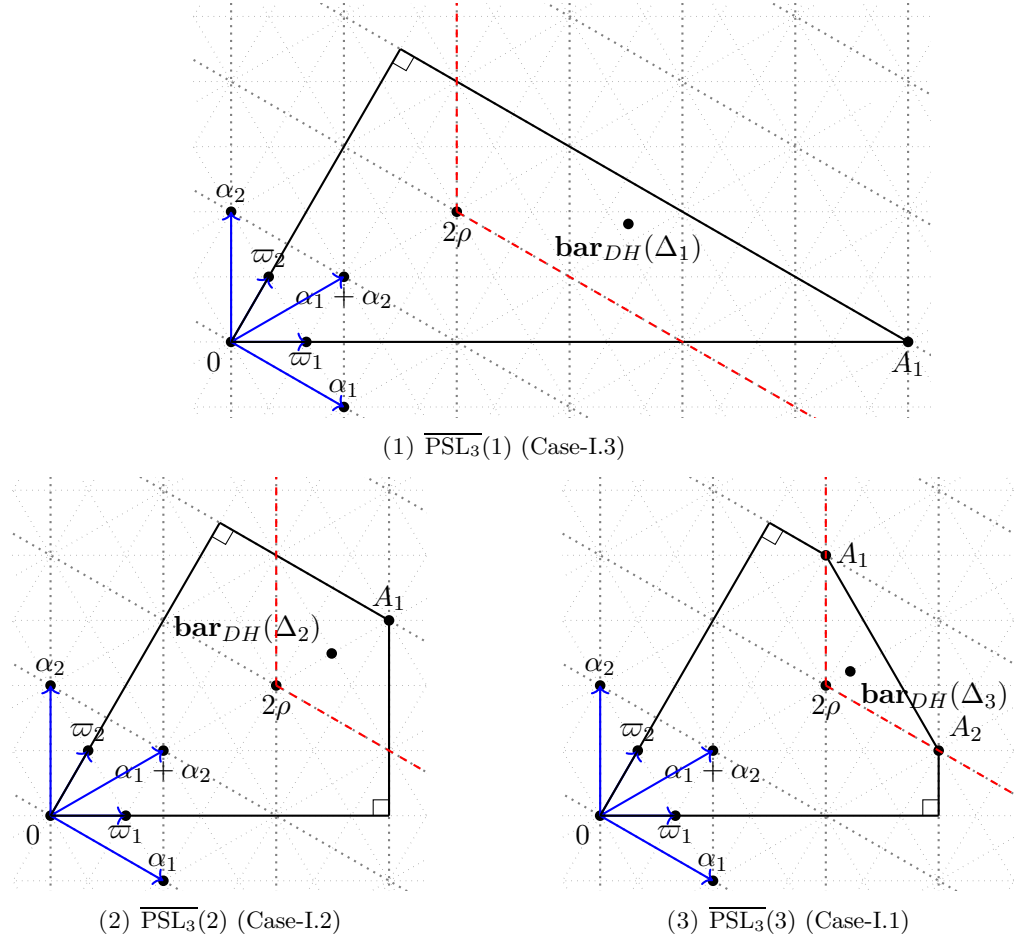
\begin{figure}[h]

\begin{minipage}[b]{.5 \textwidth}
 \centering
\begin{tikzpicture}
\clip (-0.5,-1) rectangle (9.3,4.5); 
\foreach \x  in {-9,-7.5,...,9}{
  \draw[help lines,thick,dotted]
    (\x,-9) -- (\x,9)
    [rotate=-30] (-9,\x) -- (9,\x);
}
\foreach \x  in {-9,-7.5,...,9}{
  \draw[help lines,dotted]
    [rotate=30] (-9,\x) -- (9,\x);
}

\begin{scope}[y=(60:1)]

\coordinate (pi1) at (1,0);
\coordinate (pi2) at (0,1);
\coordinate (v0) at (0,9/2);
\coordinate (v1) at (9,0);
\coordinate (a1) at (2,-1);
\coordinate (a2) at (-1,2);
\coordinate (barycenter) at (5.28125-0.90625, 0.90625*2);

\coordinate (Origin) at (0,0);
\coordinate (asum) at ($(a1)+(a2)$);
\coordinate (2rho) at ($2*(asum)$);

\foreach \x  in {-10,-9,...,12}{
  \draw[help lines,dotted]
    (\x,-10) -- (\x,12)
    (-10,\x) -- (12,\x) 
    [rotate=60] (\x,-10) -- (\x,12) ;
}

\fill (Origin) circle (2pt) node[below left] {0};
\fill (pi1) circle (2pt) node[below] {$\varpi_1$};
\fill (pi2) circle (2pt) node[above] {$\varpi_2$};
\fill (a1) circle (2pt) node[above] {$\alpha_1$};
\fill (a2) circle (2pt) node[above] {$\alpha_2$};
\fill (asum) circle (2pt) node[below] {$\alpha_1+\alpha_2$};
\fill (2rho) circle (2pt) node[below] {$2\rho$};

\fill (v1) circle (2pt) node[below] {$A_1$};

\fill (barycenter) circle (2pt) node[below] {$\textbf{bar}_{DH}(\Delta_1)$};

\draw[->,blue,thick](Origin)--(pi1);
\draw[->,blue,thick](Origin)--(pi2); 
\draw[->,blue,thick](Origin)--(a1);
\draw[->,blue,thick](Origin)--(a2);
\draw[->,blue,thick](Origin)--(asum); 

\draw[thick](Origin)--(v0);
\draw[thick](v0)--(v1);
\draw[thick](Origin)--(v1);

\aeMarkRightAngle[size=6pt](Origin,v0,v1)

\draw [shorten >=-5cm, red, thick, dashed] (2rho) to ($(2rho)+(a1)$);
\draw [shorten >=-4cm, red, thick, dashed] (2rho) to ($(2rho)+(a2)$);
\end{scope}
\end{tikzpicture} 
\subcaption{$\overline{\PSL_3}(1)$ (Case-I.3)}
\label{PSL-I.3}
\medskip
\end{minipage}

\begin{minipage}[b]{.45 \textwidth}
 \centering
\begin{tikzpicture}
\clip (-0.5,-1) rectangle (5,4.5); 
\foreach \x  in {-9,-7.5,...,9}{
  \draw[help lines,thick,dotted]
    (\x,-9) -- (\x,9)
    [rotate=-30] (-9,\x) -- (9,\x);
}
\foreach \x  in {-9,-7.5,...,9}{
  \draw[help lines,dotted]
    [rotate=30] (-9,\x) -- (9,\x);
}
\begin{scope}[y=(60:1)]

\coordinate (pi1) at (1,0);
\coordinate (pi2) at (0,1);
\coordinate (v0) at (0,9/2);
\coordinate (v1) at (3,3);
\coordinate (v2) at (9/2,0);
\coordinate (a1) at (2,-1);
\coordinate (a2) at (-1,2);
\coordinate (barycenter) at (3.73801577669903-1.24600525889968, 1.24600525889968*2);

\coordinate (Origin) at (0,0);
\coordinate (asum) at ($(a1)+(a2)$);
\coordinate (2rho) at ($2*(asum)$);

\foreach \x  in {-8,-7,...,9}{
  \draw[help lines,dotted]
    (\x,-8) -- (\x,9)
    (-8,\x) -- (9,\x) 
    [rotate=60] (\x,-8) -- (\x,9) ;
}

\fill (Origin) circle (2pt) node[below left] {0};
\fill (pi1) circle (2pt) node[below] {$\varpi_1$};
\fill (pi2) circle (2pt) node[above] {$\varpi_2$};
\fill (a1) circle (2pt) node[above] {$\alpha_1$};
\fill (a2) circle (2pt) node[above] {$\alpha_2$};
\fill (asum) circle (2pt) node[below] {$\alpha_1+\alpha_2$};
\fill (2rho) circle (2pt) node[below] {$2\rho$};

\fill (v1) circle (2pt) node[above] {$A_1$};

\fill (barycenter) circle (2pt) node[above left] {$\textbf{bar}_{DH}(\Delta_2)$};

\draw[->,blue,thick](Origin)--(pi1);
\draw[->,blue,thick](Origin)--(pi2); 
\draw[->,blue,thick](Origin)--(a1);
\draw[->,blue,thick](Origin)--(a2);
\draw[->,blue,thick](Origin)--(asum); 

\draw[thick](Origin)--(v0);
\draw[thick](v0)--(v1);
\draw[thick](v1)--(v2);
\draw[thick](Origin)--(v2);

\aeMarkRightAngle[size=6pt](Origin,v0,v1)
\aeMarkRightAngle[size=6pt](v1,v2,Origin)

\draw [shorten >=-4cm, red, thick, dashed] (2rho) to ($(2rho)+(a1)$);
\draw [shorten >=-4cm, red, thick, dashed] (2rho) to ($(2rho)+(a2)$);
\end{scope}
\end{tikzpicture} 
\subcaption{$\overline{\PSL_3}(2)$ (Case-I.2)}
\label{PSL-I.2}

\end{minipage}
\begin{minipage}[b]{.45 \textwidth}
 \centering
\begin{tikzpicture}
\clip (-0.5,-1) rectangle (5.5,4.5); 
\foreach \x  in {-9,-7.5,...,9}{
  \draw[help lines,thick,dotted]
    (\x,-9) -- (\x,9)
    [rotate=-30] (-9,\x) -- (9,\x);
}
\foreach \x  in {-9,-7.5,...,9}{
  \draw[help lines,dotted]
    [rotate=30] (-9,\x) -- (9,\x);
}
\begin{scope}[y=(60:1)]

\coordinate (pi1) at (1,0);
\coordinate (pi2) at (0,1);
\coordinate (v0) at (0,9/2);
\coordinate (v1) at (1,4);
\coordinate (v2) at (4,1);
\coordinate (v3) at (9/2,0);
\coordinate (a1) at (2,-1);
\coordinate (a2) at (-1,2);
\coordinate (barycenter) at (3.32535611369143-1.10845203789714, 1.10845203789714*2);

\coordinate (Origin) at (0,0);
\coordinate (asum) at ($(a1)+(a2)$);
\coordinate (2rho) at ($2*(asum)$);

\foreach \x  in {-8,-7,...,9}{
  \draw[help lines,dotted]
    (\x,-8) -- (\x,9)
    (-8,\x) -- (9,\x) 
    [rotate=60] (\x,-8) -- (\x,9) ;
}

\fill (Origin) circle (2pt) node[below left] {0};
\fill (pi1) circle (2pt) node[below] {$\varpi_1$};
\fill (pi2) circle (2pt) node[above] {$\varpi_2$};
\fill (a1) circle (2pt) node[above] {$\alpha_1$};
\fill (a2) circle (2pt) node[above] {$\alpha_2$};
\fill (asum) circle (2pt) node[below] {$\alpha_1+\alpha_2$};
\fill (2rho) circle (2pt) node[below] {$2\rho$};

\fill (v1) circle (2pt) node[right] {$A_1$};
\fill (v2) circle (2pt) node[above right] {$A_2$};

\fill (barycenter) circle (2pt) node[below right] {$\textbf{bar}_{DH}(\Delta_3)$};

\draw[->,thick,blue](Origin)--(pi1);
\draw[->,thick,blue](Origin)--(pi2); 
\draw[->,thick,blue](Origin)--(a1);
\draw[->,thick,blue](Origin)--(a2);
\draw[->,thick,blue](Origin)--(asum); 

\draw[thick](Origin)--(v0);
\draw[thick](v0)--(v1);
\draw[thick](v1)--(v2);
\draw[thick](v2)--(v3);
\draw[thick](Origin)--(v3);

\aeMarkRightAngle[size=6pt](Origin,v0,v1)
\aeMarkRightAngle[size=6pt](v2,v3,Origin)

\draw [shorten >=-4cm, red, thick, dashed] (2rho) to ($(2rho)+(a1)$);
\draw [shorten >=-4cm, red, thick, dashed] (2rho) to ($(2rho)+(a2)$);
\end{scope}
\end{tikzpicture} 
\subcaption{$\overline{\PSL_3}(3)$ (Case-I.1)}
\label{PSL-I.1}

\end{minipage}
\caption{Moment polytopes of Gorenstein Fano compactifications of $\PSL_3(\mathbb{C})$.}
\label{PSL-polytope}
\end{figure}

\clearpage
\subsection{$\mathsf{B_2}$-type}
\subsubsection{Gorenstein Fano group compactifications of $\Sp_4(\mathbb C) \cong \Spin_5(\mathbb C)$}

As $\Sp_4(\mathbb C)$ is simply-connected, 
the spherical weight lattice $\mathcal M$ is spanned by $\varpi_1 = \Big(\frac{1}{2}, \frac{1}{2}\Big)$ and $\varpi_2 = (0, 1)$, and 
its dual lattice $\mathcal N$ is spanned by $\alpha_1^{\vee} = (2, 0)$ and $\alpha_2^{\vee} = (-1, 1)$.
The Weyl walls are given by $W_1=\{y=x\}$ and $W_2=\{x=0\}$.
The sum of positive roots is $2\rho=(1,3)$.
Choosing a realization of the root system $\mathsf{C_2}$ in the Euclidean plane $\mathbb R^2$ with $\alpha_1 = (1, 0)$ and $\alpha_2 = (-1, 1)$, 
for $p=(x, y)$ 
we obtain the Duistermaat--Heckman measure 
\begin{equation*}
\prod_{\alpha \in \Phi^+} \kappa(\alpha, p)^2 \, dp = x^2 (-x+y)^2 y^2 (x+y)^2 \, dxdy 
\end{equation*}
because $\mathsf{C_2}$ has 4 positive roots:
$\Phi^+ = \{ \alpha_1, \alpha_2, \alpha_1 + \alpha_2, 2\alpha_1 + \alpha_2 \} = \{ (1, 0), (-1, 1), (0, 1), (1, 1) \}$.

\begin{theorem}
There are four Gorenstein Fano equivariant compactifications of $\Sp_4(\mathbb{C})$ up to isomorphism: three smooth compactifications and one singular compactification. 
Their moment polytopes are given in the following Table~\ref{table;Sp4} and Figure~\ref{Sp-polytope}. 
Among them only two smooth Fano compactifications admit K\"{a}hler--Einstein metrics. 
\vskip -0.5em
{\rm\small
\begin{table}[h]
{\renewcommand{\arraystretch}{1.1}
\begin{tabular}{|c|c|l|l|l|c|c|}\hline
$X$ & $\Delta(K_X^{-1})$ & 	Case & Edges (except Weyl walls)  & Vertices &	Smoothness & 	KE \\
\hline
$\overline{\Sp_4}(1)$  & $\Delta_1$	&	I.3	&	$y=\frac{7}{2}$  & $\Big(0,\frac{7}{2}\Big), \Big(\frac{7}{2},\frac{7}{2}\Big)$  	&	smooth	 	& Yes  
\\ \hline
$\overline{\Sp_4}(2)$  & $\Delta_2$	&	I.2	&	$y=\frac{7}{2},\ x+y=5$  & $\Big(0,\frac{7}{2}\Big), \Big(\frac{3}{2},\frac{7}{2}\Big), \Big(\frac{5}{2},\frac{5}{2}\Big)$  	&	smooth	 	& Yes  
\\ \hline
\multirow{2}{*}{$\overline{\Sp_4}(3)$}  & \multirow{2}{*}{$\Delta_3$}	&	\multirow{2}{*}{I.1.2} &	$y=\frac{7}{2},\ x+y=5,$  & \multirow{2}{*}{$\Big(0,\frac{7}{2}\Big), \Big(\frac{1}{2},\frac{7}{2}\Big),(2,3), \Big(\frac{5}{2},\frac{5}{2}\Big)$}	&	\multirow{2}{*}{smooth}	 	& \multirow{2}{*}{No} 
\\ 
 & & & $x+3y=11$ & & &
\\ \hline
$\overline{\Sp_4}(4)$  & $\Delta_4$	&	I.1.1 &	$y=\frac{7}{2},\ 2x+4y=15$  & $\Big(0,\frac{7}{2}\Big), \Big(\frac{1}{2},\frac{7}{2}\Big),\Big(\frac{5}{2},\frac{5}{2}\Big)$	&	singular	 	& No 
\\ \hline
\end{tabular}}
\caption{Gorenstein Fano equivariant compactifications of $\Sp_4(\mathbb{C})$.}
\label{table;Sp4}
\vskip -1em
\end{table}}
\end{theorem}

\begin{proof}
Let $m_i\alpha_1^{\vee}+n_i\alpha_2^{\vee}=(2m_i-n_i,n_i)\in\mathcal N$ be the primitive outer normal vector of the facet
$$
F_i=\{(2m_i-n_i)x+n_iy=2m_i+2n_i+1\},
$$
where $n_i\geq2m_i-n_i\geq0$ by the convexity of the polytope.

Let $F_1$ be the facet which intersects the Weyl wall $W_2$, then we have two cases:

\begin{itemize}
\item {\bf Case-I:} $F_1$ is orthogonal to $W_2$.\\
$F_1\perp W_2$ $\Rightarrow$ $(m_1,n_1)=(1,2)$ $\Rightarrow$ $F_1=\Big\{y=\frac{7}{2}\Big\}$, $F_1\cap W_2=\Big(0,\frac{7}{2}\Big)$.
Let $A_1:=F_1\cap F_2$.\\
$A_1=\Big(\frac{4m_2-3n_2+2}{4m_2-2n_2},\frac{7}{2}\Big)\in \mathcal M$ $\Rightarrow$ $A_1=\varpi_1+3\varpi_2$ or $3\varpi_1+2\varpi_2$ or $5\varpi_1+\varpi_2$ or $7\varpi_1$.\\
But if $A_1=5\varpi_1+1\varpi_2=\Big(\frac{5}{2},\frac{7}{2}\Big)$, there is no lattice polytope.
\begin{itemize}
\item {\bf Case-I.1:} $A_1=\varpi_1+3\varpi_2=\Big(\frac{1}{2},\frac{7}{2}\Big)$ 
$\Rightarrow$ $m_2+1=n_2$,
$F_2=\{(m_2-1)x+(m_2+1)y=4m_2+3\}$.
By convexity and primitive condition, $A_2=5\varpi_1$ or $4\varpi_1+\varpi_2$.
\begin{itemize}
\item {\bf Case-I.1.1:} $A_2=5\varpi_1$ $\Rightarrow$ $(m_2,n_2)=(3,4)$, $F_2=\{2x+4y=15\}$\\ {\small\bf($\Delta_4$,\  Figure ~\ref{Sp-I.1.1})}.
\item {\bf Case-I.1.2:} $A_2=4\varpi_1+\varpi_2$ $\Rightarrow$ $(m_2,n_2)=(2,3)$, $F_2=\{x+3y=11\}$\\ 
In this case, $(m_3,n_3)=(1,1)$ so that $F_3=\{x+y=5\}\perp W_1$
{\small\bf($\Delta_3$,\ Figure ~\ref{Sp-I.1.2})}.
\end{itemize}
\item {\bf Case-I.2:} $A_1=3\varpi_1+2\varpi_2=\Big(\frac{3}{2},\frac{7}{2}\Big)$ $\Rightarrow$ 
$(m_2,n_2)=(1,1)$, $F_2=\{x+y=5\}$\\ 
{\small\bf($\Delta_2$,\ Figure ~\ref{Sp-I.2})}.
\item{\bf Case-I.3:} $A_1=7\varpi_1=(\frac{7}{2},\frac{7}{2})$  {\small\bf($\Delta_1$,\ Figure ~\ref{Sp-I.3})}.
\end{itemize}
\item {\bf Case-II:} $F_1$ is not orthogonal to $W_2$.\\
$F_1\cap W_2=:A_1=(0,\frac{2m_1+2n_1+1}{n_1})\in \mathcal M$ 
$\Rightarrow$ $3+\frac{1}{n_1}\leq\frac{2m_1+2n_1+1}{n_1}<4+\frac{1}{n_1}$ $\Rightarrow$ $\frac{2m_1+2n_1+1}{n_1}=4$\\
In this case, there is no primitive vector.
\end{itemize}

(1) Case-I.3.
The variety $\overline{\Sp_4}(1)$ is isomorphic to the 10-dimensional Lagrangian Grassmannian $\text{Lag}(4, 8)$ parametrizing 4-dimensional isotropic subspaces in the 8-dimensional symplectic vector space (see \cite[Theorem~5]{Ruzzi2010}). 
Since $\overline{\Sp_4}(1)$ is homogeneous, it admits a K\"{a}hler--Einstein metric. 

(2) Case-I.2.
The variety $\overline{\Sp_4}(2)$ is the blow-up of the smooth Fano compactification $\overline{\Sp_4}(1)$ along the closed orbit corresponding to a vertex $A_1$ in Figure~\ref{Sp-I.3}. 
Indeed, $\overline{\Sp_4}(2)$ is the wonderful compactification of $\Sp_4(\mathbb C)$. 
Since $2 \rho = (1, 3)$ and the cone $\mathcal C^+$ is generated by the vectors $(1, 0)$ and $(-1, 1)$, 
the barycenter 
$$
\textbf{bar}_{DH}(\Delta_{2}) = {\scriptstyle\left(\frac{135148980025}{104829824704}, \frac{5019760035}{1637966011}\right)} \approx (1.289, 3.065)
$$
of the moment polytope $\Delta_{2}$ with respect to the Duistermaat--Heckman measure is in the relative interior of the translated cone $2 \rho + \mathcal C^+$. 
By Theorem~\ref{KE criterion for Q-Fano group compactifications}, $\overline{\Sp_4}(2)$ admits a K\"{a}hler--Einstein metric. 

(3) Case-I.1.2.
The variety $\overline{\Sp_4}(3)$ is the blow-up of the wonderful compactification $\overline{\Sp_4}(2)$ along the unique closed orbit corresponding to a vertex $A_1$ in Figure~\ref{Sp-I.2}. 
Since the barycenter 
$$
\textbf{bar}_{DH}(\Delta_{3}) = {\scriptstyle\left(\frac{27756440595}{22318407232}, \frac{3043253830}{1046175339}\right)} \approx (1.244, 2.909)
$$
of the moment polytope $\Delta_{3}$ with respect to the Duistermaat--Heckman measure is not in the relative interior of the translated cone $2 \rho + \mathcal C^+$, 
$\overline{\Sp_4}(3)$ does not admit any K\"{a}hler--Einstein metric as already showed in \cite[Example~5.4]{Del17}. 

(4) Case-I.1.1.
By Proposition~\ref{AK}, $\overline{\Sp_4}(4)$ is singular. 
Since the barycenter 
$$
\textbf{bar}_{DH}(\Delta_{4}) = {\scriptstyle\left(\frac{53741124025}{47717371328}, \frac{192699595}{67780357}\right)} \approx (1.262, 2.843)
$$
of the moment polytope $\Delta_{4}$ with respect to the Duistermaat--Heckman measure is not in the relative interior of the translated cone $2 \rho + \mathcal C^+$, 
$\overline{\Sp_4}(4)$ does not admit any K\"{a}hler--Einstein metric.
\end{proof}

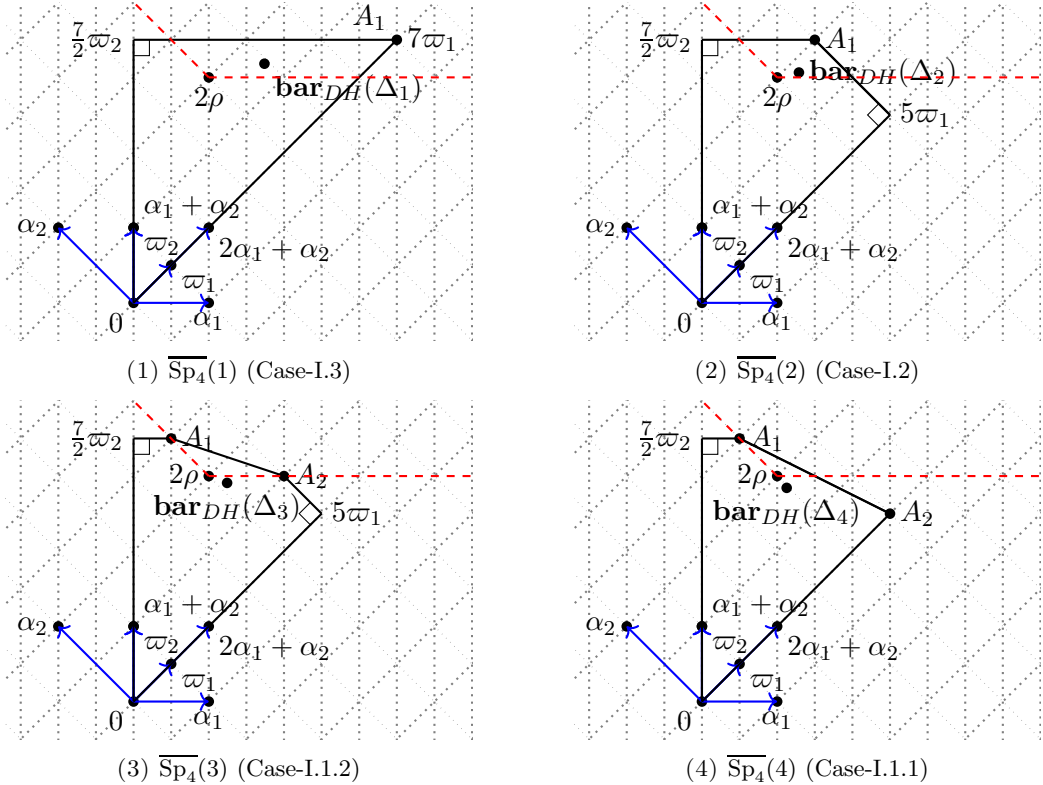
\begin{figure}[h]

\begin{minipage}[b]{.45 \textwidth}
 \centering

\begin{tikzpicture}
\clip (-1.7,-0.5) rectangle (4.5,4); 

\coordinate (pi1) at (0.5,0.5);
\coordinate (pi2) at (0,1);
\coordinate (v0) at ($3.5*(pi2)$);
\coordinate (v2) at ($7*(pi1)$);
\coordinate (a1) at (1,0);
\coordinate (a2) at (-1,1);
\coordinate (a3) at ($2*(a1)+(a2)$);
\coordinate (barycenter) at (512/273-32/9,64/9);

\coordinate (Origin) at (0,0);
\coordinate (asum) at ($(a1)+(a2)$);
\coordinate (2rho) at (1,3);

\coordinate (barycenter) at (1.74005681818182, 3.18181818181818);

\foreach \x  in {-8,-7.5,...,9}{
  \draw[help lines,thick,dotted]
    (\x,-8) -- (\x,9);
}


\begin{scope}[y=(45:1)]
\foreach \y  in {-8,-7,...,9}{
  \draw[help lines,thick,dotted]
   [rotate=45]  (-8,\y) -- (9,\y) ;
}

\foreach \y  in {-8,-7,...,9}{
  \draw[help lines,dotted]
   [rotate=-45]  (-8,\y) -- (9,\y) ;
}
\end{scope}

\fill (Origin) circle (2pt) node[below left] {0};
\fill (pi1) circle (2pt) node[below right] {$\varpi_1$};
\fill (pi2) circle (2pt) node[below right] {$\varpi_2$};
\fill (a1) circle (2pt) node[below] {$\alpha_1$};
\fill (a2) circle (2pt) node[left] {$\alpha_2$};
\fill (a3) circle (2pt) node[below right] {$2\alpha_1+\alpha_2$};

\fill (asum) circle (2pt) node[above right] {$\alpha_1+\alpha_2$};
\fill (2rho) circle (2pt) node[below] {$2\rho$};

\fill (v0) node[left] {$\frac{7}{2}\varpi_2$};
\fill (v2) circle (2pt) node[above left] {$A_1$};
\fill (v2) node[right] {$7\varpi_1$};

\fill (barycenter) circle (2pt) node[below right] {$\textbf{bar}_{DH}(\Delta_{1})$};

\draw[->,blue,thick](Origin)--(pi1);
\draw[->,blue,thick](Origin)--(pi2); 
\draw[->,blue,thick](Origin)--(a1);
\draw[->,blue,thick](Origin)--(a2);
\draw[->,blue,thick](Origin)--(a3); 
\draw[->,blue,thick](Origin)--(asum); 

\draw[thick](Origin)--(v0);
\draw[thick](v0)--(v2);
\draw[thick](Origin)--(v2);

\aeMarkRightAngle[size=6pt](Origin,v0,v2)

\draw [shorten >=-4cm, red, thick, dashed] (2rho) to ($(2rho)+(a1)$);
\draw [shorten >=-4cm, red, thick, dashed] (2rho) to ($(2rho)+(a2)$);
\end{tikzpicture} 
\subcaption{$\overline{\Sp_4}(1)$ (Case-I.3)}
\label{Sp-I.3}
\medskip
\end{minipage}
\begin{minipage}[b]{.45 \textwidth}
\centering

\begin{tikzpicture}
\clip (-1.7,-0.5) rectangle (4.5,4); 

\coordinate (pi1) at (0.5,0.5);
\coordinate (pi2) at (0,1);
\coordinate (v0) at ($3.5*(pi2)$);
\coordinate (v1) at ($3*(pi1)+2*(pi2)$);
\coordinate (v2) at ($5*(pi1)$);
\coordinate (a1) at (1,0);
\coordinate (a2) at (-1,1);
\coordinate (a3) at ($2*(a1)+(a2)$);
\coordinate (barycenter) at (512/273-32/9,64/9);

\coordinate (Origin) at (0,0);
\coordinate (asum) at ($(a1)+(a2)$);
\coordinate (2rho) at (1,3);

\coordinate (barycenter) at (1.28922260822824, 3.06463015794532);

\foreach \x  in {-8,-7.5,...,9}{
  \draw[help lines,thick,dotted]
    (\x,-8) -- (\x,9);
}


\begin{scope}[y=(45:1)]
\foreach \y  in {-8,-7,...,9}{
  \draw[help lines,thick,dotted]
   [rotate=45]  (-8,\y) -- (9,\y) ;
}

\foreach \y  in {-8,-7,...,9}{
  \draw[help lines,dotted]
   [rotate=-45]  (-8,\y) -- (9,\y) ;
}
\end{scope}

\fill (Origin) circle (2pt) node[below left] {0};
\fill (pi1) circle (2pt) node[below right] {$\varpi_1$};
\fill (pi2) circle (2pt) node[below right] {$\varpi_2$};
\fill (a1) circle (2pt) node[below] {$\alpha_1$};
\fill (a2) circle (2pt) node[left] {$\alpha_2$};
\fill (a3) circle (2pt) node[below right] {$2\alpha_1+\alpha_2$};

\fill (asum) circle (2pt) node[above right] {$\alpha_1+\alpha_2$};
\fill (2rho) circle (2pt) node[below] {$2\rho$};

\fill (v0) node[left] {$\frac{7}{2}\varpi_2$};
\fill (v1) circle (2pt) node[right] {$A_1$};
\fill (v2) node[right] {$5\varpi_1$};

\fill (barycenter) circle (2pt) node[right] {$\textbf{bar}_{DH}(\Delta_{2})$};

\draw[->,blue,thick](Origin)--(pi1);
\draw[->,blue,thick](Origin)--(pi2); 
\draw[->,blue,thick](Origin)--(a1);
\draw[->,blue,thick](Origin)--(a2);
\draw[->,blue,thick](Origin)--(a3); 
\draw[->,blue,thick](Origin)--(asum); 

\draw[thick](Origin)--(v0);
\draw[thick](v0)--(v1);
\draw[thick](v1)--(v2);
\draw[thick](Origin)--(v2);
\draw[thick](v1)--(v2);

\aeMarkRightAngle[size=6pt](Origin,v0,v1)
\aeMarkRightAngle[size=6pt](v1,v2,Origin)

\draw [shorten >=-4cm, red, thick, dashed] (2rho) to ($(2rho)+(a1)$);
\draw [shorten >=-4cm, red, thick, dashed] (2rho) to ($(2rho)+(a2)$);
\end{tikzpicture} 
\subcaption{$\overline{\Sp_4}(2)$ (Case-I.2)}
\label{Sp-I.2}
\medskip
\end{minipage}

\begin{minipage}[b]{.45 \textwidth}
 \centering

\begin{tikzpicture}
\clip (-1.7,-0.5) rectangle (4.5,4); 

\coordinate (pi1) at (0.5,0.5);
\coordinate (pi2) at (0,1);
\coordinate (v0) at ($3.5*(pi2)$);
\coordinate (v1) at ($1*(pi1)+3*(pi2)$);
\coordinate (v2) at ($4*(pi1)+1*(pi2)$);
\coordinate (v3) at ($5*(pi1)$);
\coordinate (a1) at (1,0);
\coordinate (a2) at (-1,1);
\coordinate (a3) at ($2*(a1)+(a2)$);
\coordinate (barycenter) at (512/273-32/9,64/9);

\coordinate (Origin) at (0,0);
\coordinate (asum) at ($(a1)+(a2)$);
\coordinate (2rho) at (1,3);

\coordinate (barycenter) at (1.24365687508394, 2.90893286866132);

\foreach \x  in {-8,-7.5,...,9}{
  \draw[help lines,thick,dotted]
    (\x,-8) -- (\x,9);
}


\begin{scope}[y=(45:1)]
\foreach \y  in {-8,-7,...,9}{
  \draw[help lines,thick,dotted]
   [rotate=45]  (-8,\y) -- (9,\y) ;
}

\foreach \y  in {-8,-7,...,9}{
  \draw[help lines,dotted]
   [rotate=-45]  (-8,\y) -- (9,\y) ;
}
\end{scope}

\fill (Origin) circle (2pt) node[below left] {0};
\fill (pi1) circle (2pt) node[below right] {$\varpi_1$};
\fill (pi2) circle (2pt) node[below right] {$\varpi_2$};
\fill (a1) circle (2pt) node[below] {$\alpha_1$};
\fill (a2) circle (2pt) node[left] {$\alpha_2$};
\fill (a3) circle (2pt) node[below right] {$2\alpha_1+\alpha_2$};

\fill (asum) circle (2pt) node[above right] {$\alpha_1+\alpha_2$};
\fill (2rho) circle (2pt) node[left] {$2\rho$};

\fill (v0) node[left] {$\frac{7}{2}\varpi_2$};
\fill (v1) circle (2pt) node[right] {$A_1$};
\fill (v2) circle (2pt) node[right] {$A_2$};
\fill (v3) node[right] {$5\varpi_1$};

\fill (barycenter) circle (2pt) node[below] {$\textbf{bar}_{DH}(\Delta_{3})$};

\draw[->,blue,thick](Origin)--(pi1);
\draw[->,blue,thick](Origin)--(pi2); 
\draw[->,blue,thick](Origin)--(a1);
\draw[->,blue,thick](Origin)--(a2);
\draw[->,blue,thick](Origin)--(a3); 
\draw[->,blue,thick](Origin)--(asum); 

\draw[thick](Origin)--(v0);
\draw[thick](v0)--(v1);
\draw[thick](v1)--(v2);
\draw[thick](v2)--(v3);
\draw[thick](Origin)--(v3);

\aeMarkRightAngle[size=6pt](Origin,v0,v1)
\aeMarkRightAngle[size=6pt](v2,v3,Origin)

\draw [shorten >=-4cm, red, thick, dashed] (2rho) to ($(2rho)+(a1)$);
\draw [shorten >=-4cm, red, thick, dashed] (2rho) to ($(2rho)+(a2)$);
\end{tikzpicture} 
\subcaption{$\overline{\Sp_4}(3)$ (Case-I.1.2)}
\label{Sp-I.1.2}

\end{minipage}
\begin{minipage}[b]{.45 \textwidth}
 \centering

\begin{tikzpicture}
\clip (-1.7,-0.5) rectangle (4.5,4); 

\coordinate (pi1) at (0.5,0.5);
\coordinate (pi2) at (0,1);
\coordinate (v0) at ($3.5*(pi2)$);
\coordinate (v1) at ($1*(pi1)+3*(pi2)$);
\coordinate (v2) at ($5*(pi1)$);
\coordinate (a1) at (1,0);
\coordinate (a2) at (-1,1);
\coordinate (a3) at ($2*(a1)+(a2)$);
\coordinate (barycenter) at (512/273-32/9,64/9);

\coordinate (Origin) at (0,0);
\coordinate (asum) at ($(a1)+(a2)$);
\coordinate (2rho) at (1,3);

\coordinate (barycenter) at (1.12623815037073, 2.84300059086440);

\foreach \x  in {-8,-7.5,...,9}{
  \draw[help lines,thick,dotted]
    (\x,-8) -- (\x,9);
}


\begin{scope}[y=(45:1)]
\foreach \y  in {-8,-7,...,9}{
  \draw[help lines,thick,dotted]
   [rotate=45]  (-8,\y) -- (9,\y) ;
}

\foreach \y  in {-8,-7,...,9}{
  \draw[help lines,dotted]
   [rotate=-45]  (-8,\y) -- (9,\y) ;
}
\end{scope}

\fill (Origin) circle (2pt) node[below left] {0};
\fill (pi1) circle (2pt) node[below right] {$\varpi_1$};
\fill (pi2) circle (2pt) node[below right] {$\varpi_2$};
\fill (a1) circle (2pt) node[below] {$\alpha_1$};
\fill (a2) circle (2pt) node[left] {$\alpha_2$};
\fill (a3) circle (2pt) node[below right] {$2\alpha_1+\alpha_2$};

\fill (asum) circle (2pt) node[above right] {$\alpha_1+\alpha_2$};
\fill (2rho) circle (2pt) node[left] {$2\rho$};

\fill (v0) node[left] {$\frac{7}{2}\varpi_2$};
\fill (v1) circle (2pt) node[right] {$A_1$};
\fill (v2) circle (2pt) node[right] {$A_2$};

\fill (barycenter) circle (2pt) node[below] {$\textbf{bar}_{DH}(\Delta_4)$};

\draw[->,blue,thick](Origin)--(pi1);
\draw[->,blue,thick](Origin)--(pi2); 
\draw[->,blue,thick](Origin)--(a1);
\draw[->,blue,thick](Origin)--(a2);
\draw[->,blue,thick](Origin)--(a3); 
\draw[->,blue,thick](Origin)--(asum); 

\draw[thick](Origin)--(v0);
\draw[thick](v0)--(v1);
\draw[thick](v1)--(v2);
\draw[thick](Origin)--(v2);
\draw[thick](v1)--(v2);

\aeMarkRightAngle[size=6pt](Origin,v0,v1)

\draw [shorten >=-4cm, red, thick, dashed] (2rho) to ($(2rho)+(a1)$);
\draw [shorten >=-4cm, red, thick, dashed] (2rho) to ($(2rho)+(a2)$);
\end{tikzpicture} 
\subcaption{$\overline{\Sp_4}(4)$ (Case-I.1.1)}
\label{Sp-I.1.1}

\end{minipage}

\caption{Moment polytopes of Gorenstein Fano compactifications of $\Sp_4(\mathbb{C})$.}
\label{Sp-polytope}

\end{figure}

\clearpage

\subsubsection{Gorenstein Fano group compactification of $\PSp_4(\mathbb C) \cong \SO_5(\mathbb{C})$}

As $\PSp_4(\mathbb C)$ is of adjoint type, 
the spherical weight lattice $\mathcal M$ is spanned by $\alpha_1 = (1, 0)$ and $\alpha_2=(-1, 1)$, and  
its dual lattice $\mathcal N$ is spanned by $\varpi_1^{\vee}=(1,1)$ and $\varpi_2^{\vee}=(0, 1)$.
The Weyl walls are given by $W_1=\{y=x\}$ and $W_2=\{x=0\}$.
The sum of positive roots is $2\rho=(1,3)$.

\begin{theorem}
There are four Gorenstein Fano equivariant compactifications of $\SO_5(\mathbb{C})$: two smooth compactifications and two singular compactifications. 
Their moment polytopes are given in the following Table~\ref{table;SO5} and Figure~\ref{SO-polytope}. 
They all admit singular K\"{a}hler--Einstein metrics. 
{\rm 
\begin{table}[h]
\begin{tabular}{|c|c|l|l|l|c|c|}\hline
$X$ & $\Delta(K_X^{-1})$ & 	Case & Edges (except Weyl walls)  & Vertices	&	Smoothness & 	KE \\
\hline
$\overline{\SO_5}(1)$ & $\Delta_1$	&	I.2 &	$y=4$  & $(0,4), (4,4)$	& smooth	&	Yes 
\\ \hline
$\overline{\SO_5}(2)$ & $\Delta_2$	&	I.1 &	$y=4,\ x+y=5$  & $(0,4), (1,4), \left(\frac{5}{2},\frac{5}{2}\right)$	& smooth	&	Yes	
\\ \hline
$\overline{\SO_5}(3)$ & $\Delta_3$	&	II.1	&	$x+y=5$  & $(0,5), \left(\frac{5}{2},\frac{5}{2}\right)$	& singular	&	Yes 
\\ \hline
$\overline{\SO_5}(4)$ & $\Delta_4$	&	II.2	&	$x+2y=8,\ x+y=5$  & $(0,4),(2,3), \left(\frac{5}{2},\frac{5}{2}\right)$	& singular	&	Yes  
\\ \hline
\end{tabular}
\caption{Gorenstein Fano equivariant compactifications of $\SO_5(\mathbb{C})$.}
\label{table;SO5}
\end{table}}
\end{theorem}

\begin{proof}
Let $m_i\varpi_1^{\vee}+n_i\varpi_2^{\vee}=(m_i,m_i+n_i)\in\mathcal N$ be the primitive outer normal vector of the facet
$$
F_i=\{m_ix+(m_i+n_i)y=4m_i+3n_i+1\},
$$
where $m_i\geq m_i+n_i\geq0$ by the convexity of the polytope.

Let $F_1$ be the facet which intersects the Weyl wall $W_2$, then we have two cases:
\begin{itemize}
\item {\bf Case-I:} $F_1\perp W_2$ $\Rightarrow$ $(m_1,n_1)=(0,1)$ so that $F_1=\{y=4\}$, $F_1\cap W_2=(0,4)$.\\
Note that $F_1\cap F_2=:A_1=\Big(\frac{-n_2+1}{m_2},4\Big)\in \mathcal M$.\\ $\Rightarrow$ $A_1=5\alpha_1+4\alpha_2$ or $6\alpha_1+4\alpha_2$ or $7\alpha_1+4\alpha_2$ or $8\alpha_1+4\alpha_2$.\\
If $A_1=6\alpha_1+4\alpha_2$ or $7\alpha_1+4\alpha_2$, there is no lattice polytope.
\begin{itemize}
\item {\bf Case-I.1:} $A_1=5\alpha_1+4\alpha_2=(1,4)$ 
$\Rightarrow$ $m_2+n_2=1$\\ $\Rightarrow$ $F_2=\{m_2x+y=m_2+4\}$\\
Convexity and primitive condition $\Rightarrow$ $m_2=1$ so that $ F_2=\{x+y=5\}$\\ {\small\bf ($\Delta_2$,\ Figure ~\ref{SO-I.1})}. 
\item {\bf Case-I.2:} $A_1=8\alpha_1+4\alpha_2=(4,4)$ {\small\bf ($\Delta_1$,\ Figure ~\ref{SO-I.2})}. 
\end{itemize}
\item {\bf Case-II:} $F_1\cap W_2=:A_1=\Big(0,\frac{4m_1+3n_1+1}{m_1+n_1}\Big)\in \mathcal M$.\\
By convexity $m_1+n_1\geq m_1>0$ $\Rightarrow$ $A_1=(0,5)$ or $(0,4)$.
\begin{itemize}
\item {\bf Case-II.1:} $A_1=(0,5)$ $\Rightarrow$ $(m_1,n_1)=(1,0)$ $\Rightarrow$ $F_1=\{x+y=5\}\perp W_1$\\ {\small\bf ($\Delta_3$,\ Figure ~\ref{SO-II.1})}. 
\item {\bf Case-II.2:} $A_1=(0,4)$ $\Rightarrow$ $n_1=1$\\ $\Rightarrow$ $F_1=\{m_1x+(m_1+1)y=4(m_1+1)\}$\\
Convexity and primitive conditions $\Rightarrow$ $(m_1,n_1)=(1,1)$ so that $F_1=\{x+2y=8\}$\\
$A_2:=F_1\cap F_2=4\varpi_1+1\varpi_2=(2,3)$ and $F_2=\{x+y=5\}\perp W_1$ {\small\bf ($\Delta_4$,\ Figure ~\ref{SO-II.2})}. 
\end{itemize}
\end{itemize}

(1) Case-I.2. 
The variety $\overline{\SO_5}(1)$ is the 10-dimensional spinor variety $\mathbb S_5 \subset \mathbb P^{15}$ which is one component of the variety parametrizing 5-dimensional isotropic subspaces in the 10-dimensional orthogonal vector space (see \cite[Theorem~5]{Ruzzi2010}). 
As $\overline{\SO_5}(1)$ is a homogeneous variety, it naturally admits a K\"{a}hler--Einstein metric. 

(2) Case-I.1.  
The variety $\overline{\SO_5}(2)$ is the blow-up of the smooth Fano compactification $\overline{\SO_5}(1)$ along the closed orbit corresponding to a vertex $A_1$ in Figure~\ref{SO-I.2}. 
Indeed, $\overline{\SO_5}(2)$ is the wonderful compactification of $\SO_5(\mathbb C)$. 
Since the barycenter 
$$
\textbf{bar}_{DH}(\Delta_{2}) = \left(\frac{6332682925}{5547479872}, \frac{18534175955}{5547479872}\right) \approx (1.415, 3.341)
$$
of the moment polytope $\Delta_{2}$ with respect to the Duistermaat--Heckman measure is in the relative interior of the translated cone $2 \rho + \mathcal C^+$, 
$\overline{\SO_5}(2)$ admit a K\"{a}hler--Einstein metric by Theorem~\ref{KE criterion for Q-Fano group compactifications}.

(3) Case-II.1. 
As the toric polytope $\Delta_{3}^{toric}$ formed by the Weyl group action from the moment polytope $\Delta_3$ is not a Delzant polytope, 
$\overline{\SO_5}(3)$ is singular. 
Since the barycenter 
$$
\textbf{bar}_{DH}(\Delta_{3}) = \left(\frac{725}{704}, \frac{225}{64}\right) \approx (1.030, 3.516)
$$
is in the relative interior of the cone $2 \rho + \mathcal C^+$, 
$\overline{\SO_5}(3)$ admits a singular K\"{a}hler--Einstein metric.

(4) Case-II.2. 
As the toric polytope $\Delta_{4}^{toric}$ formed by the Weyl group action from the moment polytope $\Delta_4$ is not a Delzant polytope, 
$\overline{\SO_5}(4)$ is singular. 
Since the barycenter 
$$
\textbf{bar}_{DH}(\Delta_{4}) = \left(\frac{1110073615}{943350848}, \frac{8612750675}{2830052544}\right) \approx (1.177, 3.043)
$$
is in the relative interior of the cone $2 \rho + \mathcal C^+$, 
$\overline{\SO_5}(4)$ admits a singular K\"{a}hler--Einstein metric.
\end{proof}

\vskip 1em


\begin{figure}[h]

\begin{minipage}[b]{.45 \textwidth}
 \centering

\begin{tikzpicture}
\clip (-1.7,-0.5) rectangle (4.5,4.25); 

\coordinate (pi1) at (0.5,0.5);
\coordinate (pi2) at (0,1);
\coordinate (v0) at ($4*(pi2)$);
\coordinate (v2) at ($8*(pi1)$);
\coordinate (a1) at (1,0);
\coordinate (a2) at (-1,1);
\coordinate (a3) at ($2*(a1)+(a2)$);
\coordinate (barycenter) at (512/273-32/9,64/9);

\coordinate (Origin) at (0,0);
\coordinate (asum) at ($(a1)+(a2)$);
\coordinate (2rho) at (1,3);

\coordinate (barycenter) at (1.98863636363636, 3.63636363636364);

\foreach \x  in {-8,-7,...,9}{
  \draw[help lines,dotted]
    (\x,-8) -- (\x,9);
}

\foreach \x  in {-8.5,-7.5,...,9.5}{
  \draw[help lines,dotted]
    (\x,-8) -- (\x,9);
}

\foreach \x  in {-8,-7,...,9}{
  \draw[help lines,thick,dotted]
    (-8,\x) -- (9,\x);
}

\begin{scope}[y=(45:1)]
\foreach \y  in {-8,-7,...,9}{
  \draw[help lines,dotted]
   [rotate=45]  (-8,\y) -- (9,\y) ;
}
\foreach \y  in {-8,-6,...,8}{
  \draw[help lines,dotted]
   [rotate=45]  (-8,\y) -- (9,\y) ;
}
\end{scope}

\begin{scope}[y=(45:1)]
\foreach \y  in {-8,-7,...,9}{
  \draw[help lines,dotted]
   [rotate=-45]  (-8,\y) -- (9,\y) ;
}
\foreach \y  in {-8,-7,...,8}{
  \draw[help lines,thick,dotted]
   [rotate=-45]  (-8,\y) -- (9,\y) ;
}
\end{scope}

\fill (Origin) circle (2pt) node[below left] {0};
\fill (pi1) circle (2pt) node[below right] {$\varpi_1$};
\fill (pi2) circle (2pt) node[below right] {$\varpi_2$};
\fill (a1) circle (2pt) node[below] {$\alpha_1$};
\fill (a2) circle (2pt) node[left] {$\alpha_2$};
\fill (a3) circle (2pt) node[below right] {$2\alpha_1+\alpha_2$};

\fill (asum) circle (2pt) node[above right] {$\alpha_1+\alpha_2$};
\fill (2rho) circle (2pt) node[below] {$2\rho$};

\fill (v0) node[left] {$4\varpi_2$};
\fill (v2) circle (2pt) node[below] {$A_1$};

\fill (barycenter) circle (2pt) node[below] {$\textbf{bar}_{DH}(\Delta_{1})$};

\draw[->,blue,thick](Origin)--(pi1);
\draw[->,blue,thick](Origin)--(pi2); 
\draw[->,blue,thick](Origin)--(a1);
\draw[->,blue,thick](Origin)--(a2);
\draw[->,blue,thick](Origin)--(a3); 
\draw[->,blue,thick](Origin)--(asum); 

\draw[thick](Origin)--(v0);
\draw[thick](v0)--(v2);
\draw[thick](Origin)--(v2);

\aeMarkRightAngle[size=6pt](Origin,v0,v2)

\draw [shorten >=-4cm, red, thick, dashed] (2rho) to ($(2rho)+(a1)$);
\draw [shorten >=-4cm, red, thick, dashed] (2rho) to ($(2rho)+(a2)$);
\end{tikzpicture} 
\subcaption{$\overline{\SO_5}(1)$ (Case-I.2)}
\label{SO-I.2}
\medskip
\end{minipage}
 \begin{minipage}[b]{.45 \textwidth}
 \centering

\begin{tikzpicture}
\clip (-1.7,-0.5) rectangle (4.5,4.25); 

\coordinate (pi1) at (0.5,0.5);
\coordinate (pi2) at (0,1);
\coordinate (v0) at ($4*(pi2)$);
\coordinate (v1) at ($2*(pi1)+3*(pi2)$);
\coordinate (v2) at ($5*(pi1)$);
\coordinate (a1) at (1,0);
\coordinate (a2) at (-1,1);
\coordinate (a3) at ($2*(a1)+(a2)$);
\coordinate (barycenter) at (512/273-32/9,64/9);

\coordinate (Origin) at (0,0);
\coordinate (asum) at ($(a1)+(a2)$);
\coordinate (2rho) at (1,3);

\coordinate (barycenter) at (1.14154229868651, 3.34100823845224);

\foreach \x  in {-8,-7,...,9}{
  \draw[help lines,dotted]
    (\x,-8) -- (\x,9);
}

\foreach \x  in {-8.5,-7.5,...,9.5}{
  \draw[help lines,dotted]
    (\x,-8) -- (\x,9);
}

\foreach \x  in {-8,-7,...,9}{
  \draw[help lines,thick,dotted]
    (-8,\x) -- (9,\x);
}

\begin{scope}[y=(45:1)]
\foreach \y  in {-8,-7,...,9}{
  \draw[help lines,dotted]
   [rotate=45]  (-8,\y) -- (9,\y) ;
}
\foreach \y  in {-8,-6,...,8}{
  \draw[help lines,dotted]
   [rotate=45]  (-8,\y) -- (9,\y) ;
}
\end{scope}

\begin{scope}[y=(45:1)]
\foreach \y  in {-8,-7,...,9}{
  \draw[help lines,dotted]
   [rotate=-45]  (-8,\y) -- (9,\y) ;
}
\foreach \y  in {-8,-7,...,8}{
  \draw[help lines,thick,dotted]
   [rotate=-45]  (-8,\y) -- (9,\y) ;
}
\end{scope}

\fill (Origin) circle (2pt) node[below left] {0};
\fill (pi1) circle (2pt) node[below right] {$\varpi_1$};
\fill (pi2) circle (2pt) node[below right] {$\varpi_2$};
\fill (a1) circle (2pt) node[below] {$\alpha_1$};
\fill (a2) circle (2pt) node[left] {$\alpha_2$};
\fill (a3) circle (2pt) node[below right] {$2\alpha_1+\alpha_2$};

\fill (asum) circle (2pt) node[above right] {$\alpha_1+\alpha_2$};
\fill (2rho) circle (2pt) node[below] {$2\rho$};

\fill (v0) node[left] {$4\varpi_2$};
\fill (v1) circle (2pt) node[right] {$A_1$};
\fill (v2) node[right] {$5\varpi_1$};

\fill (barycenter) circle (2pt) node[right] {$\textbf{bar}_{DH}(\Delta_{2})$};

\draw[->,blue,thick](Origin)--(pi1);
\draw[->,blue,thick](Origin)--(pi2); 
\draw[->,blue,thick](Origin)--(a1);
\draw[->,blue,thick](Origin)--(a2);
\draw[->,blue,thick](Origin)--(a3); 
\draw[->,blue,thick](Origin)--(asum); 

\draw[thick](Origin)--(v0);
\draw[thick](v0)--(v1);
\draw[thick](v1)--(v2);
\draw[thick](Origin)--(v2);
\draw[thick](v1)--(v2);

\aeMarkRightAngle[size=6pt](Origin,v0,v1)
\aeMarkRightAngle[size=6pt](v1,v2,Origin)

\draw [shorten >=-4cm, red, thick, dashed] (2rho) to ($(2rho)+(a1)$);
\draw [shorten >=-4cm, red, thick, dashed] (2rho) to ($(2rho)+(a2)$);
\end{tikzpicture} 
\subcaption{$\overline{\SO_5}(2)$ (Case-I.1)}
\label{SO-I.1}
\medskip
\end{minipage}

\begin{minipage}[b]{.45 \textwidth}
 \centering

\begin{tikzpicture}
\clip (-1.7,-0.5) rectangle (4.5,5.2); 

\coordinate (pi1) at (0.5,0.5);
\coordinate (pi2) at (0,1);
\coordinate (v0) at ($5*(pi2)$);
\coordinate (v1) at ($5*(pi1)$);
\coordinate (a1) at (1,0);
\coordinate (a2) at (-1,1);
\coordinate (a3) at ($2*(a1)+(a2)$);
\coordinate (barycenter) at (512/273-32/9,64/9);

\coordinate (Origin) at (0,0);
\coordinate (asum) at ($(a1)+(a2)$);
\coordinate (2rho) at (1,3);

\coordinate (barycenter) at (1.02982954545455, 3.515625);

\foreach \x  in {-8,-7,...,9}{
  \draw[help lines,dotted]
    (\x,-8) -- (\x,9);
}

\foreach \x  in {-8.5,-7.5,...,9.5}{
  \draw[help lines,dotted]
    (\x,-8) -- (\x,9);
}

\foreach \x  in {-8,-7,...,9}{
  \draw[help lines,thick,dotted]
    (-8,\x) -- (9,\x);
}

\begin{scope}[y=(45:1)]
\foreach \y  in {-8,-7,...,9}{
  \draw[help lines,dotted]
   [rotate=45]  (-8,\y) -- (9,\y) ;
}
\foreach \y  in {-8,-6,...,8}{
  \draw[help lines,dotted]
   [rotate=45]  (-8,\y) -- (9,\y) ;
}
\end{scope}

\begin{scope}[y=(45:1)]
\foreach \y  in {-8,-7,...,9}{
  \draw[help lines,dotted]
   [rotate=-45]  (-8,\y) -- (9,\y) ;
}
\foreach \y  in {-8,-7,...,8}{
  \draw[help lines,thick,dotted]
   [rotate=-45]  (-8,\y) -- (9,\y) ;
}
\end{scope}

\fill (Origin) circle (2pt) node[below left] {0};
\fill (pi1) circle (2pt) node[below right] {$\varpi_1$};
\fill (pi2) circle (2pt) node[below right] {$\varpi_2$};
\fill (a1) circle (2pt) node[below] {$\alpha_1$};
\fill (a2) circle (2pt) node[left] {$\alpha_2$};
\fill (a3) circle (2pt) node[below right] {$2\alpha_1+\alpha_2$};

\fill (asum) circle (2pt) node[above right] {$\alpha_1+\alpha_2$};
\fill (2rho) circle (2pt) node[below] {$2\rho$};

\fill (v0) circle (2pt) node[right] {$A_1$};
\fill (v1) node[right] {$5\varpi_1$};

\fill (barycenter) circle (2pt) node[right] {$\textbf{bar}_{DH}(\Delta_{3})$};

\draw[->,blue,thick](Origin)--(pi1);
\draw[->,blue,thick](Origin)--(pi2); 
\draw[->,blue,thick](Origin)--(a1);
\draw[->,blue,thick](Origin)--(a2);
\draw[->,blue,thick](Origin)--(a3); 
\draw[->,blue,thick](Origin)--(asum); 

\draw[thick](Origin)--(v0);
\draw[thick](v0)--(v1);
\draw[thick](Origin)--(v1);

\aeMarkRightAngle[size=6pt](v0,v1,Origin)

\draw [shorten >=-4cm, red, thick, dashed] (2rho) to ($(2rho)+(a1)$);
\draw [shorten >=-4cm, red, thick, dashed] (2rho) to ($(2rho)+(a2)$);
\end{tikzpicture} 
\subcaption{$\overline{\SO_5}(3)$ (Case-II.1)}
\label{SO-II.1}

\end{minipage}
\begin{minipage}[b]{.45 \textwidth}
 \centering

\begin{tikzpicture}
\clip (-1.7,-0.5) rectangle (4.5,5.2); 

\coordinate (pi1) at (0.5,0.5);
\coordinate (pi2) at (0,1);
\coordinate (v0) at ($4*(pi2)$);
\coordinate (v1) at ($4*(pi1)+(pi2)$);
\coordinate (v2) at ($5*(pi1)$);
\coordinate (a1) at (1,0);
\coordinate (a2) at (-1,1);
\coordinate (a3) at ($2*(a1)+(a2)$);
\coordinate (barycenter) at (512/273-32/9,64/9);

\coordinate (Origin) at (0,0);
\coordinate (asum) at ($(a1)+(a2)$);
\coordinate (2rho) at (1,3);

\coordinate (barycenter) at (1.17673463415385, 3.04331829218504);

\foreach \x  in {-8,-7,...,9}{
  \draw[help lines,dotted]
    (\x,-8) -- (\x,9);
}

\foreach \x  in {-8.5,-7.5,...,9.5}{
  \draw[help lines,dotted]
    (\x,-8) -- (\x,9);
}

\foreach \x  in {-8,-7,...,9}{
  \draw[help lines,thick,dotted]
    (-8,\x) -- (9,\x);
}

\begin{scope}[y=(45:1)]
\foreach \y  in {-8,-7,...,9}{
  \draw[help lines,dotted]
   [rotate=45]  (-8,\y) -- (9,\y) ;
}
\foreach \y  in {-8,-6,...,8}{
  \draw[help lines,dotted]
   [rotate=45]  (-8,\y) -- (9,\y) ;
}
\end{scope}

\begin{scope}[y=(45:1)]
\foreach \y  in {-8,-7,...,9}{
  \draw[help lines,dotted]
   [rotate=-45]  (-8,\y) -- (9,\y) ;
}
\foreach \y  in {-8,-7,...,8}{
  \draw[help lines,thick,dotted]
   [rotate=-45]  (-8,\y) -- (9,\y) ;
}
\end{scope}

\fill (Origin) circle (2pt) node[below left] {0};
\fill (pi1) circle (2pt) node[below right] {$\varpi_1$};
\fill (pi2) circle (2pt) node[below right] {$\varpi_2$};
\fill (a1) circle (2pt) node[below] {$\alpha_1$};
\fill (a2) circle (2pt) node[left] {$\alpha_2$};
\fill (a3) circle (2pt) node[below right] {$2\alpha_1+\alpha_2$};

\fill (asum) circle (2pt) node[above right] {$\alpha_1+\alpha_2$};
\fill (2rho) circle (2pt) node[below] {$2\rho$};

\fill (v0) circle (2pt) node[above] {$A_1$};
\fill (v1) circle (2pt) node[below] {$A_2$};
\fill (v2) node[right] {$5\varpi_1$};

\fill (barycenter) circle (2pt) node[above right] {$\textbf{bar}_{DH}(\Delta_{4})$};

\draw[->,blue,thick](Origin)--(pi1);
\draw[->,blue,thick](Origin)--(pi2); 
\draw[->,blue,thick](Origin)--(a1);
\draw[->,blue,thick](Origin)--(a2);
\draw[->,blue,thick](Origin)--(a3); 
\draw[->,blue,thick](Origin)--(asum); 

\draw[thick](Origin)--(v0);
\draw[thick](v0)--(v1);
\draw[thick](v1)--(v2);
\draw[thick](Origin)--(v2);

\aeMarkRightAngle[size=6pt](v1,v2,Origin)

\draw [shorten >=-4cm, red, thick, dashed] (2rho) to ($(2rho)+(a1)$);
\draw [shorten >=-4cm, red, thick, dashed] (2rho) to ($(2rho)+(a2)$);
\end{tikzpicture} 
\subcaption{$\overline{\SO_5}(4)$ (Case-II.2)}
\label{SO-II.2}

\end{minipage}
\caption{Moment polytopes of Gorenstein Fano compactifications of $\SO_5(\mathbb{C})$.}
\label{SO-polytope}
\end{figure}
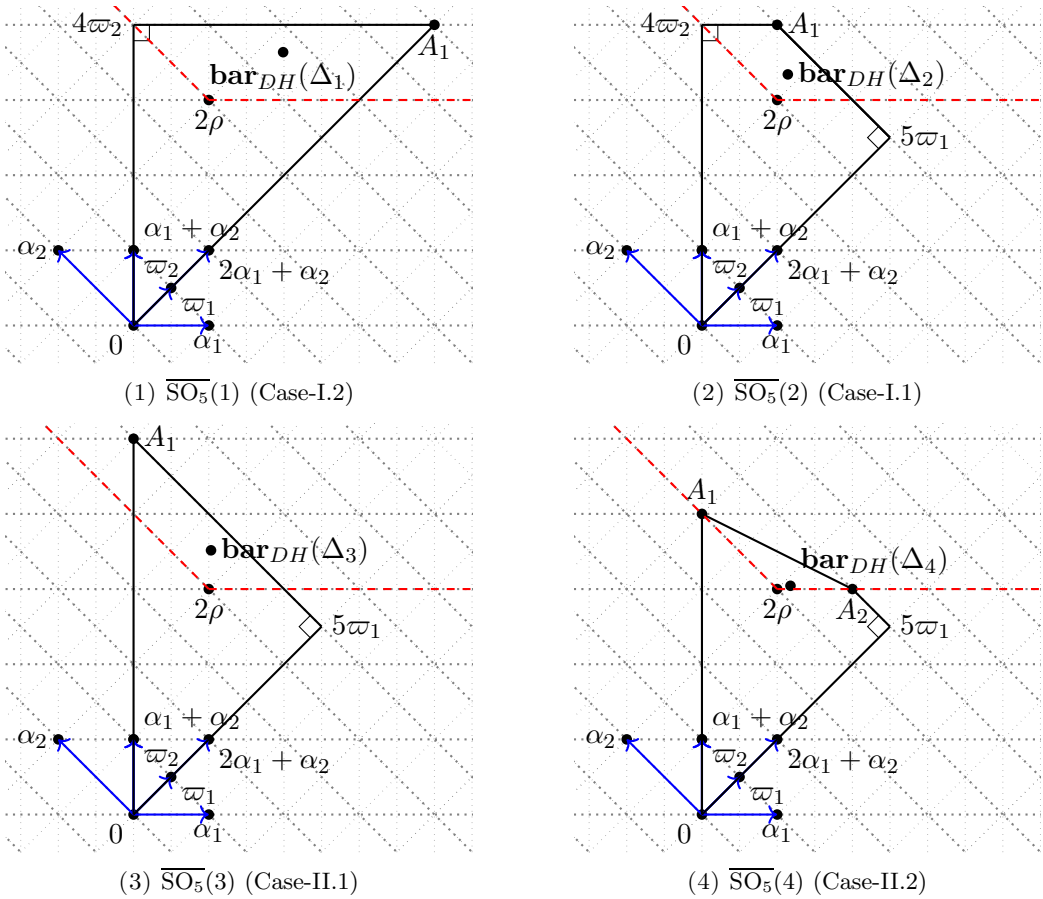

\newpage
\subsection{Gorenstein Fano group compactification of $G_2$} 
As the root lattice and weight lattice of $G_2$ coincide, the simply-connected complex Lie group of type $G_2$ is of adjoint type. 
Thus, the spherical weight lattice $\mathcal M $ is spanned by the fundamental weights $\varpi_1=\Big(\frac{1}{2},\frac{\sqrt{3}}{2}\Big)$ and $\varpi_2=(0,\sqrt{3})$, and  
its dual lattice $\mathcal N $ is spanned by $\alpha_1^{\vee}=(2, 0)$ and $\alpha_2^{\vee}=(-1, \frac{1}{\sqrt{3}})$. 
The Weyl walls are given by $W_1=\{y=\sqrt{3}x\}$ and $W_2=\{x=0\}$.
The sum of positive roots is $2\rho=(1,3\sqrt{3})$.
Note that the complex Lie group $G_2$ has 6 positive roots:
 \begin{align*}
 \Phi^+ & = \{ \alpha_1, \alpha_2, \alpha_1 + \alpha_2, 2 \alpha_1 + \alpha_2, 3\alpha_1 + \alpha_2, 3\alpha_1 + 2\alpha_2 \} \\
 & = \Big\{ (1, 0), \Big(-\frac{3}{2}, \frac{\sqrt{3}}{2} \Big), \Big(-\frac{1}{2}, \frac{\sqrt{3}}{2} \Big), \Big(\frac{1}{2}, \frac{\sqrt{3}}{2} \Big), \Big(\frac{3}{2}, \frac{\sqrt{3}}{2} \Big),  (0, \sqrt{3}) \Big\}.
 \end{align*}
Thus, the Duistermaat--Heckman measure is given as   
\[
\prod_{\alpha \in \Phi^+} \kappa(\alpha, p)^2 \, dp 
= x^2 \Big(-\frac{3}{2}x + \frac{\sqrt{3}}{2}y\Big)^2 \Big(-\frac{1}{2}x + \frac{\sqrt{3}}{2}y\Big)^2 \Big(\frac{1}{2}x + \frac{\sqrt{3}}{2}y\Big)^2 \Big(\frac{3}{2}x + \frac{\sqrt{3}}{2}y\Big)^2 (\sqrt{3}y)^2 \, dxdy.
\]

\begin{theorem}
There are only two Gorenstein Fano equivariant compactifications of $G_2$. 
Their moment polytopes are given in the following Table~\ref{table;G2} and Figure~\ref{G-polytope}. 
They all are smooth and admit K\"{a}hler--Einstein metrics. 
{\rm \small
\begin{table}[h]
\begin{tabular}{|c|c|l|l|l|c|c|}\hline
		$X$ & $\Delta(K_X^{-1})$ & 	Case & Edges (except Weyl walls)  & Vertices  &	Smoothness &	KE \\
\hline
$\overline{G_2}(1)$	& $\Delta_1$ &	I.1	&	$y=\frac{7}{2}\sqrt{3},\ x+\sqrt{3}y=11$  & $\Big(0,\frac{7}{2}\sqrt{3}\Big), \Big(\frac{1}{2},\frac{7}{2}\sqrt{3}\Big), \Big(\frac{11}{4},\frac{11}{4}\sqrt{3}\Big)$	& smooth	& Yes
\\ \hline
$\overline{G_2}(2)$	&	$\Delta_2$ & I.2	&	$y=\frac{7}{2}\sqrt{3}$  &	$\Big(\frac{7}{2},\frac{7}{2}\sqrt{3}\Big)$ 	& smooth	& Yes 
\\ \hline
\end{tabular}
\caption{Gorenstein Fano equivariant compactifications of $G_2$.}
\label{table;G2}
\vskip -1em
\end{table}}
\end{theorem}
\begin{proof}
Let $m_i\alpha_1^{\vee}+n_i\alpha_2^{\vee}=\Big(2m_i-n_i,\frac{n_i}{\sqrt{3}}\Big)\in\mathcal N$ be the primitive outer normal vector of the facet 
$$
F_i=\left\{(2m_i-n_i)x+\frac{n_i}{\sqrt{3}}y=2m_i+2n_i+1 \right\},
$$
where $\frac{n_i}{3}\geq2m_i-n_i\geq0$ by the convexity of the polytope.

Let $F_1$ be the facet which intersects the Weyl wall $W_2$, we have two cases:
\begin{itemize}
\item {\bf Case-I:} 
$F_1\perp W_2$ $\Rightarrow$ $F_1=\Big\{y=\frac{7}{2}\sqrt{3}\Big\}$\\ 
$\Rightarrow$ $A_1=\varpi_1+3\varpi_2$ or $3\varpi_1+2\varpi_2$ or $5\varpi_1+\varpi_2$ or $7\varpi_1$\\
If $A_1=3\varpi_1+2\varpi_2$ or $5\varpi_1+\varpi_2$, there is no lattice polytope.
\begin{itemize}
\item {\bf Case-I.1}: $A_1=\varpi_1+3\varpi_2=\Big(\frac{1}{2},\frac{7}{2}\sqrt{3}\Big)$ $\Rightarrow$ $(m_2,n_2)=(2,3)$ and $F_2=\{x+\sqrt{3}y=11\}$ {\small\bf($\Delta_1$,\ Figure ~\ref{G-I.1})}
\item {\bf Case-I.2}: $A_1=7\varpi_1=\Big(\frac{7}{2},\frac{7}{2}\sqrt{3}\Big)$ {\small\bf($\Delta_2$,\ Figure ~\ref{G-I.2})}
\end{itemize}
\item {\bf Case-II:} $F_1$ is not orthogonal to $W_2$.\\
$F_1\cap W_2=\Big(0,(\frac{2m_1+2n_1+1}{n_1})\sqrt{3}\Big)\in \mathcal M$ $\Rightarrow$ 
$\frac{4}{3}+\frac{1}{n_1}\geq\frac{2m_1+1}{n_1}>1+\frac{1}{n_1}$ $\Rightarrow$ No lattice polytope
\end{itemize}

(1) Case-I.1. 
This variety $\overline{G_2}(1)$ is the wonderful compactification of $G_2$. 
Computation shows $$\text{Vol}_{DH}(\Delta_{1}) = 
\displaystyle \int_{0}^{\frac{11 \sqrt{3}}{4}} \int_{0}^{\frac{y}{\sqrt{3}}} \prod_{\alpha \in \Phi^+} \kappa(\alpha, p )^2 \, dp  + \int_{\frac{11 \sqrt{3}}{4}}^{\frac{7 \sqrt{3}}{2}} \int_{0}^{-\sqrt{3}y+11} \prod_{\alpha \in \Phi^+} \kappa(\alpha, p )^2 \, dp
=\frac{107945390367459}{8830976}\sqrt{3}.$$  
This implies that 
$$
\textbf{bar}_{DH}(\Delta_{1}) = \Big(\frac{32567112922303267}{27292859194142720}, \frac{247470028273390111}{81878577582428160} \sqrt{3}\Big) \approx (1.193, 3.022 \times \sqrt{3}).
$$
Since it is in the relative interior of the translated cone $2 \rho + \mathcal C^+$, $\overline{G_2}(1)$ admits a K\"{a}hler--Einstein metric. 

(2) Case-I.2. 
This variety $\overline{G_2}(2)$ is the smooth Fano compactification of $G_2$ with Picard number one called the \emph{double Cayley Grassmannian} (see \cite{Manivel20}).
In this case, the barycenter 
$$\textbf{bar}_{DH}(\Delta_{2}) = \Big(\frac{139601}{79360}, \frac{49}{15} \sqrt{3}\Big) \approx (1.759, 3.267 \times \sqrt{3})$$
is in the relative interior of the translated cone $2 \rho + \mathcal C^+$ so that $\overline{G_2}(2)$ admits a K\"{a}hler--Einstein metric (see \cite[Section 3.6]{LPY21}). 
\end{proof}

\vskip 1em

\begin{figure}[h!]

\begin{minipage}[b]{.45 \textwidth}
 \centering

\begin{tikzpicture}
\clip (-2.5,-0.5) rectangle (4.5,6.5); 

\foreach \x  in {-8,-7.5,...,9}{
  \draw[help lines,thick,dotted]
    (\x,-8) -- (\x,9);
}
\begin{scope}[y=(60:1)]

\coordinate (pi1) at (0,1);
\coordinate (pi2) at (-1,2);

\coordinate (v0) at ($7/2*(pi2)$);
\coordinate (v1) at ($(pi1)+3*(pi2)$);
\coordinate (v2) at ($11/2*(pi1)$);

\coordinate (a1) at (1,0);
\coordinate (a2) at (-2,1);
\coordinate (a3) at ($3*(a1)+(a2)$);

\coordinate (Origin) at (0,0);
\coordinate (asum) at ($(a1)+(a2)$);
\coordinate (2rho) at (-2,6);

\coordinate (barycenter) at (1.193-3.022,6.044);

\foreach \x  in {-8,-7,...,9}{
  \draw[help lines,dotted]
    (\x,-8) -- (\x,9)
    (-8,\x) -- (9,\x) 
    [rotate=60] (\x,-8) -- (\x,9) ;
}
\foreach \x  in {-8,-7,...,9}{
  \draw[help lines,thick,dotted]
    (\x,-8) -- (\x,9);
}

\fill (Origin) circle (2pt) node[below left] {0};
\fill (pi1) circle (2pt) node[right] {$\varpi_1$};
\fill (pi2) circle (2pt) node[right] {$\varpi_2$};
\fill (a1) circle (2pt) node[below] {$\alpha_1$};
\fill (a2) circle (2pt) node[above left] {$\alpha_2$};
\fill (a3) circle (2pt) node[below right] {$3\alpha_1+\alpha_2$};

\fill (asum) circle (2pt) node[above] {$\alpha_1+\alpha_2$};
\fill (2rho) circle (2pt) node[below] {$2\rho$};

\fill (v1) circle (2pt) node[right] {$A_1$};

\fill (barycenter) circle (2pt) node[right] {$\textbf{bar}_{DH}(\Delta_{1})$};

\draw[->,blue,thick](Origin)--(pi1);
\draw[->,blue,thick](Origin)--(pi2); 
\draw[->,blue,thick](Origin)--(a1);
\draw[->,blue,thick](Origin)--(a2);
\draw[->,blue,thick](Origin)--(a3); 
\draw[->,blue,thick](Origin)--(asum); 

\draw[thick](Origin)--(v0);
\draw[thick](Origin)--(v2);
\draw[thick](v0)--(v1);
\draw[thick](v1)--(v2);

\aeMarkRightAngle[size=6pt](Origin,v0,v1)
\aeMarkRightAngle[size=6pt](v1,v2,Origin)

\draw [shorten >=-4cm, red, thick, dashed] (2rho) to ($(2rho)+(a1)$);
\draw [shorten >=-4cm, red, thick, dashed] (2rho) to ($(2rho)+(a2)$);
\end{scope}
\end{tikzpicture} 
\subcaption{$\overline{G_2}(1)$ (Case-I.1)}
\label{G-I.1}

\end{minipage}
\begin{minipage}[b]{.45 \textwidth}
\centering

\begin{tikzpicture}
\clip (-2.5,-0.5) rectangle (4.7,6.5); 
\foreach \x  in {-8,-7.5,...,9}{
  \draw[help lines,thick,dotted]
    (\x,-8) -- (\x,9);
}
\begin{scope}[y=(60:1)]

\coordinate (pi1) at (0,1);
\coordinate (pi2) at (-1,2);
\coordinate (v1) at ($7*(pi1)$);
\coordinate (v2) at ($7/2*(pi2)$);
\coordinate (a1) at (1,0);
\coordinate (a2) at (-2,1);
\coordinate (a3) at ($3*(a1)+(a2)$);

\coordinate (Origin) at (0,0);
\coordinate (asum) at ($(a1)+(a2)$);
\coordinate (2rho) at (-2,6);

\coordinate (barycenter) at (1.759-49/15,98/15);

\foreach \x  in {-8,-7,...,9}{
  \draw[help lines,dotted]
    (\x,-8) -- (\x,9)
    (-8,\x) -- (9,\x) 
    [rotate=60] (\x,-8) -- (\x,9) ;
}
\foreach \x  in {-8,-7,...,9}{
  \draw[help lines,thick,dotted]
    (\x,-8) -- (\x,9);
}

\fill (Origin) circle (2pt) node[below left] {0};
\fill (pi1) circle (2pt) node[right] {$\varpi_1$};
\fill (pi2) circle (2pt) node[right] {$\varpi_2$};
\fill (a1) circle (2pt) node[below] {$\alpha_1$};
\fill (a2) circle (2pt) node[above left] {$\alpha_2$};
\fill (a3) circle (2pt) node[below right] {$3\alpha_1+\alpha_2$};

\fill (asum) circle (2pt) node[above] {$\alpha_1+\alpha_2$};
\fill (2rho) circle (2pt) node[below] {$2\rho$};

\fill (v1) circle (2pt) node[right] {$A_1$};
\fill (v2) node[left] {$\frac{7}{2}\varpi_2$};

\fill (barycenter) circle (2pt) node[below right] {$\, \textbf{bar}_{DH}(\Delta_{2})$};

\draw[->,blue,thick](Origin)--(pi1);
\draw[->,blue,thick](Origin)--(pi2); 
\draw[->,blue,thick](Origin)--(a1);
\draw[->,blue,thick](Origin)--(a2);
\draw[->,blue,thick](Origin)--(a3); 
\draw[->,blue,thick](Origin)--(asum); 

\draw[thick](Origin)--(v1);
\draw[thick](Origin)--(v2);
\draw[thick](v1)--(v2);

\aeMarkRightAngle[size=6pt](Origin,v2,v1)

\draw [shorten >=-4cm, red, thick, dashed] (2rho) to ($(2rho)+(a1)$);
\draw [shorten >=-4cm, red, thick, dashed] (2rho) to ($(2rho)+(a2)$);
\end{scope}
\end{tikzpicture} 
\subcaption{$\overline{G_2}(2)$ (Case-I.2)}
\label{G-I.2}

\end{minipage}
\caption{Moment polytopes of Gorenstein Fano compactifications of $G_2$.}
\label{G-polytope}
\end{figure}
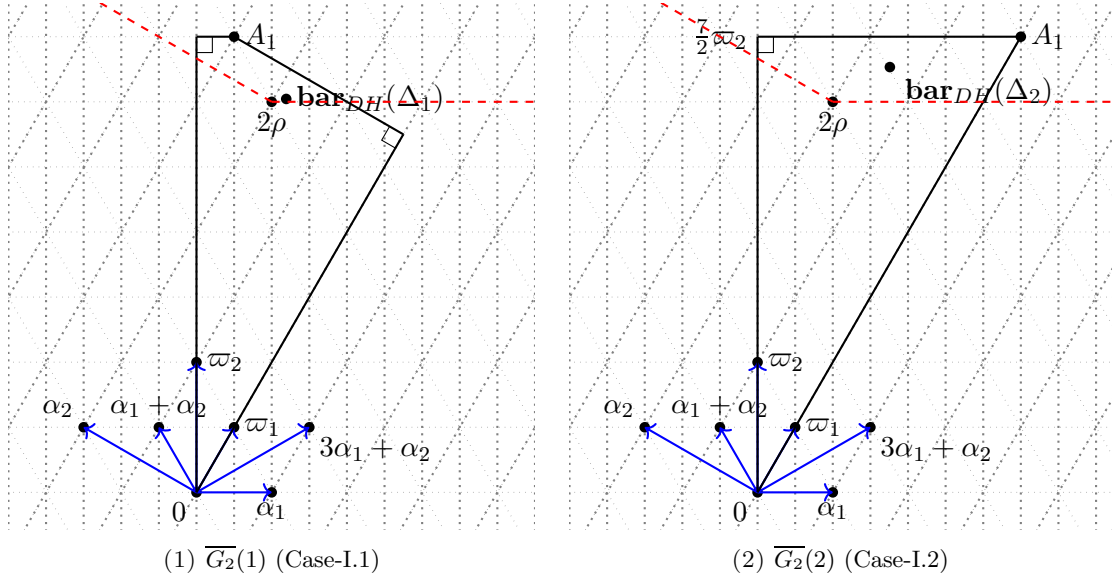

\clearpage
\subsection{$\mathsf{A_1}\times\mathsf{A_1}$-type}
\subsubsection{Gorenstein Fano group compactification of $\SL_2(\mathbb{C}) \times \SL_2(\mathbb{C})$}

As $\SL_2(\mathbb C) \times \SL_2(\mathbb C)$ is simply-connected, 
the spherical weight lattice $\mathcal M$ is spanned by $\varpi_1 = (1, 0)$ and $\varpi_2 = (0, 1)$, and 
its dual lattice $\mathcal N$ is spanned by $\alpha_1^{\vee} = (1, 0)$ and $\alpha_2^{\vee} = (0, 1)$.
The Weyl walls are given by $W_1=\{y=0\}$ and $W_2=\{x=0\}$.
The sum of positive roots is $2\rho=(2,2)$.
Choosing a realization of the root system $\mathsf{A_1} \times \mathsf{A_1}$ in the Euclidean plane $\mathbb R^2$ with $\varpi_1 = (1, 0)$ and $\varpi_2 = (0, 1)$, 
for $p=(x, y)$ 
we obtain the Duistermaat--Heckman measure 
\begin{equation*}
\prod_{\alpha \in \Phi^+} \kappa(\alpha, p)^2 \, dp 
= (2x)^2 (2y)^2 \, dxdy
\end{equation*}
because $\mathsf{A_1} \times \mathsf{A_1}$ has 2 positive roots:
$\Phi^+ = \{ \alpha_1, \alpha_2\} = \{ (2, 0), (0, 2) \}$.


\begin{theorem}
There are 15 Gorenstein Fano equivariant compactifications of $\SL_2(\mathbb{C}) \times \SL_2(\mathbb{C})$ up to isomorphism: two smooth compactifications and 13 singular compactifications. 
Their moment polytopes are given in the following Table~\ref{table:SL2*SL2} and Figure~\ref{SLSL-polytope1}. 
Among them only two smooth Fano compactifications and a unique singular one admit K\"{a}hler--Einstein metrics. 
{\rm
\begin{table}[h]
{\renewcommand{\arraystretch}{1.2}
\footnotesize{
\begin{tabular}{|c|l|l|l|c|c|c|}\hline
$X$ & 	Case & Edges (except Weyl walls)  & Vertices	&	Smooth? & 	KE & $\textbf{bar}(\Delta)$\\
\hline
$\overline{\SL_2\times \SL_2}(1)$	&	I.1.1 &	$y=3,\ x+2y=7$  & $(0,3),(1,3),(7,0)$	& singular	&	No & $\Big(\frac{11427}{3794}, \frac{5661}{3794}\Big)$
\\ \hline
$\overline{\SL_2\times \SL_2}(2)$	&	I.1.1.a &	$y=3,\ x+2y=7,\ x+y=5$  & $(0,3),(1,3),(3,2),(5,0)$ 	& singular	&	No & $\Big(\frac{174753}{72086}, \frac{119919}{72086}\Big)$
\\ \hline
$\overline{\SL_2\times \SL_2}(3)$	&	I.1.1.b &	$y=3,\ x+2y=7,\ x=3$  & $(0,3),(1,3),(3,2),(3,0)$ 	& singular	&	No & $\Big(\frac{106833}{52486}, \frac{4233}{2282}\Big)$
\\ \hline
$\overline{\SL_2\times \SL_2}(4)$	&	I.1.2 &	$y=3,\ 2x+3y=11,\ x+y=5$  & $(0,3),(1,3),(4,1),(5,0)$		& singular	&	No & $\Big(\frac{130089}{55657}, \frac{87441}{55657}\Big)$
\\ \hline
$\overline{\SL_2\times \SL_2}(5)$&	I.1.3 &	$y=3,\ 3x+4y=15$  & $(0,3),(1,3),(5,0)$ 	& singular	&	No & $\Big(\frac{243}{112}, \frac{177}{112}\Big)$
\\ \hline
$\overline{\SL_2\times \SL_2}(6)$&	I.2.1 &	$y=3,\ x+y=5$  & $(0,3),(2,3),(5,0)$ 	& singular	&	No & $\Big(\frac{1569}{665}, \frac{1233}{665}\Big)$
\\ \hline
$\overline{\SL_2\times \SL_2}(7)$	&	I.2.1.a &	$y=3,\ x+y=5,\ x=3$  & $(0,3),(2,3),(3,2),(3,0)$	& smooth	&	Yes & $\Big(\frac{28779}{14035}, \frac{28779}{14035}\Big)$
\\ \hline
$\overline{\SL_2\times \SL_2}(8)$	&	I.2.2 &	$y=3,\ 3x+y=9$  & $(0,3),(2,3),(3,0)$ 	& singular	&	No & $\Big(\frac{2817}{1631}, \frac{3411}{1631}\Big)$
\\ \hline
$\overline{\SL_2\times \SL_2}(9)$	&	I.3	&	$y=3,\ x=3$  & $(0,3),(3,3)(3,0),$ 	& smooth	&	Yes & $\Big(\frac{9}{4}, \frac{9}{4}\Big)$
\\ \hline
$\overline{\SL_2\times \SL_2}(10)$	&	II.1 &		$x+y=5$  & $(0,5),(5,0)$	& singular	&	Yes & $\Big(\frac{15}{7}, \frac{15}{7}\Big)$
\\ \hline
$\overline{\SL_2\times \SL_2}(11)$&	II.2.1 &	$x+3y=9$  & $(0,3),(9,0)$ 	& singular	&	No  & $\Big(\frac{27}{7}, \frac{9}{7}\Big)$
\\ \hline
$\overline{\SL_2\times \SL_2}(12)$&	II.2.1.a	&	$x+3y=9,\ x+2y=7$  & $(0,3),(3,2),(7,0)$ 	& singular	&	No & $\Big(\frac{293793}{93989}, \frac{131547}{93989}\Big)$
\\ \hline
$\overline{\SL_2\times \SL_2}(13)$&	II.2.1.b &		$x+3y=9,\ x+y=5$  & $(0,3),(3,2),(5,0)$	& singular	&	No  & $\Big(\frac{160017}{63637}, \frac{98619}{63637}\Big)$
\\ \hline
$\overline{\SL_2\times \SL_2}(14)$&	II.2.2 &	$2x+5y=15,\ x+2y=7$  & $(0,3),(5,1),(7,0)$ 	& singular	&	No & $\Big(\frac{245313}{77896}, \frac{101469}{77896}\Big)$
\\ \hline
$\overline{\SL_2\times \SL_2}(15)$	&	II.2.3 &	$3x+7y=21$  & $(0,3),(7,0)$	& singular	&	No & $\Big(3, \frac{9}{7}\Big)$
\\ \hline
\end{tabular}
}}
\caption{Gorenstein Fano equivariant compactifications of $\SL_2(\mathbb{C}) \times \SL_2(\mathbb{C})$.}
\label{table:SL2*SL2}
\end{table}}
\end{theorem}

\begin{proof}
Let $m_i\alpha_1^{\vee}+n_i\alpha_2^{\vee}=(m_i,n_i)\in\mathcal N$ be the primitive outer normal vector of the facet 
$$
F_i=\{m_ix+n_iy=2m_i+2n_i+1\},
$$
where $m_i>0, n_i\geq0$ by the convexity of the polytope.

\begin{itemize}
\item {\bf Case-I:} $F_1$ is orthogonal to $W_2$.\\
$F_1\perp W_2$ $\Rightarrow$ $F_1=\{y=3\}$\\
$F_1\cap F_2=:A_1=\Big(\frac{2m_2-n_2+1}{m_2},3\Big)=\Big(2+\frac{-n_2+1}{m_2},3\Big)\in \mathcal M$ 
$\Rightarrow$ $A_1=(1,3)$ or $(2,3)$ or $(3,3)$
\begin{itemize}
\item {\bf Case-I.1:} $A_1=(1,3)$ $\Rightarrow$ $m_2+1=n_2$
\begin{itemize}
\item {\bf Case-I.1.1:} $(m_2,n_2)=(1,2)$ and $F_2=\{x+2y=7\}$, $A_2=(7,0)$\\ {\small\bf($\Delta_1$,\ Figure ~\ref{SLSL-I.1.1})}.
\item {\bf Case-I.1.1.a:} $A_2=(3,2)$ and $F_3=\{x+y=5\}$ {\small\bf($\Delta_2$,\ Figure ~\ref{SLSL-I.1.1.a})}.
\item {\bf Case-I.1.1.b:} $A_2=(3,2)$ and $F_3=\{x=3\}$ {\small\bf($\Delta_3$,\ Figure ~\ref{SLSL-I.1.1.b})}.
\item {\bf Case-I.1.2}: $(m_2,n_2)=(2,3)$ and $F_2=\{2x+3y=11\}$, $A_2=(4,1)$ and $F_3=\{x+y=5\}$ {\small\bf($\Delta_4$,\ Figure ~\ref{SLSL-I.1.2})}.
\item {\bf Case-I.1.3}: $(m_2,n_2)=(3,4)$ and $F_2=\{3x+4y=15\}$, $A_2=(5,0)$\\ {\small\bf($\Delta_5$,\ Figure ~\ref{SLSL-I.1.3})}.
\end{itemize}
\item {\bf Case-I.2}: $A_1=(2,3)$ $\Rightarrow$ $n_2=1$\\
If $(m_2,n_2)=(2,1)$ so that $F_2=\{2x+y=7\}$, then $A_2=(3,1)$ and $F_3=\{x=3\}$.\\ 
In this case, the polytope is isomorphic to Case-I.1.1.b.
\begin{itemize}
\item {\bf Case-I.2.1}: $(m_2,n_2)=(1,1)$ and $F_2=\{x+y=5\}$, $A_2=(5,0)$\\ {\small\bf($\Delta_6$,\ Figure ~\ref{SLSL-I.2.1})}.
\item {\bf Case-I.2.1.a}: $A_2=(3,2)$ and $F_3=\{x=3\}$ {\small\bf($\Delta_7$,\ Figure ~\ref{SLSL-I.2.1.a})}.
\item {\bf Case-I.2.2}: $(m_2,n_2)=(3,1)$ and $F_2=\{3x+y=9\}$, $A_2=(3,0)$\\ {\small\bf($\Delta_8$,\ Figure ~\ref{SLSL-I.2.2})}.
\end{itemize}
\item {\bf Case-I.3}: $A_1=(3,3)$ $\Rightarrow$ $(m_2,n_2)=(1,0)$ so that $F_2=\{x=3\}$ {\small\bf($\Delta_9$,\ Figure ~\ref{SLSL-I.3})}.
\end{itemize}
\item {\bf Case-II:} $F_1$ is not orthogonal to $W_2$.\\
$F_1\cap W_2=:A_1=(0,\frac{2m_1+2n_1+1}{n_1})$ where $n_1\geq m_1>0$ (Convexity and symmetry)\\
$A_1\in \mathcal M$ $\Rightarrow$ $\frac{2m_1+1}{n_1}\in\mathbb N_+$ $\Rightarrow$
$\frac{2m_1+1}{n_1}=1 (A_1=(0,3))$ or  $\frac{2m_1+1}{n_1}=3 (A_1=(0,5))$ (primitive)
\begin{itemize}
\item {\bf Case-II.1}: $\frac{2m_1+1}{n_1}=3$ $\Rightarrow$ $(m_1,n_1)=(1,1)$ and $F_1=\{x+y=5\}$, $A_2=(5,0)$ \\ {\small\bf($\Delta_{10}$,\ Figure ~\ref{SLSL-II.1})}.
\item {\bf Case-II.2}: $\frac{2m_1+1}{n_1}=1$ $\Rightarrow$ $F_1=\{m_1x+(2m_1+1)y=6m_1+3\}$ $\Rightarrow$ $m_1=1,2,3$
\begin{itemize}
\item {\bf Case-II.2.1}: $(m_1,n_1)=(1,3)$ and $F_1=\{x+3y=9\}$, $A_2=(9,0)$ {\small\bf($\Delta_{11}$,\ Figure ~\ref{SLSL-II.2.1})}.
\item {\bf Case-II.2.1.a}: $A_2=(3,2)$ and $F_2=\{x+2y=7\}$ {\small\bf($\Delta_{12}$,\ Figure ~\ref{SLSL-II.2.1.(a)})}.
\item {\bf Case-II.2.1.b}: $A_2=(3,2)$ and $F_2=\{x+y=5\}$ {\small\bf($\Delta_{13}$,\ Figure ~\ref{SLSL-II.2.1.(b)})}.
\item {\bf Case-II.2.2}: $(m_1,n_1)=(2,5)$ and $F_1=\{2x+5y=15\}$, $A_2=(5,1)$ and $F_2=\{x+2y=7\}$ {\small\bf($\Delta_{14}$,\ Figure ~\ref{SLSL-II.2.2})}.
\item {\bf Case-II.2.3}: $(m_1,n_1)=(3,7)$ and $F_1=\{3x+7y=21\}$, $A_2=(7,0)$\\ {\small\bf($\Delta_{15}$,\ Figure ~\ref{SLSL-II.2.3})}.
\end{itemize}
\end{itemize}
\end{itemize}

In the Case-I.3,
the variety $\overline{\SL_2\times \SL_2}(9)$ is isomorphic to $\mathbb Q^3 \times \mathbb Q^3$ from Example~\ref{compactifications of SL2 and PSL2}. 
As $\overline{\SL_2\times \SL_2}(9)$ is a homogeneous variety, it naturally admits a K\"{a}hler--Einstein metric. 

Applying Theorem \ref{KE criterion for Q-Fano group compactifications}, we can determine which of them are equivariant K-stable and admit (singular) K\"{a}hler-Einstein metric.
For instance, in the Case-I.2.1.a,
the variety $\overline{\SL_2\times \SL_2}(7)$ is the blow-up of the smooth Fano compactification $\overline{\SL_2\times \SL_2}(9)$ along the closed orbit which is a diagonal embedding of $\mathbb Q^3$. 
Since the barycenter 
$$
\textbf{bar}_{DH}(\Delta_{7}) = \left(\frac{28779}{14035}, \frac{28779}{14035}\right) \approx (2.051, 2.051)
$$
of the moment polytope $\Delta_{7}$ with respect to the Duistermaat--Heckman measure is in the relative interior of the translated cone $2 \rho + \mathcal C^+$, 
$\overline{\SL_2\times \SL_2}(7)$ admits a K\"{a}hler--Einstein metric by Theorem~\ref{KE criterion for Q-Fano group compactifications}.
\end{proof}

\clearpage

\begin{figure}[b]
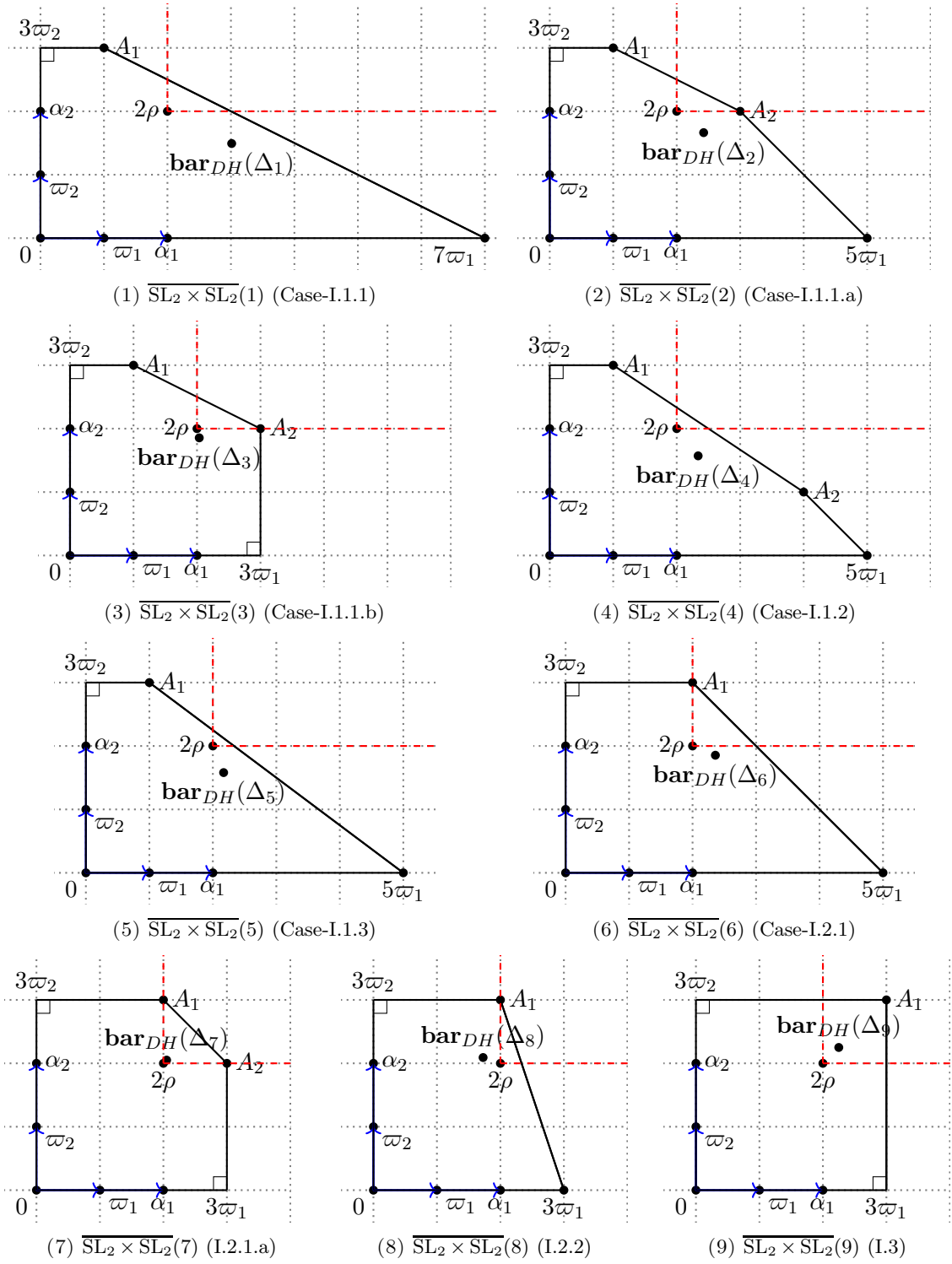


\begin{minipage}[b]{.45 \textwidth}
 \centering

 
\subcaption{$\overline{\SL_2\times \SL_2}(1)$ (Case-I.1.1)}
\label{SLSL-I.1.1}
\medskip
\end{minipage}
\begin{minipage}[b]{.45 \textwidth}
 \centering

%
 
\subcaption{$\overline{\SL_2\times \SL_2}(2)$ (Case-I.1.1.a)}
\label{SLSL-I.1.1.a}
\medskip
\end{minipage}

\begin{minipage}[b]{.45 \textwidth}
 \centering

%
 
\subcaption{$\overline{\SL_2\times \SL_2}(3)$ (Case-I.1.1.b)}
\label{SLSL-I.1.1.b}
\medskip
\end{minipage}
\begin{minipage}[b]{.45 \textwidth}
 \centering

%
 
\subcaption{$\overline{\SL_2\times \SL_2}(4)$ (Case-I.1.2)}
\label{SLSL-I.1.2}
\medskip
\end{minipage}

\begin{minipage}[b]{.45 \textwidth}
 \centering

%
 
\subcaption{$\overline{\SL_2\times \SL_2}(5)$ (Case-I.1.3)}
\label{SLSL-I.1.3}
\medskip
\end{minipage}
\begin{minipage}[b]{.45 \textwidth}
 \centering

%
 
\subcaption{$\overline{\SL_2\times \SL_2}(6)$ (Case-I.2.1)}
\label{SLSL-I.2.1}
\medskip
\end{minipage}

\begin{minipage}[b]{.3 \textwidth}

%
 
\subcaption{$\overline{\SL_2\times \SL_2}(7)$ (I.2.1.a)}
\label{SLSL-I.2.1.a}
\medskip
\end{minipage}
\begin{minipage}[b]{.3 \textwidth}
 \centering

%
 
\subcaption{$\overline{\SL_2\times \SL_2}(8)$ (I.2.2)}
\label{SLSL-I.2.2}
\medskip
\end{minipage}
\begin{minipage}[b]{.3 \textwidth}
 \centering

%
 
\subcaption{$\overline{\SL_2\times \SL_2}(9)$ (I.3)}
\label{SLSL-I.3}
\medskip
\end{minipage}
\caption{Moment polytopes of Gorenstein Fano compactifications of $\SL_2(\mathbb{C})\times\SL_2(\mathbb{C})$-(i).}
\label{SLSL-polytope1}
\end{figure}

\begin{figure}[ht]
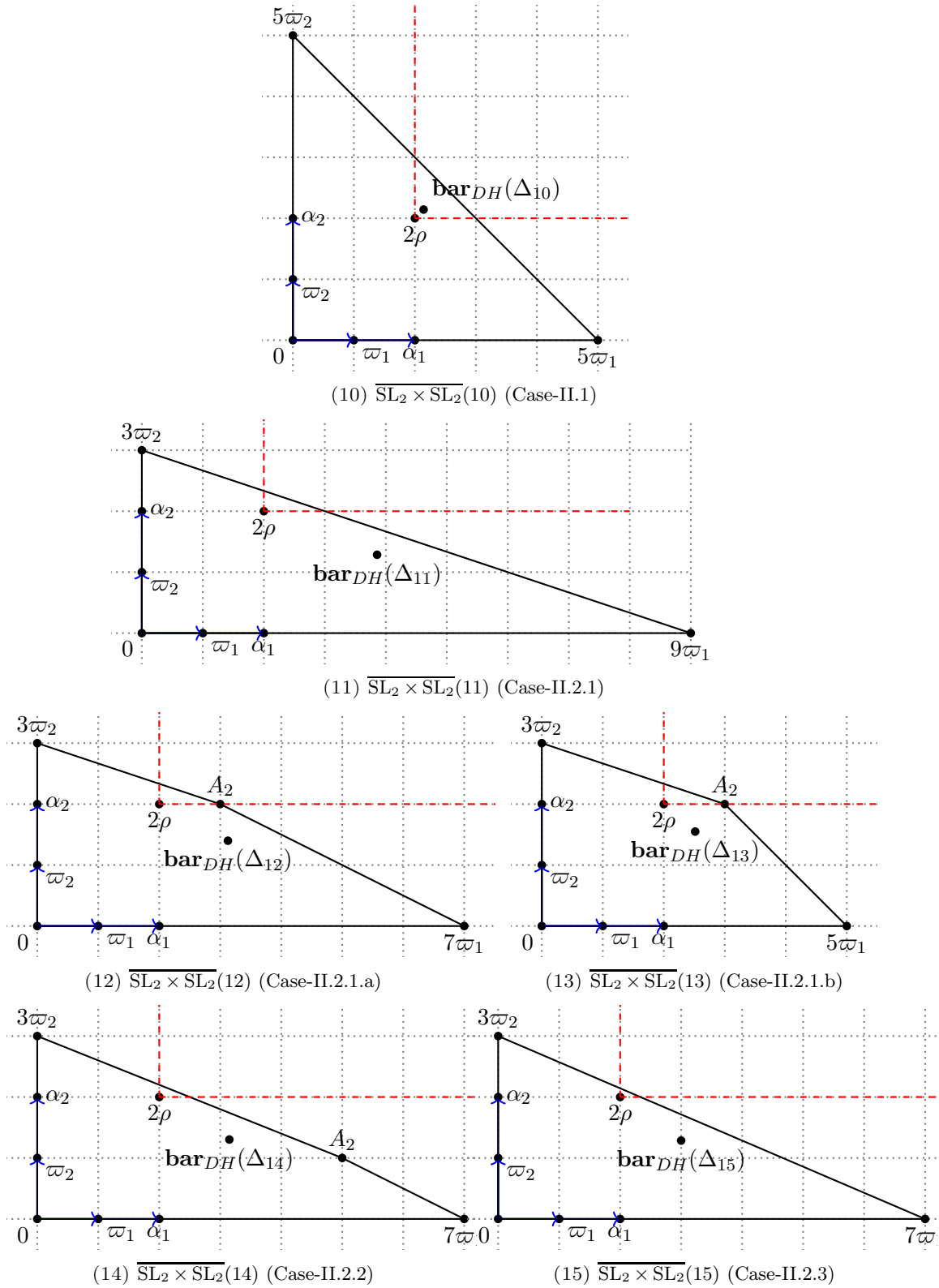

\ContinuedFloat
\begin{minipage}[b]{.4 \textwidth}

%
 
\subcaption{$\overline{\SL_2\times \SL_2}(10)$ (Case-II.1)}
\label{SLSL-II.1}
\medskip
\end{minipage}
\begin{minipage}[b]{.7 \textwidth}

%
 
\subcaption{$\overline{\SL_2\times \SL_2}(11)$ (Case-II.2.1)}
\label{SLSL-II.2.1}
\medskip
\end{minipage}

\begin{minipage}[b]{.45 \textwidth}
 \centering

%
 
\subcaption{$\overline{\SL_2\times \SL_2}(12)$ (Case-II.2.1.a)}
\label{SLSL-II.2.1.(a)}
\medskip
\end{minipage}
\begin{minipage}[b]{.45 \textwidth}
 \centering

%
 
\subcaption{$\overline{\SL_2\times \SL_2}(13)$ (Case-II.2.1.b)}
\label{SLSL-II.2.1.(b)}
\medskip
\end{minipage}

\begin{minipage}[b]{.45 \textwidth}
 \centering

%
 
\subcaption{$\overline{\SL_2\times \SL_2}(14)$ (Case-II.2.2)}
\label{SLSL-II.2.2}
\medskip
\end{minipage}
\begin{minipage}[b]{.45 \textwidth}
 \centering

%
 
\subcaption{$\overline{\SL_2\times \SL_2}(15)$ (Case-II.2.3)}
\label{SLSL-II.2.3}
\medskip
\end{minipage}
\caption{Moment polytopes of Gorenstein Fano compactifications of $\SL_2(\mathbb{C})\times\SL_2(\mathbb{C})$-(ii).}
\label{SLSL-polytope2}
\end{figure}

\clearpage

\subsubsection{Gorenstein Fano group compactification of $\PSL_2(\mathbb{C}) \times \PSL_2(\mathbb{C})$}

As $\PSL_2(\mathbb C) \times \PSL_2(\mathbb C)$ is of adjoint type, 
the spherical weight lattice $\mathcal M$ is spanned by $\alpha_1 = (2, 0)$ and $\alpha_2=(0, 2)$, and  
its dual lattice $\mathcal N$ is spanned by $\varpi_1^{\vee}=\Big(\frac{1}{2}, 0 \Big)$ and $\varpi_2^{\vee}=\Big(0, \frac{1}{2}\Big)$.
The Weyl walls are given by $W_1=\{y=0\}$ and $W_2=\{x=0\}$.
The sum of positive roots is $2\rho=(2,2)$.

\begin{theorem}
There are seven Gorenstein Fano equivariant compactifications of $\PSL_2(\mathbb{C}) \times \PSL_2(\mathbb{C})$ up to isomorphism: two smooth compactifications and five singular compactifications. 
Their moment polytopes are given in the following Table~\ref{table:PSL2*PSL2} and Figure~\ref{PSLPSL-polytope}. 
Among them only two smooth compactifications and three singular compactifications admit K\"{a}hler--Einstein metrics. 
{\rm 
\begin{table}[h]
{\renewcommand{\arraystretch}{1.2}
\footnotesize{
\begin{tabular}{|c|l|l|l|c|c|c|}\hline
$X$ & 	Case & Edges (except Weyl walls)  & Vertices	&	Smooth? & 	KE &  $\textbf{bar}(\Delta)$\\
\hline
$\overline{\PSL_2\times \PSL_2}(1)$&	I.1.1 &	$y=4,\ x+y=6$  & $(0,4),(2,4),(6,0)$	& singular	&	Yes & $\Big(\frac{783}{287}, \frac{678}{287}\Big)$
\\ \hline
$\overline{\PSL_2\times \PSL_2}(2)$&	I.1.1.a &	$y=4,\ x+y=6,\ x=4$  & $(0,4),(2,4),(4,2),(4,0)$	& smooth	&	Yes  & $\Big(\frac{10254}{4081}, \frac{10254}{4081}\Big)$
\\ \hline
$\overline{\PSL_2\times \PSL_2}(3)$	&	I.1.2 &	$y=4,\ 2x+y=8$  & $(0,4),(2,4),(4,0)$ 	& singular	&	Yes  & $\Big(\frac{99}{49}, \frac{128}{49}\Big)$
\\ \hline
$\overline{\PSL_2\times \PSL_2}(4)$	&	I.2 &	$y=4,\ x=4$  & $(0,4),(4,4),(4,0)$	& smooth	&	Yes  	& $(3,3)$
\\ \hline
$\overline{\PSL_2\times \PSL_2}(5)$	&	II.1 &	$x+y=6$  & $(0,6),(6,0)$ 	& singular	&	Yes & $\Big(\frac{18}{7}, \frac{18}{7}\Big)$
\\ \hline
$\overline{\PSL_2\times \PSL_2}(6)$	&	II.2.1 &	$x+2y=8$  & $(0,4),(8,0)$ 	& singular	&	No & $\Big(\frac{24}{7}, \frac{12}{7}\Big)$
\\ \hline
$\overline{\PSL_2\times \PSL_2}(7)$	&	II.2.1.a &$x+2y=8,\ x+y=6$  & $(0,4),(4,2),(6,0)$	& singular	&	No & $\Big(\frac{8418}{2863}, \frac{5346}{2863}\Big)$
\\ \hline
\end{tabular}
}}
\caption{Gorenstein Fano equivariant compactifications of $\PSL_2(\mathbb{C}) \times \PSL_2(\mathbb{C})$.}
\label{table:PSL2*PSL2}
\end{table}}
\end{theorem}

\begin{proof}
Let $m_i\varpi_1^{\vee}+n_i\varpi_2^{\vee}=\Big(\frac{m_i}{2},\frac{n_i}{2}\Big)\in\mathcal N$
be the primitive outer normal vector of the facet 
$$
F_i=\{m_ix+n_iy=2(m_i+n_i+1)\},
$$
where  $m_i\geq0, n_i\geq0$ by the convexity of the polytope.

For the facet $F_1$ which intersects the Weyl wall $W_2$, we have two cases:
\begin{itemize}
\item {\bf Case-I:} $F_1$ is orthogonal to $W_2$. $\Rightarrow$ $F_1=\{y=4\}$, 
$F_1\cap F_2=:A_1=\Big(2+2\frac{-n_2+1}{m_2},4\Big)\in \mathcal M$
\begin{itemize}
\item {\bf Case-I.1}: $A_1=(2,4)$ $\Rightarrow$ $n_2=1$ 
\begin{itemize}
\item {\bf Case-I.1.1}: $(m_2,n_2)=(1,1)$ and $F_2=\{x+y=6\}$, $A_2=(6,0)$\\ {\small\bf($\Delta_{1}$,\ Figure ~\ref{PSLPSL-I.1.1})}.
\item {\bf Case-I.1.1.a}: $A_2=(4,2)$ and $F_3=\{x=4\}$ {\small\bf($\Delta_{2}$,\ Figure ~\ref{PSLPSL-I.1.1.(a)})}.
\item {\bf Case-I.1.2}: $(m_2,n_2)=(2,1)$ and $F_2=\{2x+y=8\}$, $A_2=(4,0)$\\ {\small\bf($\Delta_{3}$,\ Figure ~\ref{PSLPSL-I.1.2})}.
\end{itemize}
\item {\bf Case-I.2}: $A_1=(4,4)$ $\Rightarrow$ $(m_2,n_2)=(1,0)$ so that $F_2=\{x=4\}$ {\small\bf($\Delta_{4}$,\ Figure ~\ref{PSLPSL-I.2})}.
\end{itemize}
\item {\bf Case-II:} $F_1$ is not orthogonal to $W_2$.\\
$A_1=(0,2\frac{m_1+n_1+1}{n_1})\in \mathcal M$
$\Rightarrow$ $\frac{m_1+n_1+1}{n_1}=2 (A_1=(0,4))$ or  $\frac{m_1+n_1+1}{n_1}=3 (A_1=(0,6))$
\begin{itemize}
\item {\bf Case-II.1}: $\frac{m_1+n_1+1}{n_1}=3$ $\Rightarrow$ $(m_1,n_1)=(1,1)$ and $F_1=\{x+y=6\}$, $A_2=(6,0)$\\ {\small\bf($\Delta_{5}$,\ Figure ~\ref{PSLPSL-II.1})}.
\item {\bf Case-II.2}: $\frac{m_1+n_1+1}{n_1}=2$ $\Rightarrow$ $m_1+1=n_1$ $\Rightarrow$ $F_1=\{m_1x+(m_1+1)y=4(m_1+1)\}$
\begin{itemize}
\item {\bf Case-II.2.1}: $m_1=1$ $\Rightarrow$ $F_1=\{x+2y=8\}$, $A_2=(8,0)$ {\small\bf($\Delta_{6}$,\ Figure ~\ref{PSLPSL-II.2.1})}.
\item {\bf Case-II.2.1.a}: $A_2=(4,2)$ and $F_2=\{x+y=6\}$
{\small\bf($\Delta_{7}$,\ Figure ~\ref{PSLPSL-II.2.1.(a)})}.

\end{itemize}
\end{itemize}
\end{itemize}

In the Case-I.2,
the variety $\overline{\PSL_2\times \PSL_2}(4)$ is isomorphic to $\mathbb P^3 \times \mathbb P^3$ from Example~\ref{compactifications of SL2 and PSL2}. 
As $\overline{\PSL_2\times \PSL_2}(4)$ is a homogeneous variety, it naturally admits a K\"{a}hler--Einstein metric. 

Applying Theorem \ref{KE criterion for Q-Fano group compactifications}, we can determine which of them are equivariant K-stable and admit (singular) K\"{a}hler-Einstein metric.
For instance, in the Case-I.1.1.a,  
the variety $\overline{\PSL_2\times \PSL_2}(2)$ is the blow-up of the smooth Fano compactification $\overline{\PSL_2\times \PSL_2}(4)$ along the closed orbit which is a diagonal embedding of $\mathbb P^3$. 
Since the barycenter 
$
\textbf{bar}_{DH}(\Delta_{2}) = \displaystyle \left(\frac{10254}{4081}, \frac{10254}{4081}\right) \approx (2.513, 2.513)
$
is in the relative interior of the cone $2 \rho + \mathcal C^+$, 
$\overline{\PSL_2\times \PSL_2}(2)$ admits a K\"{a}hler--Einstein metric.
\end{proof}


\begin{figure}[b]

\begin{minipage}[b]{.45 \textwidth}
 \centering

\begin{tikzpicture}[scale=0.79]
\clip (-0.5,-0.5) rectangle (6.5,4.2); 

\coordinate (pi1) at (1,0);
\coordinate (pi2) at (0,1);
\coordinate (v0) at ($4*(pi2)$);
\coordinate (v1) at ($2*(pi1)+4*(pi2)$);
\coordinate (v2) at ($6*(pi1)$);
\coordinate (a1) at (2,0);
\coordinate (a2) at (0,2);
\coordinate (barycenter) at (512/273-32/9,64/9);

\coordinate (Origin) at (0,0);
\coordinate (2rho) at (2,2);

\coordinate (barycenter) at (2.72822299651568, 2.36236933797909);

\foreach \x  in {-8,-7,...,10}{
  \draw[help lines,dotted]
    (\x,-8) -- (\x,10);
}

\foreach \x  in {-8,-7,...,10}{
  \draw[help lines,dotted]
    (-8,\x) -- (10,\x);
}

\foreach \x  in {-8,-6,...,10}{
  \draw[help lines,thick,dotted]
    (\x,-8) -- (\x,10);
}

\foreach \x  in {-8,-6,...,10}{
  \draw[help lines,thick,dotted]
     (-8,\x) -- (10,\x);
}

\fill (Origin) circle (2pt) node[below left] {0};
\fill (pi1) circle (2pt) node[below right] {$\varpi_1$};
\fill (pi2) circle (2pt) node[below right] {$\varpi_2$};
\fill (a1) circle (2pt) node[below] {$\alpha_1$};
\fill (a2) circle (2pt) node[right] {$\alpha_2$};

\fill (2rho) circle (2pt) node[below] {$2\rho$};

\fill (v1) circle (2pt) node[right] {$A_1$};
\fill (v2) circle (2pt) node[below left] {$3\alpha_1$};

\fill (barycenter) circle (2pt) node[right] {$\textbf{bar}_{DH}(\Delta_{1})$};

\draw[->,blue,thick](Origin)--(pi1);
\draw[->,blue,thick](Origin)--(pi2); 
\draw[->,blue,thick](Origin)--(a1);
\draw[->,blue,thick](Origin)--(a2);

\draw[thick](Origin)--(v0);
\draw[thick](v0)--(v1);
\draw[thick](v1)--(v2);
\draw[thick](Origin)--(v2);
\draw[thick](v1)--(v2);

\aeMarkRightAngle[size=6pt](Origin,v0,v1)

\draw [shorten >=-4cm, red, thick, dashed] (2rho) to ($(2rho)+(a1)$);
\draw [shorten >=-4cm, red, thick, dashed] (2rho) to ($(2rho)+(a2)$);
\end{tikzpicture} 
\subcaption{$\overline{\PSL_2\times \PSL_2}(1)$ (Case-I.1.1)}
\label{PSLPSL-I.1.1}
\medskip
\end{minipage}
\begin{minipage}[b]{.45 \textwidth}
 \centering

\begin{tikzpicture}[scale=0.79]
\clip (-0.5,-0.5) rectangle (6.5,4.2); 

\coordinate (pi1) at (1,0);
\coordinate (pi2) at (0,1);
\coordinate (v0) at ($4*(pi2)$);
\coordinate (v1) at ($2*(pi1)+4*(pi2)$);
\coordinate (v2) at ($4*(pi1)+2*(pi2)$);
\coordinate (v3) at ($4*(pi1)$);
\coordinate (a1) at (2,0);
\coordinate (a2) at (0,2);
\coordinate (barycenter) at (512/273-32/9,64/9);

\coordinate (Origin) at (0,0);
\coordinate (2rho) at (2,2);

\coordinate (barycenter) at (2.51261945601568, 2.51261945601568);

\foreach \x  in {-8,-7,...,10}{
  \draw[help lines,dotted]
    (\x,-8) -- (\x,10);
}

\foreach \x  in {-8,-7,...,10}{
  \draw[help lines,dotted]
    (-8,\x) -- (10,\x);
}

\foreach \x  in {-8,-6,...,10}{
  \draw[help lines,thick,dotted]
    (\x,-8) -- (\x,10);
}

\foreach \x  in {-8,-6,...,10}{
  \draw[help lines,thick,dotted]
     (-8,\x) -- (10,\x);
}

\fill (Origin) circle (2pt) node[below left] {0};
\fill (pi1) circle (2pt) node[below right] {$\varpi_1$};
\fill (pi2) circle (2pt) node[below right] {$\varpi_2$};
\fill (a1) circle (2pt) node[below] {$\alpha_1$};
\fill (a2) circle (2pt) node[right] {$\alpha_2$};

\fill (2rho) circle (2pt) node[below] {$2\rho$};

\fill (v1) circle (2pt) node[right] {$A_1$};
\fill (v2) circle (2pt) node[above] {$A_2$};
\fill (v3) node[below] {$2\alpha_1$};

\fill (barycenter) circle (2pt) node[above left] {$\textbf{bar}_{DH}(\Delta_{2})$};

\draw[->,blue,thick](Origin)--(pi1);
\draw[->,blue,thick](Origin)--(pi2); 
\draw[->,blue,thick](Origin)--(a1);
\draw[->,blue,thick](Origin)--(a2);

\draw[thick](Origin)--(v0);
\draw[thick](v0)--(v1);
\draw[thick](v1)--(v2);
\draw[thick](v2)--(v3);
\draw[thick](Origin)--(v3);

\aeMarkRightAngle[size=6pt](Origin,v0,v1)
\aeMarkRightAngle[size=6pt](v2,v3,Origin)

\draw [shorten >=-4cm, red, thick, dashed] (2rho) to ($(2rho)+(a1)$);
\draw [shorten >=-4cm, red, thick, dashed] (2rho) to ($(2rho)+(a2)$);
\end{tikzpicture} 
\subcaption{$\overline{\PSL_2\times \PSL_2}(2)$ (Case-I.1.1.a)}
\label{PSLPSL-I.1.1.(a)}

\medskip
\end{minipage}

\begin{minipage}[b]{.45 \textwidth}
 \centering

\begin{tikzpicture}[scale=0.79]
\clip (-0.5,-0.5) rectangle (6.5,4.2); 

\coordinate (pi1) at (1,0);
\coordinate (pi2) at (0,1);
\coordinate (v0) at ($4*(pi2)$);
\coordinate (v1) at ($2*(pi1)+4*(pi2)$);
\coordinate (v2) at ($4*(pi1)$);
\coordinate (a1) at (2,0);
\coordinate (a2) at (0,2);
\coordinate (barycenter) at (512/273-32/9,64/9);

\coordinate (Origin) at (0,0);
\coordinate (2rho) at (2,2);

\coordinate (barycenter) at (99/49, 128/49);

\foreach \x  in {-8,-7,...,10}{
  \draw[help lines,dotted]
    (\x,-8) -- (\x,10);
}

\foreach \x  in {-8,-7,...,10}{
  \draw[help lines,dotted]
    (-8,\x) -- (10,\x);
}

\foreach \x  in {-8,-6,...,10}{
  \draw[help lines,thick,dotted]
    (\x,-8) -- (\x,10);
}

\foreach \x  in {-8,-6,...,10}{
  \draw[help lines,thick,dotted]
     (-8,\x) -- (10,\x);
}

\fill (Origin) circle (2pt) node[below left] {0};
\fill (pi1) circle (2pt) node[below right] {$\varpi_1$};
\fill (pi2) circle (2pt) node[below right] {$\varpi_2$};
\fill (a1) circle (2pt) node[below] {$\alpha_1$};
\fill (a2) circle (2pt) node[right] {$\alpha_2$};

\fill (2rho) circle (2pt) node[below] {$2\rho$};

\fill (v1) circle (2pt) node[right] {$A_1$};
\fill (v2) circle (2pt) node[below left] {$2\alpha_1$};

\fill (barycenter) circle (2pt) node[right] {$\textbf{bar}_{DH}(\Delta_{3})$};

\draw[->,blue,thick](Origin)--(pi1);
\draw[->,blue,thick](Origin)--(pi2); 
\draw[->,blue,thick](Origin)--(a1);
\draw[->,blue,thick](Origin)--(a2);

\draw[thick](Origin)--(v0);
\draw[thick](v0)--(v1);
\draw[thick](v1)--(v2);
\draw[thick](Origin)--(v2);
\draw[thick](v1)--(v2);

\aeMarkRightAngle[size=6pt](Origin,v0,v1)

\draw [shorten >=-4cm, red, thick, dashed] (2rho) to ($(2rho)+(a1)$);
\draw [shorten >=-4cm, red, thick, dashed] (2rho) to ($(2rho)+(a2)$);
\end{tikzpicture} 
\subcaption{$\overline{\PSL_2\times \PSL_2}(3)$ (Case-I.1.2)}
\label{PSLPSL-I.1.2}
\medskip
\end{minipage}
\begin{minipage}[b]{.45 \textwidth}
 \centering

\begin{tikzpicture}[scale=0.79]
\clip (-0.5,-0.5) rectangle (6.5,4.2); 

\coordinate (pi1) at (1,0);
\coordinate (pi2) at (0,1);
\coordinate (v0) at ($4*(pi2)$);
\coordinate (v1) at ($4*(pi1)+4*(pi2)$);
\coordinate (v2) at ($4*(pi1)$);
\coordinate (a1) at (2,0);
\coordinate (a2) at (0,2);
\coordinate (barycenter) at (512/273-32/9,64/9);

\coordinate (Origin) at (0,0);
\coordinate (2rho) at (2,2);

\coordinate (barycenter) at (3, 3);

\foreach \x  in {-8,-7,...,10}{
  \draw[help lines,dotted]
    (\x,-8) -- (\x,10);
}

\foreach \x  in {-8,-7,...,10}{
  \draw[help lines,dotted]
    (-8,\x) -- (10,\x);
}

\foreach \x  in {-8,-6,...,10}{
  \draw[help lines,thick,dotted]
    (\x,-8) -- (\x,10);
}

\foreach \x  in {-8,-6,...,10}{
  \draw[help lines,thick,dotted]
     (-8,\x) -- (10,\x);
}

\fill (Origin) circle (2pt) node[below left] {0};
\fill (pi1) circle (2pt) node[below right] {$\varpi_1$};
\fill (pi2) circle (2pt) node[below right] {$\varpi_2$};
\fill (a1) circle (2pt) node[below] {$\alpha_1$};
\fill (a2) circle (2pt) node[right] {$\alpha_2$};

\fill (2rho) circle (2pt) node[below] {$2\rho$};

\fill (v1) circle (2pt) node[right] {$A_1$};
\fill (v2) circle (2pt) node[below left] {$2\alpha_1$};

\fill (barycenter) circle (2pt) node[above] {$\textbf{bar}_{DH}(\Delta_{4})$};

\draw[->,blue,thick](Origin)--(pi1);
\draw[->,blue,thick](Origin)--(pi2); 
\draw[->,blue,thick](Origin)--(a1);
\draw[->,blue,thick](Origin)--(a2);

\draw[thick](Origin)--(v0);
\draw[thick](v0)--(v1);
\draw[thick](v1)--(v2);
\draw[thick](Origin)--(v2);
\draw[thick](v1)--(v2);

\aeMarkRightAngle[size=6pt](Origin,v0,v1)
\aeMarkRightAngle[size=6pt](v1,v2,Origin)

\draw [shorten >=-4cm, red, thick, dashed] (2rho) to ($(2rho)+(a1)$);
\draw [shorten >=-4cm, red, thick, dashed] (2rho) to ($(2rho)+(a2)$);
\end{tikzpicture} 
\subcaption{$\overline{\PSL_2\times \PSL_2}(4)$ (Case-I.2)}
\label{PSLPSL-I.2}
\medskip
\end{minipage}

\begin{minipage}[b]{.45 \textwidth}
 \centering

\begin{tikzpicture}[scale=0.79]
\clip (-0.5,-0.5) rectangle (6.5,6.2); 

\coordinate (pi1) at (1,0);
\coordinate (pi2) at (0,1);
\coordinate (v0) at ($6*(pi2)$);
\coordinate (v1) at ($6*(pi1)$);
\coordinate (a1) at (2,0);
\coordinate (a2) at (0,2);
\coordinate (barycenter) at (512/273-32/9,64/9);

\coordinate (Origin) at (0,0);
\coordinate (2rho) at (2,2);

\coordinate (barycenter) at (18/7, 18/7);

\foreach \x  in {-8,-7,...,10}{
  \draw[help lines,dotted]
    (\x,-8) -- (\x,10);
}

\foreach \x  in {-8,-7,...,10}{
  \draw[help lines,dotted]
    (-8,\x) -- (10,\x);
}

\foreach \x  in {-8,-6,...,10}{
  \draw[help lines,thick,dotted]
    (\x,-8) -- (\x,10);
}

\foreach \x  in {-8,-6,...,10}{
  \draw[help lines,thick,dotted]
     (-8,\x) -- (10,\x);
}

\fill (Origin) circle (2pt) node[below left] {0};
\fill (pi1) circle (2pt) node[below right] {$\varpi_1$};
\fill (pi2) circle (2pt) node[below right] {$\varpi_2$};
\fill (a1) circle (2pt) node[below] {$\alpha_1$};
\fill (a2) circle (2pt) node[right] {$\alpha_2$};

\fill (2rho) circle (2pt) node[below] {$2\rho$};

\fill (v0) circle (2pt) node[right] {$3\alpha_2$};
\fill (v1) circle (2pt) node[below left] {$3\alpha_1$};

\fill (barycenter) circle (2pt) node[right] {$\textbf{bar}_{DH}(\Delta_{5})$};

\draw[->,blue,thick](Origin)--(pi1);
\draw[->,blue,thick](Origin)--(pi2); 
\draw[->,blue,thick](Origin)--(a1);
\draw[->,blue,thick](Origin)--(a2);

\draw[thick](Origin)--(v0);
\draw[thick](v0)--(v1);
\draw[thick](Origin)--(v1);


\draw [shorten >=-4cm, red, thick, dashed] (2rho) to ($(2rho)+(a1)$);
\draw [shorten >=-4cm, red, thick, dashed] (2rho) to ($(2rho)+(a2)$);
\end{tikzpicture} 
\subcaption{$\overline{\PSL_2\times \PSL_2}(5)$ (Case-II.1)}
\label{PSLPSL-II.1}
\medskip
\end{minipage}
\begin{minipage}[b]{.45 \textwidth}
 \centering

\begin{tikzpicture}[scale=0.79]
\clip (-0.5,-0.5) rectangle (8.5,4.2); 

\coordinate (pi1) at (1,0);
\coordinate (pi2) at (0,1);
\coordinate (v0) at ($4*(pi2)$);
\coordinate (v1) at ($8*(pi1)$);
\coordinate (a1) at (2,0);
\coordinate (a2) at (0,2);
\coordinate (barycenter) at (512/273-32/9,64/9);

\coordinate (Origin) at (0,0);
\coordinate (2rho) at (2,2);

\coordinate (barycenter) at (24/7, 12/7);

\foreach \x  in {-8,-7,...,10}{
  \draw[help lines,dotted]
    (\x,-8) -- (\x,10);
}

\foreach \x  in {-8,-7,...,10}{
  \draw[help lines,dotted]
    (-8,\x) -- (10,\x);
}

\foreach \x  in {-8,-6,...,10}{
  \draw[help lines,thick,dotted]
    (\x,-8) -- (\x,10);
}

\foreach \x  in {-8,-6,...,10}{
  \draw[help lines,thick,dotted]
     (-8,\x) -- (10,\x);
}

\fill (Origin) circle (2pt) node[below left] {0};
\fill (pi1) circle (2pt) node[below right] {$\varpi_1$};
\fill (pi2) circle (2pt) node[below right] {$\varpi_2$};
\fill (a1) circle (2pt) node[below] {$\alpha_1$};
\fill (a2) circle (2pt) node[right] {$\alpha_2$};

\fill (2rho) circle (2pt) node[below] {$2\rho$};

\fill (v0) circle (2pt) node[right] {$2\alpha_2$};
\fill (v1) circle (2pt) node[below left] {$4\alpha_1$};

\fill (barycenter) circle (2pt) node[below] {$\textbf{bar}_{DH}(\Delta_{6})$};

\draw[->,blue,thick](Origin)--(pi1);
\draw[->,blue,thick](Origin)--(pi2); 
\draw[->,blue,thick](Origin)--(a1);
\draw[->,blue,thick](Origin)--(a2);

\draw[thick](Origin)--(v0);
\draw[thick](v0)--(v1);
\draw[thick](Origin)--(v1);


\draw [shorten >=-4cm, red, thick, dashed] (2rho) to ($(2rho)+(a1)$);
\draw [shorten >=-4cm, red, thick, dashed] (2rho) to ($(2rho)+(a2)$);
\end{tikzpicture} 
\subcaption{$\overline{\PSL_2\times \PSL_2}(6)$ (Case-II.2.1)}
\label{PSLPSL-II.2.1}
\medskip
\end{minipage}

\begin{minipage}[b]{.45 \textwidth}
 \centering

\begin{tikzpicture}[scale=0.79]
\clip (-0.5,-0.5) rectangle (6.5,4.2); 

\coordinate (pi1) at (1,0);
\coordinate (pi2) at (0,1);
\coordinate (v0) at ($4*(pi2)$);
\coordinate (v1) at ($4*(pi1)+2*(pi2)$);
\coordinate (v2) at ($6*(pi1)$);
\coordinate (a1) at (2,0);
\coordinate (a2) at (0,2);
\coordinate (barycenter) at (512/273-32/9,64/9);

\coordinate (Origin) at (0,0);
\coordinate (2rho) at (2,2);

\coordinate (barycenter) at (2.94027244149494, 1.86727209221097);

\foreach \x  in {-8,-7,...,10}{
  \draw[help lines,dotted]
    (\x,-8) -- (\x,10);
}

\foreach \x  in {-8,-7,...,10}{
  \draw[help lines,dotted]
    (-8,\x) -- (10,\x);
}

\foreach \x  in {-8,-6,...,10}{
  \draw[help lines,thick,dotted]
    (\x,-8) -- (\x,10);
}

\foreach \x  in {-8,-6,...,10}{
  \draw[help lines,thick,dotted]
     (-8,\x) -- (10,\x);
}

\fill (Origin) circle (2pt) node[below left] {0};
\fill (pi1) circle (2pt) node[below right] {$\varpi_1$};
\fill (pi2) circle (2pt) node[below right] {$\varpi_2$};
\fill (a1) circle (2pt) node[below] {$\alpha_1$};
\fill (a2) circle (2pt) node[right] {$\alpha_2$};

\fill (2rho) circle (2pt) node[left] {$2\rho$};

\fill (v0) circle (2pt) node[right] {$2\alpha_2$};
\fill (v1) circle (2pt) node[above] {$A_2$};
\fill (v2) circle (2pt) node[below left] {$3\alpha_1$};

\fill (barycenter) circle (2pt) node[below] {$\textbf{bar}_{DH}(\Delta_{7})$};

\draw[->,blue,thick](Origin)--(pi1);
\draw[->,blue,thick](Origin)--(pi2); 
\draw[->,blue,thick](Origin)--(a1);
\draw[->,blue,thick](Origin)--(a2);

\draw[thick](Origin)--(v0);
\draw[thick](v0)--(v1);
\draw[thick](v1)--(v2);
\draw[thick](Origin)--(v2);


\draw [shorten >=-4cm, red, thick, dashed] (2rho) to ($(2rho)+(a1)$);
\draw [shorten >=-4cm, red, thick, dashed] (2rho) to ($(2rho)+(a2)$);
\end{tikzpicture} 
\subcaption{$\overline{\PSL_2\times \PSL_2}(7)$ (Case-II.2.1.a)}
\label{PSLPSL-II.2.1.(a)}

\end{minipage}
\caption{Moment polytopes of Gorenstein Fano compactifications of $\PSL_2(\mathbb{C})\times\PSL_2(\mathbb{C})$.}
\label{PSLPSL-polytope}
\end{figure}

\clearpage

\subsubsection{Gorenstein Fano group compactification of $\SL_2(\mathbb{C}) \times \PSL_2(\mathbb{C})$}

The spherical weight lattice $\mathcal M$ is spanned by $\varpi_1 = (1, 0)$ and $\alpha_2=(0, 2)$, and 
its dual lattice $\mathcal N$ is spanned by $\alpha_1^{\vee} = (1, 0)$ and $\varpi_2^{\vee}=\Big(0, \frac{1}{2}\Big)$. 
The Weyl walls are given by $W_1=\{y=0\}$ and $W_2=\{x=0\}$.
The sum of positive roots is $2\rho=(2,2)$.


\begin{theorem}
There are 14 Gorenstein Fano equivariant compactifications of $\SL_2(\mathbb{C}) \times \PSL_2(\mathbb{C})$ up to isomorphism: two smooth compactifications and 12 singular compactifications. 
Their moment polytopes are given in the following Table~\ref{table:SL2*PSL2} and Figure~\ref{SLPSL-polytope}. 
Among them only one smooth compactification and three singular compactifications admit singular K\"{a}hler--Einstein metrics. 
{\rm 
\begin{table}[h]
{\renewcommand{\arraystretch}{1.2}
\footnotesize{
\begin{tabular}{|c|l|l|l|c|c|c|}\hline
$X$ & 	Case & Edges (except Weyl walls)  & Vertices	&	Smooth? & 	KE &  $\textbf{bar}(\Delta)$\\
\hline
$\overline{\SL_2\times \PSL_2}(1)$&	I.1.1 &	$y=4,\ x+y=5$  & $(0,4),(1,4),(5,0)$	& singular	&	Yes & $\Big(\frac{243}{112}, \frac{59}{28}\Big)$
\\ \hline
$\overline{\SL_2\times \PSL_2}(2)$	&	I.1.1.a &	$y=4,\ x+y=5,\ x+\frac{1}{2}y=4$  & $(0,4),(1,4),(3,2),(4,0)$ 	& singular	&	Yes & $\Big(\frac{12069}{6034}, \frac{13383}{6034}\Big)$
\\ \hline
$\overline{\SL_2\times \PSL_2}(3)$	&	I.1.1.b &	$y=4,\ x+y=5,\ x=3$  & $(0,4),(1,4),(3,2),(3,0)$ 	& singular	&	No & $\Big(\frac{2067}{1099}, \frac{2550}{1099}\Big)$
\\ \hline
$\overline{\SL_2\times \PSL_2}(4)$&	I.1.2 &	$y=4,\ 4x+3y=16$  & $(0,4),(1,4),(4,0)$	& singular	&	No & $\Big(\frac{1797}{1022}, \frac{1128}{511}\Big)$
\\ \hline
$\overline{\SL_2\times \PSL_2}(5)$	&	I.2.1 &	$y=4,\ x+\frac{1}{2}y=4$  & $(0,4),(2,4),(4,0)$ 	& singular	&	Yes & $\Big(\frac{99}{49}, \frac{128}{49}\Big)$
\\ \hline
$\overline{\SL_2\times \PSL_2}(6)$	&	I.2.1.a &	$y=4,\ x+\frac{1}{2}y=4,\ x=3$  & $(0,4),(2,4),(3,2),(3,0)$ 	& smooth	&	No & $\Big(\frac{17274}{8869}, \frac{23943}{8869}\Big)$
\\ \hline
$\overline{\SL_2\times \PSL_2}(7)$&	I.2.2 &	$y=4,\ 4x+y=12$  & $(0,4),(2,4),(3,0)$		& singular	&	No & $\Big(\frac{2817}{1631}, \frac{4548}{1631}\Big)$
\\ \hline
$\overline{\SL_2\times \PSL_2}(8)$	&	I.3 &	$y=4,\ x=3$  & $(0,4),(3,4),(3,0)$ 	& smooth	&	Yes & $\Big(\frac{9}{4},3\Big)$
\\ \hline
$\overline{\SL_2\times \PSL_2}(9)$	&	II.1	&	$4x+y=12$  & $(0,12),(3,0)$ 	& singular	&	No & $\Big(\frac{9}{7}, \frac{36}{7}\Big)$
\\ \hline
$\overline{\SL_2\times \PSL_2}(10)$	&	II.2.1 &	$2x+y=8$  & $(0,8),(4,0)$		& singular	&	No & $\Big(\frac{12}{7}, \frac{24}{7}\Big)$
\\ \hline
$\overline{\SL_2\times \PSL_2}(11)$	&	II.2.1.a &	$2x+y=8,\ x=3$  & $(0,8),(3,2),(3,0)$ 	& singular	&	No & $\Big(\frac{846}{511}, \frac{1797}{511}\Big)$
\\ \hline
$\overline{\SL_2\times \PSL_2}(12)$&	II.2.1.b	&	$2x+y=8,\ 4x+y=12$  & $(0,8),(3,2),(3,0)$ 	& singular	&	No & $\Big(\frac{4209}{2863}, \frac{10692}{2863}\Big)$
\\ \hline
$\overline{\SL_2\times \PSL_2}(13)$	&	II.2.1.c &	$2x+y=8,\ 3x+y=9$  & $(0,8),(1,6),(3,0)$		& singular	&	No & $\Big(\frac{10887}{8288}, \frac{30927}{8288}\Big)$
\\ \hline
$\overline{\SL_2\times \PSL_2}(14)$&	II.2.2 &	$8x+3y=24$  & $(0,8),(3,0)$ 	& singular	&	No & $\Big(\frac{9}{7}, \frac{24}{7}\Big)$
\\ \hline
\end{tabular}
}}
\caption{Gorenstein Fano equivariant compactifications of $\SL_2(\mathbb{C}) \times \PSL_2(\mathbb{C})$.}
\label{table:SL2*PSL2}
\end{table}}
\end{theorem}

\begin{proof}
Let $m_i\varpi_1^{\vee}+n_i\alpha_2^{\vee}=(m_i,2n_i)\in\mathcal N$
be the primitive outer normal vector of the facet 
$$
F_i= \left \{m_ix+\frac{n_i}{2}y=2m_i+n_i+1 \right \},
$$
where $m_i\geq0, n_i\geq0$ by the convexity of the polytope.\\
For the facet $F_1$ which intersects the Weyl wall $W_2$, we have two cases:
\begin{itemize}
\item {\bf Case-I:} $F_1$ is orthogonal to $W_2$.\\
$F_1\perp W_2$ $\Rightarrow$ $F_1=\{y=4\}$\\
$F_1\cap F_2=:A_1=\Big(\frac{2m_2-n_2+1}{m_2},4\Big)=\Big(2+\frac{-n_2+1}{m_2},4\Big)\in \mathcal M$
\begin{itemize}
\item {\bf Case-I.1}: $A_1=(1,4)$ $\Rightarrow$ $m_2+1=n_2$
\begin{itemize}
\item {\bf Case-I.1.1}: $(m_2,n_2)=(1,2)$ and $F_2=\{x+y=5\}$, $A_2=(5,0)$\\ {\small\bf($\Delta_{1}$,\ Figure ~\ref{SLPSL-I.1.1})}.
\item {\bf Case-I.1.1.a}: $A_2=(3,2)$ and $F_3=\{x+\frac{1}{2}y=4\}$ {\small\bf($\Delta_{2}$,\ Figure ~\ref{SLPSL-I.1.1.(a)})}.
\item {\bf Case-I.1.1.b}: $A_2=(3,2)$ and $F_3=\{x=3\}$ {\small\bf($\Delta_{3}$,\ Figure ~\ref{SLPSL-I.1.1.(b)})}.
\item {\bf Case-I.1.2}: $(m_2,n_2)=(2,3)$ and $F_2=\{4x+3y=16\}$, $A_2=(4,0)$\\ {\small\bf($\Delta_{4}$,\ Figure ~\ref{SLPSL-I.1.2})}.
\end{itemize}
\item {\bf Case-I.2}: $A_1=(2,4)$ $\Rightarrow$ $n_2=1$
\begin{itemize}
\item {\bf Case-I.2.1}: $(m_2,n_2)=(1,1)$ and $F_2=\Big\{x+\frac{1}{2}y=4\Big\}$, $A_2=(4,0)$\\ {\small\bf($\Delta_{5}$,\ Figure ~\ref{SLPSL-I.2.1})}.
\item {\bf Case-I.2.1.a}: $A_2=(3,2)$ and $F_3=\{x=3\}$ {\small\bf($\Delta_{6}$,\ Figure ~\ref{SLPSL-I.2.1.(a)})}.
\item {\bf Case-I.2.2}: $(m_2,n_2)=(2,1)$ and $F_2=\{4x+y=12\}$, $A_2=(3,0)$\\ {\small\bf($\Delta_{7}$,\ Figure ~\ref{SLPSL-I.2.2})}.
\end{itemize}
\item {\bf Case-I.3}: $A_1=(3,4)$ $\Rightarrow$ $m_2+n_2=1$ $\Rightarrow$ $(m_2,n_2)=(1,0)$ and $F_2=\{x=3\}$\\ {\small\bf($\Delta_{8}$,\ Figure ~\ref{SLPSL-I.3})}.
\end{itemize}
\item {\bf Case-II:} $F_1$ is not orthogonal to $W_2$.\\
$F_1\cap W_2=:A_1=\Big(0,2\frac{2m_1+n_1+1}{n_1}\Big)=\Big(0,2(1+\frac{2m_1+1}{n_1})\Big)\in \mathcal M$ $\Rightarrow$ $\frac{2m_1+1}{n_1}\in\mathbb N_+$\\
Note that $\emptyset\neq F_i\cap W_1=\Big(0,2+\frac{n_i+1}{m_i}\Big)\in \mathcal M$ $\Rightarrow$ $\frac{n_i+1}{m_i}\geq1$$\Rightarrow$ $n_i+1\geq m_i$
\begin{itemize}
\item {\bf Case-II.1}: $\frac{2m_1+1}{n_1}=5$ $\Rightarrow$ $(m_1,n_1)=(2,1)$ and $F_1=\{4x+y=12\}$, $A_2=(3,0)$\\ {\small\bf($\Delta_{9}$,\ Figure ~\ref{SLPSL-II.1})}.
\item {\bf Case-II.2}: $\frac{2m_1+1}{n_1}=3$ $\Rightarrow$ $(m_1,n_1)=(1,1)$ or $(m_1,n_1)=(4,3)$
\begin{itemize}
\item {\bf Case-II.2.1}: $(m_1,n_1)=(1,1)$ and $F_1=\{2x+y=8\}$, $A_2=(4,0)$\\ {\small\bf($\Delta_{10}$,\ Figure ~\ref{SLPSL-II.2.1})}.
\item {\bf Case-II.2.1.a}: $A_2=(3,2)$ and $F_2=\{x=3\}$ {\small\bf($\Delta_{11}$,\ Figure ~\ref{SLPSL-II.2.1.(a)})}.
\item {\bf Case-II.2.1.b}: $A_2=(2,4)$ and $F_2=\{4x+y=12\}$ {\small\bf($\Delta_{12}$,\ Figure ~\ref{SLPSL-II.2.1.(b)})}.
\item {\bf Case-II.2.1.c}: $A_2=(1,6)$ and $F_2=\{3x+y=9\}$ {\small\bf($\Delta_{3}$,\ Figure ~\ref{SLPSL-II.2.1.(c)})}.
\item {\bf Case-II.2.2}: $(m_1,n_1)=(4,3)$ and $F_1=\{8x+3y=24\}$, $A_2=(3,0)$\\ {\small\bf($\Delta_{14}$,\ Figure ~\ref{SLPSL-II.2.2})}.
\end{itemize}
\end{itemize}
\end{itemize}

In the Case-I.2.1.a, 
the toric polytope $\Delta_6^{toric}$ formed by the Weyl group action from the moment polytope $\Delta_6$ is a Delzant polytope and no vertices of $\Delta_6^{toric}$ lie in Weyl walls, 
hence $\overline{\SL_2\times \PSL_2}(6)$ is smooth by Proposition~\ref{AK}.
Since the barycenter of the moment polytope $\Delta_{6}$ with respect to the Duistermaat--Heckman measure is given by
$$
\textbf{bar}_{DH}(\Delta_{6}) = \left(\frac{17274}{8869}, \frac{23943}{8869}\right) \approx (1.948, 2.700),
$$
it is not in the relative interior of the translated cone $2 \rho + \mathcal C^+$.
Thus $\overline{\SL_2\times \PSL_2}(6)$ does not admit any singular K\"{a}hler--Einstein metric by Theorem~\ref{KE criterion for Q-Fano group compactifications}.

Let $Q$ be the point at which the half-line starting from the barycenter $C=\textbf{bar}_{DH}(\Delta_{6})$ in the direction of $A=(2, 2)$ intersects the boundary of $(\Delta_{6} - \mathcal C^+)$. 
Considering the line $x=3$ giving a part of $\partial(\Delta_{6} - \mathcal C^+)$ and the half-line $\overrightarrow{CA}$, 
by Corollary~\ref{formula for greatest Ricci lower bounds}, the greatest Ricci lower bound is
$$
R(\overline{\SL_2\times \PSL_2}(6)) = \displaystyle \frac{\overline{AQ}}{\overline{CQ}} = \frac{1}{3-\frac{17274}{8869}} = \frac{8869}{9333} \approx 0.9503 
.
$$

In the Case-I.3,
the variety $\overline{\SL_2\times \PSL_2}(8)$ is isomorphic to $\mathbb Q^3 \times \mathbb P^3$ from Example~\ref{compactifications of SL2 and PSL2}. 
As $\overline{\SL_2\times \PSL_2}(8)$ is a homogeneous variety, it naturally admits a K\"{a}hler--Einstein metric. 
\end{proof}

\clearpage

\begin{figure}[b]

\begin{minipage}[b]{.45 \textwidth}
 \centering

 
\subcaption{$\overline{\SL_2\times \PSL_2}(1)$ (Case-I.1.1)}
\label{SLPSL-I.1.1}
\medskip
\end{minipage}
\begin{minipage}[b]{.45 \textwidth}
 \centering

%
 
\subcaption{$\overline{\SL_2\times \PSL_2}(2)$ (Case-I.1.1.a)}
\label{SLPSL-I.1.1.(a)}
\medskip
\end{minipage}

\begin{minipage}[b]{.45 \textwidth}
 \centering

%
 
\subcaption{$\overline{\SL_2\times \PSL_2}(3)$ (Case-I.1.1.b)}
\label{SLPSL-I.1.1.(b)}
\medskip
\end{minipage}
\begin{minipage}[b]{.45 \textwidth}
 \centering

%
 
\subcaption{$\overline{\SL_2\times \PSL_2}(4)$ (Case-I.1.2)}
\label{SLPSL-I.1.2}
\medskip
\end{minipage}

\begin{minipage}[b]{.45 \textwidth}
 \centering

%
 
\subcaption{$\overline{\SL_2\times \PSL_2}(5)$ (Case-I.2.1)}
\label{SLPSL-I.2.1}
\medskip
\end{minipage}
\begin{minipage}[b]{.45 \textwidth}
 \centering

%
 
\subcaption{$\overline{\SL_2\times \PSL_2}(6)$ (Case-I.2.1.a)}
\label{SLPSL-I.2.1.(a)}
\medskip
\end{minipage}

\begin{minipage}[b]{.45 \textwidth}
 \centering

%
 
\subcaption{$\overline{\SL_2\times \PSL_2}(8)$ (Case-I.3)}
\label{SLPSL-I.3}
\medskip
\end{minipage}
\caption{Moment polytopes of Gorenstein Fano compactifications of $\SL_2(\mathbb{C}) \times \PSL_2(\mathbb{C})$-(i).}
\label{SLPSL-polytope}

\end{figure}

\clearpage
\begin{figure}[ht]
\ContinuedFloat
\begin{minipage}[b]{.3 \textwidth}
 \centering

%
 
\subcaption{$\overline{\SL_2\times \PSL_2}(9)$ (Case-II.1)}
\label{SLPSL-II.1}
\medskip
\end{minipage}
\begin{minipage}[b]{.3 \textwidth}
 \centering

%
 
\subcaption{$\overline{\SL_2\times \PSL_2}(10)$ (Case-II.2.1)}
\label{SLPSL-II.2.1}
\medskip
\end{minipage}
\begin{minipage}[b]{.3 \textwidth}
 \centering

%
 
\subcaption{\tiny$\overline{\SL_2\times \PSL_2}(11)$ (Case-II.2.1.a)}
\label{SLPSL-II.2.1.(a)}
\medskip
\end{minipage}

\begin{minipage}[b]{.3 \textwidth}
 \centering

%
 
\subcaption{\tiny$\overline{\SL_2\times \PSL_2}(12)$ (Case-II.2.1.b)}
\label{SLPSL-II.2.1.(b)}

\end{minipage}
\begin{minipage}[b]{.3 \textwidth}
 \centering

%
 
\subcaption{\tiny$\overline{\SL_2\times \PSL_2}(13)$ (Case-II.2.1.c)}
\label{SLPSL-II.2.1.(c)}

\end{minipage}
\begin{minipage}[b]{.3 \textwidth}
 \centering

%
 
\subcaption{$\overline{\SL_2\times \PSL_2}(14)$ (Case-II.2.2)}
\label{SLPSL-II.2.2}

\end{minipage}
\caption{Moment polytopes of Gorenstein Fano compactifications of $\SL_2(\mathbb{C}) \times \PSL_2(\mathbb{C})$-(ii).}
\label{SLPSL-polytope}
\end{figure}

\clearpage 

\subsubsection{Gorenstein Fano equivariant compactifications of $\SO_4(\mathbb C)$ }
The spherical weight lattice $\mathcal M$ of $\SO_4(\mathbb C)$ is spanned by $\varpi_1 + \varpi_2 = (1, 1)$ and $- \varpi_1 + \varpi_2 = (-1, 1)$, and its dual lattice $\mathcal N$ is spanned by $\varpi_1^{\vee}+\varpi_2^{\vee}=\Big(\frac{1}{2}, \frac{1}{2}\Big)$ and $- \varpi_1^{\vee} + \varpi_2^{\vee}=\Big(- \frac{1}{2}, \frac{1}{2}\Big)$.
The Weyl walls are given by $W_1=\{y=0\}$ and $W_2=\{x=0\}$.
The sum of positive roots is $2\rho=(2,2)$.

In \cite[Section 7.1]{LTZ20}, Li--Tian--Zhu obtained the following results.
Note that they used different bases for $\mathcal M$ and $\mathcal N$.
In fact, taking coordinate changes via the linear transformation given by a matrix 
$
\begin{pmatrix}
\frac{1}{2} & \frac{1}{2}\\
-\frac{1}{2} & \frac{1}{2}
\end{pmatrix}
$,
the bases $\varpi_1 + \varpi_2 = (1, 1)$ and $- \varpi_1 + \varpi_2 = (-1, 1)$ of $\mathcal M$ can be transformed as $(1, 0)$ and $(0, 1)$, respectively, as in \cite[Section 7.1]{LTZ20}.

\begin{theorem} \cite[Section 7.1]{LTZ20}
There are six Gorenstein Fano equivariant compactifications of $\SO_4(\mathbb C)$ up to isomorphism: three smooth compactifications and three singular compactifications. 
Their moment polytopes are given in the following Table~\ref{table:SO4} and Figure~\ref{SO4-polytope}. 
Among them, the only one compactification in each cases admits a K\"{a}hler--Einstein metric.
{\rm
\begin{table}[h]
{\renewcommand{\arraystretch}{1.2}
\small{
\begin{tabular}{|c|l|l|l|c|c|c|}\hline
$X$ & 	Case & Edges (except Weyl walls)  & Vertices	&	Smooth? & 	KE & $\textbf{bar}(\Delta)$\\
\hline
$\overline{\SO_4}(1)$	&	I.1.1 &	$y=3,\ x+3y=10$  & $(0,3),(1,3),(10,0)$	& singular	&	No & 	$\Big(\frac{8229}{1918}, \frac{1368}{959}\Big)$
\\ \hline
$\overline{\SO_4}(2)$	&	I.1.2 &	$y=3,\ x+3y=10,\ x+y=6$  & $(0,3),(1,3),(4,2),(6,0)$ 	& smooth	&	No & 	$\Big(\frac{467331}{150514}, \frac{125928}{75257}\Big)$
\\ \hline
$\overline{\SO_4}(3)$	&	I.1.3 &	$y=3,\ 3x+5y=18$  & $(0,3),(1,3),(6,0)$ 	& singular	&	No & 	$\Big(\frac{1341}{518}, \frac{396}{259}\Big)$
\\ \hline
$\overline{\SO_4}(4)$	&	I.2.1 &	$y=3,\ x+y=6$  & $(0,3),(3,3),(6,0)$		& smooth	&	No & 	$\Big(\frac{297}{98}, \frac{96}{49}\Big)$
\\ \hline 
$\overline{\SO_4}(5)$&	I.2.2 &	$y=3,\ x=3$  & $(0,3),(3,3),(3,0)$ 	& singular	&	Yes &  	$\Big(\frac{9}{4}, \frac{9}{4}\Big)$
\\ \hline
$\overline{\SO_4}(6)$&	II &	$x+y=6$  & $(0,6),(6,0)$ 	& smooth	&	Yes &   	$\Big(\frac{18}{7}, \frac{18}{7}\Big)$
\\ \hline
\end{tabular}
}}
\caption{Gorenstein Fano equivariant compactifications of $\SO_4(\mathbb C)$.}
\label{table:SO4}
\end{table}}
\end{theorem}

\begin{proof}[Proof of Theorem~\ref{greatest Ricci lower bounds: Gorenstein Fano equivariant compactifications} for K-unstable smooth compactifications of $\SO_4(\mathbb C)$]

(1) In the Case I.1.2 (Case 1.1.2 in \cite[Section 7.1]{LTZ20}), $\overline{\SO_4}(2)$ is a smooth compactification of $\SO_4(\mathbb C)$ corresponding to Figure~\ref{SO4-I.1.2} and the barycenter of its moment polytope is
$$(\bar{x}, \bar{y}) = \left(\frac{467331}{150514}, \frac{125928}{75257} \right) \approx (3.1049, 0.7538).$$ 
Let $Q$ be the point at which the half-line starting from the barycenter $C$ in the direction of $A=(2, 2)$ intersects the boundary of the moment polytope. 
Considering the line $y=3$ and the half-line $\overrightarrow{CA}$, 
we can compute 
$Q = \left(-\frac{67959}{49172}, 3 \right)$. 
Thus, we obtain the greatest Ricci lower bound $R = \displaystyle \frac{\overline{AQ}}{\overline{CQ}} = \frac{75257}{99843} \approx 0.7538$ 
by Corollary~\ref{formula for greatest Ricci lower bounds}.

(2) In the Case I.2.1 (Case 1.2.1 in \cite[Section 7.1]{LTZ20}), $\overline{\SO_4}(4)$ is a smooth compactification of $\SO_4(\mathbb C)$ corresponding to Figure~\ref{SO4-I.2.1} and the barycenter of its moment polytope is
$$(\bar{x}, \bar{y}) = \left(\frac{297}{98}, \frac{96}{49} \right) \approx (3.0306, 1.9592).$$
Let $Q$ be the point at which the half-line starting from the barycenter $C$ in the direction of $A=(2, 2)$ intersects the boundary of the moment polytope. 
Considering the line $y=3$ and the half-line $\overrightarrow{CA}$, 
we compute 
$Q = \left(-\frac{93}{4}, 3 \right)$. 
Thus, we obtain the greatest Ricci lower bound $R = \displaystyle \frac{\overline{AQ}}{\overline{CQ}} = \frac{49}{51} \approx 0.9608$
by Corollary~\ref{formula for greatest Ricci lower bounds}.
\end{proof}
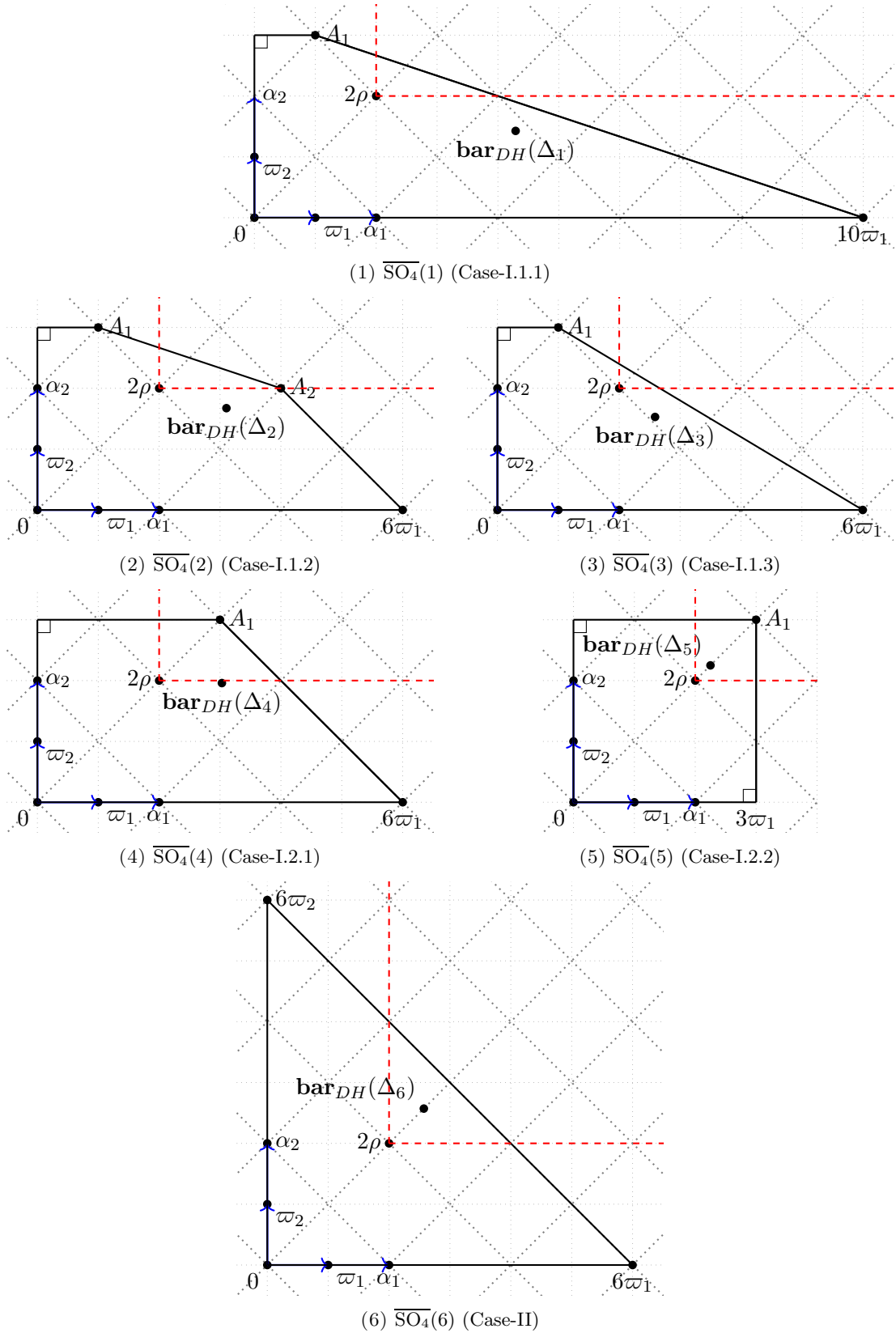
\begin{figure}[h]

\begin{minipage}[b]{.45 \textwidth}
 \centering

\begin{tikzpicture}
\clip (-0.5,-0.5) rectangle (10.5,3.5); 

\coordinate (pi1) at (1,0);
\coordinate (pi2) at (0,1);
\coordinate (v0) at (0,3);
\coordinate (v1) at (1,3);
\coordinate (v2) at (10,0);
\coordinate (a1) at (2,0);
\coordinate (a2) at (0,2);

\coordinate (Origin) at (0,0);
\coordinate (2rho) at (2,2);

\coordinate (barycenter) at (4.29040667361835, 1.42648592283629);

\foreach \x  in {-8,-7,...,11}{
  \draw[help lines,dotted]
    (\x,-8) -- (\x,11);
}

\foreach \x  in {-8,-7,...,11}{
  \draw[help lines,dotted]
    (-8,\x) -- (11,\x);
}

\begin{scope}[y=(45:1)]
\foreach \y  in {-12,-10,...,16}{
  \draw[help lines,thick,dotted]
   [rotate=45]  (-12,\y) -- (16,\y) ;
}
\foreach \y  in {-12,-10,...,12}{
  \draw[help lines,thick,dotted]
   [rotate=-45]  (-12,\y) -- (16,\y) ;
}
\end{scope}

\fill (Origin) circle (2pt) node[below left] {0};
\fill (pi1) circle (2pt) node[below right] {$\varpi_1$};
\fill (pi2) circle (2pt) node[below right] {$\varpi_2$};
\fill (a1) circle (2pt) node[below] {$\alpha_1$};
\fill (a2)  node[right] {$\alpha_2$};

\fill (2rho) circle (2pt) node[left] {$2\rho$};

\fill (v1) circle (2pt) node[right] {$A_1$};
\fill (v2) circle (2pt) node[below] {$10\varpi_1$};

\fill (barycenter) circle (2pt) node[below] {$\textbf{bar}_{DH}(\Delta_{1})$};

\draw[->,blue,thick](Origin)--(pi1);
\draw[->,blue,thick](Origin)--(pi2); 
\draw[->,blue,thick](Origin)--(a1);
\draw[->,blue,thick](Origin)--(a2);

\draw[thick](Origin)--(v0);
\draw[thick](v0)--(v1);
\draw[thick](v1)--(v2);
\draw[thick](Origin)--(v2);
\draw[thick](v1)--(v2);

\aeMarkRightAngle[size=6pt](Origin,v0,v1)

\draw [shorten >=-10cm, red, thick, dashed] (2rho) to ($(2rho)+(a1)$);
\draw [shorten >=-4cm, red, thick, dashed] (2rho) to ($(2rho)+(a2)$);
\end{tikzpicture} 
\subcaption{$\overline{\SO_4}(1)$ (Case-I.1.1)}
\label{SO4-I.1.1}
\medskip
\end{minipage}

\begin{minipage}[b]{.45 \textwidth}
 \centering

\begin{tikzpicture}
\clip (-0.5,-0.5) rectangle (6.5,3.5); 

\coordinate (pi1) at (1,0);
\coordinate (pi2) at (0,1);
\coordinate (v0) at ($3*(pi2)$);
\coordinate (v1) at ($1*(pi1)+3*(pi2)$);
\coordinate (v2) at ($4*(pi1)+2*(pi2)$);
\coordinate (v3) at ($6*(pi1)$);
\coordinate (a1) at (2,0);
\coordinate (a2) at (0,2);

\coordinate (Origin) at (0,0);
\coordinate (2rho) at (2,2);

\coordinate (barycenter) at (3.10490054081348, 1.67330613763504);

\foreach \x  in {-8,-7,...,9}{
  \draw[help lines,dotted]
    (\x,-8) -- (\x,9);
}

\foreach \x  in {-8,-7,...,9}{
  \draw[help lines,dotted]
    (-8,\x) -- (9,\x);
}

\begin{scope}[y=(45:1)]
\foreach \y  in {-12,-10,...,16}{
  \draw[help lines,thick,dotted]
   [rotate=45]  (-12,\y) -- (16,\y) ;
}
\foreach \y  in {-12,-10,...,12}{
  \draw[help lines,thick,dotted]
   [rotate=-45]  (-12,\y) -- (16,\y) ;
}
\end{scope}

\fill (Origin) circle (2pt) node[below left] {0};
\fill (pi1) circle (2pt) node[below right] {$\varpi_1$};
\fill (pi2) circle (2pt) node[below right] {$\varpi_2$};
\fill (a1) circle (2pt) node[below] {$\alpha_1$};
\fill (a2) circle (2pt) node[right] {$\alpha_2$};

\fill (2rho) circle (2pt) node[left] {$2\rho$};

\fill (v1) circle (2pt) node[right] {$A_1$};
\fill (v2) circle (2pt) node[right] {$A_2$};
\fill (v3) circle (2pt) node[below] {$6\varpi_1$};

\fill (barycenter) circle (2pt) node[below] {$\textbf{bar}_{DH}(\Delta_{2})$};

\draw[->,blue,thick](Origin)--(pi1);
\draw[->,blue,thick](Origin)--(pi2); 
\draw[->,blue,thick](Origin)--(a1);
\draw[->,blue,thick](Origin)--(a2);

\draw[thick](Origin)--(v0);
\draw[thick](v0)--(v1);
\draw[thick](v1)--(v2);
\draw[thick](v2)--(v3);
\draw[thick](Origin)--(v3);

\aeMarkRightAngle[size=6pt](Origin,v0,v1)

\draw [shorten >=-4cm, red, thick, dashed] (2rho) to ($(2rho)+(a1)$);
\draw [shorten >=-4cm, red, thick, dashed] (2rho) to ($(2rho)+(a2)$);
\end{tikzpicture} 
\subcaption{$\overline{\SO_4}(2)$ (Case-I.1.2)}
\label{SO4-I.1.2}
\medskip
\end{minipage}
\begin{minipage}[b]{.45 \textwidth}
 \centering

\begin{tikzpicture}
\clip (-0.5,-0.5) rectangle (6.5,3.5); 

\coordinate (pi1) at (1,0);
\coordinate (pi2) at (0,1);
\coordinate (v0) at ($3*(pi2)$);
\coordinate (v1) at ($1*(pi1)+3*(pi2)$);

\coordinate (v3) at ($6*(pi1)$);
\coordinate (a1) at (2,0);
\coordinate (a2) at (0,2);

\coordinate (Origin) at (0,0);
\coordinate (2rho) at (2,2);

\coordinate (barycenter) at (2.58880308880309, 1.52895752895753);

\foreach \x  in {-8,-7,...,9}{
  \draw[help lines,dotted]
    (\x,-8) -- (\x,9);
}

\foreach \x  in {-8,-7,...,9}{
  \draw[help lines,dotted]
    (-8,\x) -- (9,\x);
}
\begin{scope}[y=(45:1)]
\foreach \y  in {-12,-10,...,16}{
  \draw[help lines,thick,dotted]
   [rotate=45]  (-12,\y) -- (16,\y) ;
}
\foreach \y  in {-12,-10,...,12}{
  \draw[help lines,thick,dotted]
   [rotate=-45]  (-12,\y) -- (16,\y) ;
}
\end{scope}

\fill (Origin) circle (2pt) node[below left] {0};
\fill (pi1) circle (2pt) node[below right] {$\varpi_1$};
\fill (pi2) circle (2pt) node[below right] {$\varpi_2$};
\fill (a1) circle (2pt) node[below] {$\alpha_1$};
\fill (a2) circle (2pt) node[right] {$\alpha_2$};

\fill (2rho) circle (2pt) node[left] {$2\rho$};

\fill (v1) circle (2pt) node[right] {$A_1$};
\fill (v3) circle (2pt) node[below] {$6\varpi_1$};

\fill (barycenter) circle (2pt) node[below] {$\textbf{bar}_{DH}(\Delta_{3})$};

\draw[->,blue,thick](Origin)--(pi1);
\draw[->,blue,thick](Origin)--(pi2); 
\draw[->,blue,thick](Origin)--(a1);
\draw[->,blue,thick](Origin)--(a2);

\draw[thick](Origin)--(v0);
\draw[thick](v0)--(v1);
\draw[thick](v1)--(v3);
\draw[thick](Origin)--(v3);

\aeMarkRightAngle[size=6pt](Origin,v0,v1)

\draw [shorten >=-4cm, red, thick, dashed] (2rho) to ($(2rho)+(a1)$);
\draw [shorten >=-4cm, red, thick, dashed] (2rho) to ($(2rho)+(a2)$);
\end{tikzpicture} 
\subcaption{$\overline{\SO_4}(3)$ (Case-I.1.3)}
\label{SO4-I.1.3}
\medskip
\end{minipage}

\begin{minipage}[b]{.45 \textwidth}
 \centering

\begin{tikzpicture}
\clip (-0.5,-0.5) rectangle (6.5,3.5); 

\coordinate (pi1) at (1,0);
\coordinate (pi2) at (0,1);
\coordinate (v0) at ($3*(pi2)$);
\coordinate (v1) at ($3*(pi1)+3*(pi2)$);
\coordinate (v2) at ($6*(pi1)$);
\coordinate (a1) at (2,0);
\coordinate (a2) at (0,2);

\coordinate (Origin) at (0,0);
\coordinate (2rho) at (2,2);

\coordinate (barycenter) at (3.03061224489796, 1.95918367346939);

\foreach \x  in {-8,-7,...,9}{
  \draw[help lines,dotted]
    (\x,-8) -- (\x,9);
}

\foreach \x  in {-8,-7,...,9}{
  \draw[help lines,dotted]
    (-8,\x) -- (9,\x);
}
\begin{scope}[y=(45:1)]
\foreach \y  in {-12,-10,...,16}{
  \draw[help lines,thick,dotted]
   [rotate=45]  (-12,\y) -- (16,\y) ;
}
\foreach \y  in {-12,-10,...,12}{
  \draw[help lines,thick,dotted]
   [rotate=-45]  (-12,\y) -- (16,\y) ;
}
\end{scope}

\fill (Origin) circle (2pt) node[below left] {0};
\fill (pi1) circle (2pt) node[below right] {$\varpi_1$};
\fill (pi2) circle (2pt) node[below right] {$\varpi_2$};
\fill (a1) circle (2pt) node[below] {$\alpha_1$};
\fill (a2) circle (2pt) node[right] {$\alpha_2$};

\fill (2rho) circle (2pt) node[left] {$2\rho$};

\fill (v1) circle (2pt) node[right] {$A_1$};
\fill (v2) circle (2pt) node[below] {$6\varpi_1$};

\fill (barycenter) circle (2pt) node[below] {$\textbf{bar}_{DH}(\Delta_{4})$};

\draw[->,blue,thick](Origin)--(pi1);
\draw[->,blue,thick](Origin)--(pi2); 
\draw[->,blue,thick](Origin)--(a1);
\draw[->,blue,thick](Origin)--(a2);

\draw[thick](Origin)--(v0);
\draw[thick](v0)--(v1);
\draw[thick](v1)--(v2);
\draw[thick](Origin)--(v2);
\draw[thick](v1)--(v2);

\aeMarkRightAngle[size=6pt](Origin,v0,v1)

\draw [shorten >=-4cm, red, thick, dashed] (2rho) to ($(2rho)+(a1)$);
\draw [shorten >=-4cm, red, thick, dashed] (2rho) to ($(2rho)+(a2)$);
\end{tikzpicture} 
\subcaption{$\overline{\SO_4}(4)$ (Case-I.2.1)}
\label{SO4-I.2.1}
\medskip
\end{minipage}
\begin{minipage}[b]{.45 \textwidth}
 \centering

\begin{tikzpicture}
\clip (-0.5,-0.5) rectangle (4,3.5); 

\coordinate (pi1) at (1,0);
\coordinate (pi2) at (0,1);
\coordinate (v0) at ($3*(pi2)$);
\coordinate (v1) at ($3*(pi1)+3*(pi2)$);
\coordinate (v2) at ($3*(pi1)$);
\coordinate (a1) at (2,0);
\coordinate (a2) at (0,2);

\coordinate (Origin) at (0,0);
\coordinate (2rho) at (2,2);

\coordinate (barycenter) at (2.25, 2.25);

\foreach \x  in {-8,-7,...,9}{
  \draw[help lines,dotted]
    (\x,-8) -- (\x,9);
}

\foreach \x  in {-8,-7,...,9}{
  \draw[help lines,dotted]
    (-8,\x) -- (9,\x);
}
\begin{scope}[y=(45:1)]
\foreach \y  in {-12,-10,...,16}{
  \draw[help lines,thick,dotted]
   [rotate=45]  (-12,\y) -- (16,\y) ;
}
\foreach \y  in {-12,-10,...,12}{
  \draw[help lines,thick,dotted]
   [rotate=-45]  (-12,\y) -- (16,\y) ;
}
\end{scope}

\fill (Origin) circle (2pt) node[below left] {0};
\fill (pi1) circle (2pt) node[below right] {$\varpi_1$};
\fill (pi2) circle (2pt) node[below right] {$\varpi_2$};
\fill (a1) circle (2pt) node[below] {$\alpha_1$};
\fill (a2) circle (2pt) node[right] {$\alpha_2$};

\fill (2rho) circle (2pt) node[left] {$2\rho$};

\fill (v1) circle (2pt) node[right] {$A_1$};
\fill (v2) node[below] {$3\varpi_1$};

\fill (barycenter) circle (2pt) node[above left] {$\textbf{bar}_{DH}(\Delta_{5})$};

\draw[->,blue,thick](Origin)--(pi1);
\draw[->,blue,thick](Origin)--(pi2); 
\draw[->,blue,thick](Origin)--(a1);
\draw[->,blue,thick](Origin)--(a2);

\draw[thick](Origin)--(v0);
\draw[thick](v0)--(v1);
\draw[thick](v1)--(v2);
\draw[thick](Origin)--(v2);
\draw[thick](v1)--(v2);

\aeMarkRightAngle[size=6pt](Origin,v0,v1)
\aeMarkRightAngle[size=6pt](v1,v2,Origin)

\draw [shorten >=-4cm, red, thick, dashed] (2rho) to ($(2rho)+(a1)$);
\draw [shorten >=-4cm, red, thick, dashed] (2rho) to ($(2rho)+(a2)$);
\end{tikzpicture} 
\subcaption{$\overline{\SO_4}(5)$ (Case-I.2.2)}
\label{SO4-I.2.2}
\medskip
\end{minipage}

\begin{minipage}[b]{.45 \textwidth}
 \centering

\begin{tikzpicture}
\clip (-0.5,-0.5) rectangle (6.5,6.3); 

\coordinate (pi1) at (1,0);
\coordinate (pi2) at (0,1);
\coordinate (v0) at ($6*(pi2)$);
\coordinate (v2) at ($6*(pi1)$);
\coordinate (a1) at (2,0);
\coordinate (a2) at (0,2);

\coordinate (Origin) at (0,0);
\coordinate (2rho) at (2,2);

\coordinate (barycenter) at (2.57142857142857, 2.57142857142857);

\foreach \x  in {-8,-7,...,9}{
  \draw[help lines,dotted]
    (\x,-8) -- (\x,9);
}

\foreach \x  in {-8,-7,...,9}{
  \draw[help lines,dotted]
    (-8,\x) -- (9,\x);
}
\begin{scope}[y=(45:1)]
\foreach \y  in {-12,-10,...,16}{
  \draw[help lines,thick,dotted]
   [rotate=45]  (-12,\y) -- (16,\y) ;
}
\foreach \y  in {-12,-10,...,12}{
  \draw[help lines,thick,dotted]
   [rotate=-45]  (-12,\y) -- (16,\y) ;
}
\end{scope}

\fill (Origin) circle (2pt) node[below left] {0};
\fill (pi1) circle (2pt) node[below right] {$\varpi_1$};
\fill (pi2) circle (2pt) node[below right] {$\varpi_2$};
\fill (a1) circle (2pt) node[below] {$\alpha_1$};
\fill (a2) circle (2pt) node[right] {$\alpha_2$};

\fill (2rho) circle (2pt) node[left] {$2\rho$};

\fill (v0) circle (2pt) node[right] {$6\varpi_2$};
\fill (v2) circle (2pt) node[below] {$6\varpi_1$};

\fill (barycenter) circle (2pt) node[above left] {$\textbf{bar}_{DH}(\Delta_{6})$};

\draw[->,blue,thick](Origin)--(pi1);
\draw[->,blue,thick](Origin)--(pi2); 
\draw[->,blue,thick](Origin)--(a1);
\draw[->,blue,thick](Origin)--(a2);

\draw[thick](Origin)--(v0);
\draw[thick](v0)--(v2);
\draw[thick](Origin)--(v2);


\draw [shorten >=-4cm, red, thick, dashed] (2rho) to ($(2rho)+(a1)$);
\draw [shorten >=-4cm, red, thick, dashed] (2rho) to ($(2rho)+(a2)$);
\end{tikzpicture} 
\subcaption{$\overline{\SO_4}(6)$ (Case-II)}
\label{SO4-II}
\medskip
\end{minipage}

\caption{Moment polytopes of Gorenstein Fano compactifications of $\SO_4(\mathbb{C})$.}
\label{SO4-polytope}
\end{figure}


\clearpage


\providecommand{\bysame}{\leavevmode\hbox to3em{\hrulefill}\thinspace}
\providecommand{\MR}{\relax\ifhmode\unskip\space\fi MR }
\providecommand{\MRhref}[2]{%
  \href{http://www.ams.org/mathscinet-getitem?mr=#1}{#2}
}
\providecommand{\href}[2]{#2}

\end{document}